\documentclass[a4paper,11pt]{article}
\usepackage[utf8]{inputenc} 
\usepackage[english]{babel}
\usepackage[T1]{fontenc}   
\usepackage{graphicx}
\usepackage{comment}
\usepackage{amsthm,amsmath,amsfonts,stmaryrd,mathrsfs,amssymb}
\usepackage{mathtools}
\usepackage{bbm}
\usepackage{enumitem}
\usepackage{chemarrow}
\usepackage{tikz}
\usepackage[top=2.5cm, bottom=2.5cm, left=2cm, right=2cm]{geometry}
\usepackage{subfig}
\usepackage{multirow}
\usepackage{physics}
\usepackage{array}
\usepackage{bigints}
\usepackage{etoolbox}
\usepackage{mdframed}
\usepackage{color}
\usepackage[unicode,colorlinks=true,linkcolor=blue,citecolor=red,pdfencoding=auto,psdextra]{hyperref}
\usepackage{tabulary}
\usepackage{booktabs}
\setenumerate[1]{label=(\roman*), font = \normalfont}

\newmdtheoremenv{theorem}{Theorem}
\newmdtheoremenv{proposition}[theorem]{Proposition}
\newtheorem{lemma}[theorem]{Lemma}

\newtheorem{definition}[theorem]{Definition}

\newtheorem{remark}[theorem]{Remark}
\newtheorem{fact}[theorem]{Fact}

\makeatletter
\newcommand{\mylabel}[2]{#2\def\@currentlabel{#2}\label{#1}}
\makeatother

\newcommand{\ensemblenombre}[1]{\mathbb{#1}}
\newcommand{\N}{\ensemblenombre{N}}

\newcommand{\R}{\ensemblenombre{R}}

\renewcommand{\P}{\mathbb{P}}

\newcommand{\K}{\mathbb{K}}
\newcommand{\bU}{\mathbb{U}}

\newcommand{\Ec}[1]{\mathbb{E} \left[#1\right]}

\newcommand{\Pp}[1]{\mathbb{P} \left(#1\right)}

\newcommand{\Ecsq}[2]{\mathbb{E} \left[#1\mathrel{}\middle|\mathrel{}#2\right]}

\newcommand{\Ppsq}[2]{\mathbb{P} \left(#1\mathrel{}\middle|\mathrel{}#2\right)}

\newcommand{\Var}[1]{\mathrm{Var}\left(#1\right)}
\newcommand{\Varsq}[2]{\mathrm{Var} \left(#1\mathrel{}\middle|\mathrel{}#2\right)}

\newcommand{\intervalle}[4]{\mathopen{#1}#2
                            \mathclose{}\mathpunct{},#3
                            \mathclose{#4}}
\newcommand{\intervalleff}[2]{\intervalle{[}{#1}{#2}{]}}
\newcommand{\intervalleof}[2]{\intervalle{(}{#1}{#2}{]}}
\newcommand{\intervallefo}[2]{\intervalle{[}{#1}{#2}{)}}
\newcommand{\intervalleoo}[2]{\intervalle{(}{#1}{#2}{)}}

\newcommand{\petito}[1]{o\mathopen{}\left(#1\right)}

\newcommand{\enstq}[2]{\left\lbrace#1\mathrel{}\middle|\mathrel{}#2\right\rbrace}

\newcommand{\ind}[1]{\mathbf{1}_{\lbrace #1 \rbrace}}

\newcommand{\cT}{\mathcal{T}}

\newcommand{\cL}{\mathcal{L}}

\newcommand{\cX}{\mathcal{X}}
\newcommand{\cY}{\mathcal{Y}}

\newcommand{\rM}{\mathrm{M}}
\newcommand{\rT}{\mathrm{T}}

\newcommand{\sP}{\mathscr{P}}
\newcommand{\sR}{\mathscr{R}}
\newcommand{\sA}{\mathscr{A}}
\newcommand{\sQ}{\mathscr{Q}}
\newcommand{\sS}{\mathscr{S}}

\newcommand{\sfX}{\mathsf{X}}
\newcommand{\sfY}{\mathsf{Y}}
\newcommand{\sfA}{\mathsf{A}}
\newcommand{\sfB}{\mathsf{B}}
\newcommand{\sfD}{\mathsf{D}}

\newcommand{\rmH}{\mathrm{H}}
\newcommand{\rmP}{\mathrm{P}}

\newcommand{\bfA}{\mathbf{A}}

\newcommand{\bL}{\mathbb{L}}

\newcommand{\cF}{\mathcal{F}}

\newcommand{\cW}{\mathcal{W}}
\newcommand{\cM}{\mathcal{M}}

\newcommand{\ttT}{\mathtt{T}}

\newcommand{\sD}{\mathscr{D}}

\newcommand{\cb}{{\mathcal B}}

\DeclareMathOperator{\diam}{diam}

\DeclareMathOperator{\cst}{cst}
\DeclareMathOperator{\PGW}{PGW}
\DeclareMathOperator{\Poi}{Poi}

\begin{document}
\title{The scaling limit of the root component in the Wired Minimal Spanning Forest of the Poisson Weighted Infinite Tree}
\author{Omer \textsc{Angel}\thanks{Department of Mathematics, University of British Columbia. \hfill \href{mailto:angel@math.ubc.ca}{\texttt{angel@math.ubc.ca}}}  
	 \quad and \quad Delphin \textsc{S\'enizergues}\thanks{MODAL'X, UPL, Univ. Paris Nanterre, CNRS, F92000 Nanterre, France.\hfill  \href{mailto:dsenizer@parisnanterre.fr}{\texttt{dsenizer@parisnanterre.fr}}}}
\maketitle
\begin{abstract}
	
In this paper we prove a scaling limit result for the component of the root in the Wired Minimal Spanning Forest (WMSF) of the Poisson-Weighted Infinite Tree (PWIT), where the latter tree arises as the local weak limit of the Minimal Spanning Tree (MST) on the complete graph endowed with i.i.d.\ weights on its edges. 
The limiting object $\mathcal{M}^\infty$ can be obtained by aggregating independent Brownian trees using two types of gluing procedures: one that we call the \emph{Brownian tree aggregation process} and resembles the so-called \emph{stick-breaking construction of the Brownian tree}; and another one that we call the \emph{chain} construction, which simply corresponds to gluing a sequence of metric spaces along a line.
\end{abstract}

\section{Introduction}

Let $(V, E , w)$ be a finite, connected, weighted graph, where $(V, E)$ is the underlying graph and $w:E \rightarrow \intervallefo{0}{\infty}$ is the weight function. 
A spanning tree of $(V, E)$ is a
tree that is a subgraph of $(V, E)$ with vertex set $V$. 
A minimal spanning tree (MST) $T$ of $(V, E , w)$ satisfies
\begin{align*}
	\sum_{e\in T}w(e)= \min \enstq{\sum_{e\in T'}w(e)}{T' \text{ is a spanning tree of } (V,E)}.
\end{align*}
 This combinatorial optimization problem has been extensively studied in computer science, combinatorics and probability.
 We refer to \cite[Section~1.1]{addario-berry_scaling_2017} for an overview of the probabilistic results related to the MST problem. 

In this paper we place ourselves in the following random setting: we consider $K_n$ the complete graph on $n$ vertices and assign i.i.d.\ weights to the edges distributed as uniform random variables in $\intervalleff{0}{1}$. 
In that case, the edge-weights are almost surely distinct so there almost surely exist a unique MST, which we call $M_n$. 
 We are interested in some asymptotic properties of $M_n$ as $n\rightarrow \infty$ and, in particular, in different ways of defining a notion of limit for this object. 
 
\paragraph{The scaling limit of $M_n$.}
For a measured metric space $(X,d,\mu)$, we denote by $\mathrm{Scale}(a,b;X)$ the space $X$ with distances scaled by $a$ and measure scaled by $b$.
In \cite{addario-berry_scaling_2017}, the authors consider $M_n$ as random measured metric space, by endowing its set of vertices with the graph distance and the counting measure.
They obtain the following scaling limit convergence.
We shall recall the definition of the Gromov--Hausdorff--Prokhorov topology later in Section~\ref{subsec:reminder about GHP}. 
\begin{theorem}[Theorem~1.1 of \cite{addario-berry_scaling_2017}]
Seen as a measured metric space, we have the following convergence in the Gromov--Hausdorff--Prokhorov topology
\begin{align*}
	\mathrm{Scale}(n^{-\frac13},n^{-1}; M_n) \underset{n\rightarrow \infty}{\rightarrow} \cM^{\mathrm{comp}},
\end{align*}
where the limiting object $\cM^{\mathrm{comp}}$ is a compact random tree that has Minkowski dimension equal to $3$ almost surely.
\end{theorem}
The limiting tree appearing in the statement of the theorem is not constructed in an explicit way in \cite{addario-berry_scaling_2017} and remains quite mysterious. 
This object is believed to be universal in the sense that it is expected to arise as the scaling limit of the MST of random graphs that exhibit mean field behavior when their edges are endowed with i.i.d.\ weights: random regular graphs, random graphs with given degree sequence, inhomogeneous graphs, high dimensional hypercubes and more. 
One step in this direction is the work \cite{addario-berry_geometry_2021} of Addario-Berry and Sen where they rigorously prove such a result for the uniform $3$-regular graph.    
In the recent preprint \cite{broutin_convex_2023}, Broutin and Marckert present an explicit construction of $\cM^{\mathrm{comp}}$ as the \emph{convex minorant tree} obtained from a trajectory of Brownian motion with parabolic drift, together with a sequence of independent uniform random variable $(U_i)_{i\geq 1}$ on $\intervalleoo{0}{1}$, and refer to the object $\cM^{\mathrm{comp}}$ as the \emph{parabolic Brownian tree}. 
This new construction provides deeper insight into the object's structure, and in particular, it allows them to prove that the Hausdorff dimension of the object is $3$.
\paragraph{The local weak limit.}
Another way to consider a limit of the object $\mathrm{M}_n$ is to examine its local behavior around a random vertex. 
A general theorem of Aldous and Steele \cite{aldous_objective_2004} ensures that for the appropriate \emph{weighted} local topology (one that accounts for the weights of the edges as well as the structure of the graph), if the underlying weighted graph converges locally in that topology towards a limit, then the corresponding MST converges as well towards the Wired Minimal Spanning Forest (WMSF) of the limiting graph. 
In the case of $M_n$, the underlying graph $K_n$ endowed with appropriate i.i.d.\ weights on the edges is known to converge to the Poisson-Weighted Infinite Tree (PWIT) and so the weighted local limit of $M_n$ is the WMSF of the PWIT.

In \cite{addario-berry_local_2013}, the same convergence is proved in the specific case of the object $M_n$ and the structure of the limit is investigated in more depth.
Finally, in the recent paper \cite{nachmias_wired_2024}, Nachmias and Tang prove that under some mild conditions, which apply in our case, this convergence in the weighted local topology transfers to a convergence in the usual local topology when one forgets the weights.
Note that in this case the limiting object is only a tree instead of being a forest: any vertex that is not in the root component of the WMSF is at distance infinity from the root when using the usual graph distance.  
\begin{theorem}[\cite{aldous_objective_2004,addario-berry_local_2013,nachmias_wired_2024}]\label{thm:local convergence louigi}
We have the convergence 
\begin{align*}
 M_n \underset{n\rightarrow \infty}{\rightarrow} \mathrm M^\infty,
\end{align*}
in the local weak topology, where $\mathrm M^\infty$ is the root component of the WMSF of the PWIT.
\end{theorem} 
The limiting tree is shown in \cite{addario-berry_local_2013} to be one-ended and having roughly cubic volume growth. These results are refined in \cite{nachmias_wired_2024}, where the authors also prove results about the random walk on such a graph; in particular they prove that the spectral dimension and typical displacement exponent are almost surely respectively $\frac{3}{2}$ and $\frac{1}{4}$.

\paragraph{Towards a non-compact scaling limit.}
The goal of this paper is to introduce $\mathcal{M}^\infty$, a continuous non-compact version of the object $M_n$, the MST of $K_n$. 
This should be reminiscent of the case of the uniform spanning tree $T_n$ of $K_n$, which is simply a uniform random labeled tree on $n$ vertices. 
It is well-known in that case that three limiting objects can be associated to $T_n$: its scaling limit called the \emph{Brownian CRT} (after rescaling the distances by $n^{1/2}$), its local weak limit which is an infinite discrete tree called the \emph{Kesten tree}; and its non-compact scaling limit called the \emph{infinite Brownian CRT}.
This non-compact limit arises as several different limits: the scaling limit of the Kesten tree, the local limit of the Brownian CRT (i.e.\ what we see in the CRT when we zoom in around a random point), and as the scaling limit of $T_n$ when the distances are rescaled by a quantity $o(n^{1/2})$. 

In our case, the goal is to introduce the corresponding continuous non-compact object $\mathcal{M}^\infty$ for $M_n$ and prove one of these convergences, namely the analogue of the first one i.e., that $\mathcal{M}^\infty$ is the scaling limit of $\mathrm{M}^\infty$, see Figure~\ref{fig:mst limits}.
The two other convergences will be studied in a separate paper. 

\begin{figure}
	\centering
	\includegraphics[height=7cm]{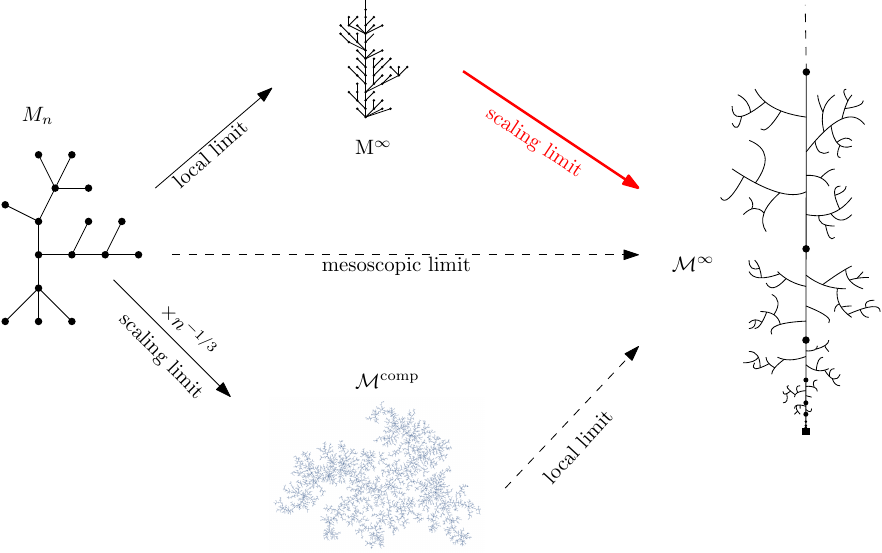}
	\caption{The objects $M_n$, $\mathcal{M}^{\mathrm{comp}}$, $\mathrm{M}^\infty$, and $\mathcal{M}^\infty$ and their relationships. This paper focuses on the red arrow. The solid black arrows are already proved in the literature, and the dashed ones are conjectured.}
	\label{fig:mst limits}
\end{figure}

\subsection{Structure of the root component of the WMSF of the PWIT}
\label{subsec:structure of root component of the WMSF of the PWIT}
\begin{figure}
	\centering
	\subfloat[Decomposition of the invasion percolation tree $\mathrm{T}^\infty$ into ponds that are linked along  an infinite spine by forward maximal edges, drawn in dotted line. 
	The edges in the dark red pond $\mathrm{T}_{\alpha_2}^\bullet$ all have weight smaller than $\alpha_2$, the forward maximal edge that exits it.]{\includegraphics[page=1,height=4cm]{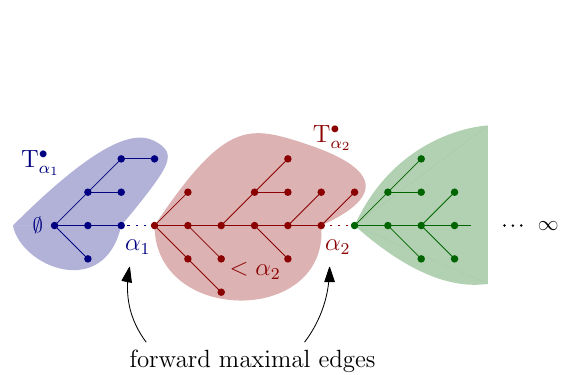}} \qquad 
	\subfloat[Decomposition of the invasion percolation tree $T(u)$ into ponds. The invasion percolation starting from $u$ explores the orange pond first, then the purple one and the dark red one. The rest of the invasion coincides with that of the root from there.]{\includegraphics[page=2,height=4cm]{structureIP.pdf}}
	\caption{Structure of invasion percolation clusters on the PWIT}
	\label{fig:structure of IPC on PWIT}
\end{figure}

First we describe the objects of study.
Let $\bU$ be the Ulam-Harris tree; this is the tree with vertex set
$V (\bU)=\bigcup_{k\geq 0}\N^k$, and for each $k\geq 1$ and each vertex $v = (n_1, \dots, n_k) \in \N^k$, an
edge between $v$ and its parent $(n_1, \dots, n_{k-1})$. 
For each $v = (n_1 ,\dots, n_k ) \in \bU$, let $(x_i^{v} , i \geq 1)$ be the atoms of a
homogeneous rate one Poisson process on $\intervallefo{0}{\infty}$, and for each $i\geq 1$ give the edge from $v$
to its child $v_i = (n_1 , \dots , n_k , i)$ the weight $W ({v, v_i }) =x^{v}_i $. 
Writing $W = \{W (e), e \in E(\bU)\}$,
where $E(\bU)$ is the edge set of $\bU$, the Poisson-weighted infinite tree is the tree $\bU$ rooted at $\emptyset$, endowed with the weight function $W$, which we write as a triple $(\mathbb U,\emptyset, W)$.

\paragraph{Invasion percolation, forward maximal edges and ponds.}
Invasion percolation (also called Prim's algorithm) on a rooted locally finite weighted graph with distinct weights is defined as follows: the algorithm grows a component from the root of the graph, adding at each step the lowest weight edge leaving the current component. 
When performed on a finite connected weighted graph, this procedure yields the minimal spanning tree of that graph. 
Performed on an infinite graph however, it yields an infinite tree whose set of vertices is a proper subset of the original graph, which we call the invasion percolation cluster of the root. 
 
We consider $\mathrm T^\infty$ the invasion cluster of the root of the weighted rooted tree $(\mathbb U,\emptyset, W)$. 
It is known that $\mathrm T^\infty$  is almost surely one-ended and that it admits a decomposition into ponds, by cutting at so-called \emph{forward maximal edges}. 
We say that an edge $e \in E(\mathrm{T}^\infty)$ is forward maximal if it disconnects the root $\emptyset$ from infinity and is such that if there exists $e'\in E(\mathrm{T}^\infty)$ such that $e$ is in the path joining $\emptyset$ and $e'$ then $W(e')<W(e)$. In terms of the invasion percolation process described above, this means that any edge explored after $e$ will have a smaller weight than $e$.
 
Denote by $\alpha_1>\alpha_2> \dots$ the respective weights of the successive forward maximal edges $e_1,e_2,\dots $ in the order in which they are discovered by the invasion percolation process, and denote by $\sA=\{\alpha_1, \alpha_2, \dots\}$ the set of all their values.  
We can remove those edges in order to disconnect $\mathrm T^\infty$ into finite connected components that we call \emph{ponds}.
We denote those ponds by $(\mathrm{T}^\bullet_{\alpha})_{\alpha \in \sA}$, by indexing them by the weight of the corresponding forward maximal edge exiting the component, which we call the \emph{type} (called \emph{activation time} in \cite{nachmias_wired_2024}) of the vertices belonging to the corresponding pond.
Here, each $\mathrm{T}^\bullet_{\alpha}$ is then a finite rooted tree, rooted at its vertex closest to the root and pointed at the vertex from which the next forward maximal edge exits. See Figure~\ref{fig:structure of IPC on PWIT} for an illustration.

\paragraph{Invasion percolation from every vertex.}
For any vertex $u\in  V (\bU)$, 
we can repeat the procedure above \emph{from vertex $u$} by just considering $T(u)$ the invasion percolation cluster of the rooted weighted tree $\bU (u) = (U, u, W )$, the PWIT re-rooted at vertex $u$.
Let $\widetilde{\mathrm M}$ be the union of all the trees $(T (u) , u \in V (\bU))$.
The forest $\widetilde{\mathrm M}$ is the so-called Wired Minimal Spanning Forest of the PWIT. 
Now the tree $\mathrm M^\infty$, which we study in this paper, is defined as the connected component of the forest $\widetilde{\mathrm M}$ containing the root $\emptyset$ of $\bU$. 
It is almost surely one-ended, see \cite[Corollary~7.2]{addario-berry_local_2013}.

The tree $\mathrm M^\infty$ of course contains $\mathrm T^\infty= T(\emptyset)$, the invasion percolation cluster of the root. 
Consider a vertex $u\in V(\bU)$ that is such that $u\notin \rT^\infty$ and its invasion percolation cluster $T(u)$. 
The tree $T(u)$ can be disconnected along its own forward maximal edges in the same way as $T(\emptyset)$. 
Now if we suppose that $u\in \mathrm{M}^\infty$ then the facts that $T(\emptyset)$ and $T(u)$ are one-ended, as well as $\mathrm{M}^\infty$, imply that the ponds and forward maximal edges of $T(\emptyset)$ and $T(u)$ coincide except for a finite number of them: this ensures that the invasion percolation process from vertex $u$ visits a finite number of ponds before starting to explore one of the ponds $(\mathrm{T}^\bullet_{\alpha})_{\alpha \in \sA}$. 
In order to avoid confusion, for $u \in \mathrm{M}^\infty$, we call any of the forward maximal edges of $T(u)$ that is not a forward maximal edge of $T(\emptyset)$ a \emph{back-edge}. 
Any pond of $T(u)$ that does not overlap with $\rT^\infty$ is denoted by $T_\beta$, where $\beta$ is the weight of the first back-edge visited by invasion percolation started from that vertex. 
Again we say that all the vertices in $T_\beta$ have type $\beta$.

Now, consider $\mathrm{M}^\infty$ and as before we disconnect it along the forward maximal edges $e_1, e_2 \dots$ of $\mathrm{T}^\infty$. 
We denote by $(\rM^\bullet_{\alpha})_{\alpha \in \sA}$ the obtained sequence of connected components. 
Each of those component $\rM^\bullet_{\alpha}$ can actually be decomposed again into the union of $\mathrm{T}^\bullet_{\alpha}$ and a set of trees $T_\beta$ connected to each other along a tree structure by back-edges. 
In fact, $\rM^\bullet_{\alpha}$ can be obtained from $\mathrm{T}^\bullet_{\alpha}$ by running the so-called Poisson--Galton--Watson Aggregation Process which we recall in Section~\ref{sec:scaling limit result}. 
The idea of this paper is to define a similar structure in the continuum and prove that it is indeed the scaling limit of $\mathrm{M}^\infty$, resp. $\mathrm{T}^\infty$. 

\subsection{Summary of the results}

\begin{figure}
	\centering
	\includegraphics[height=4cm]{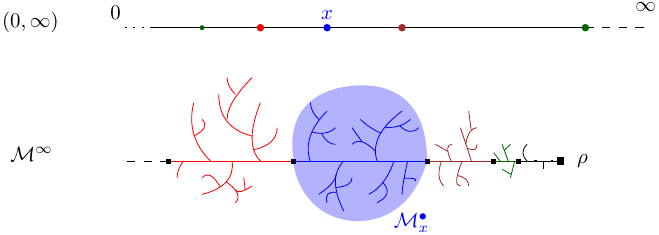}
	\caption{Construction of the non-compact continuous object $\cM^{\infty}$ (on the bottom) from a Poisson point process $\mathscr{P}$ (on top).
	The color of the trees below correspond to the atoms of the PPP above. Note that in $\mathscr{P}$ there is an accumulation of atoms near $\infty$ (not represented here).}
\end{figure}

We define an infinite one-ended continuum random tree $\cM^\infty$ jointly with $\cT^\infty \subset \cM^\infty$, which are the continuous analogues of respectively $\mathrm M^\infty$ and $\mathrm T^\infty$.

\paragraph{Joint construction of $\cT^\infty$ and $\cM^\infty$.}
We first explain how we can describe the object $\cM^\infty$ from another object that we denote by $\cM^\bullet$, which we then discuss below. 
This is slightly analogous to the construction of Brownian motion by chaining of excursions sampled from the Ito measure, though in our construction each component has another one immediately following it (among other differences).
Let 
\begin{itemize}
	\item $\mathscr{P}$ be a Poisson Point Process on $\intervalleoo{0}{\infty}$ with intensity $\frac{1}{t}\mathrm{d}t$, 
	\item $\left((\cT^\bullet_x, \cM^\bullet_x), \ x \in \mathscr{P}\right)$ a collection of random variables, where conditionally on $\mathscr{P}$ they are independent and we have
	\begin{align*}
		(\cT^\bullet_x, \cM^\bullet_x) \overset{(d)}{=} \left(\mathrm{Scale}(x^{-1},x^{-2};\cT^\bullet), \  \mathrm{Scale}(x^{-1},x^{-3};\cM^\bullet)\right),
	\end{align*}
where $(\cT^\bullet,\cM^\bullet)$ is a couple of random pointed rooted measured trees such that $\cT^\bullet\subset \cM^\bullet$ that we will describe in more details later on.
\end{itemize}
Then $\cM^\infty$ is constructed from the sequence $(\cM_x^\bullet)_{x\in \mathscr{P}}$ as 
\begin{align*}
	\mathrm{Chain}(\cM_x^\bullet, \  x\in \mathscr{P}).
\end{align*}
The tree $\cT^\infty$ is constructed in a similar manner from $(\cT_x^\bullet)_{x\in \mathscr{P}}$ and can be seen as a subset of $\cM^\infty$.
The $\mathrm{Chain}$ construction will be defined in detail in Section~\ref{sec:constructions and convergences}, but essentially consists in taking all the sequence of rooted pointed trees $(\cM_x^\bullet)_{x\in \mathscr{P}}$ and gluing them head to tail in decreasing order of $x$-value.
The resulting space is completed by adding an extra point $\rho_\infty$ which is characterized by $\rho_\infty:=\lim_{x\rightarrow \infty, x\in \mathscr{P}}\rho_x$, where for any $x\in \mathscr{P}$, the point $\rho_x$ denotes the common root of $\cT^\bullet_x$ and $\cM^\bullet_x$.
The obtained space is an unbounded complete measured tree, rooted at $\rho_\infty$. 

Notice the similar structure to that of $(\mathrm{T}^\infty, \mathrm{M}^\infty)$: the point process $\mathscr{P}$ corresponds to the set $\sA$ of weights of forward maximal edges, the trees $(\cT^\bullet_x, \cM^\bullet_x)$ for $x\in \mathscr{P}$ take the place of the $(\mathrm{T}_\alpha^\bullet, \rM_\alpha^\bullet)$ for  $\alpha \in \sA$,
and the Chain construction mimics the way that we can get $\mathrm{M}^\infty$ and $\mathrm{T}^\infty$ from linking all the $(\rM_\alpha^\bullet)_{\alpha \in \sA}$ (resp. the $(\mathrm{T}_\alpha^\bullet)_{\alpha \in \sA}$) using the forward maximal edges.

\paragraph{Convergence results.}
Before discussing the construction of $(\cT^\bullet,\cM^\bullet)$, which are the building blocks of $\cT^\infty$ and $\cM^\infty$, we state our scaling limit convergence result.  
The result below is expressed in the local Gromov--Hausdorff--Prokhorov topology, the definition of which is recalled in Section~\ref{sec:constructions and convergences}. 
\begin{theorem}\label{thm:convergence of the MST of the PWIT to Minfty}
	The tree $\cM^\infty$ is the scaling limit of $\mathrm M^\infty$ for the local Gromov-Hausdorff-Prokhorov topology.
	More precisely 
	\begin{align*}
		\mathrm{Scale}\left(r,\frac{r^3}{C_0}; \mathrm{M}^\infty \right)\underset{r\rightarrow 0}{\rightarrow} \cM^\infty,
	\end{align*}
	where the constant $C_0$ is defined later in \eqref{eq:def C0}.	
	The above convergence holds jointly with 
	\begin{align*}
		\mathrm{Scale}\left(r,r^2; \mathrm{T}^\infty \right)\underset{r\rightarrow 0}{\rightarrow} \cT^\infty.
	\end{align*}
\end{theorem}

Note that invasion percolation on a $d$-regular tree and its scaling limit were investigated in \cite{angel_scaling_2013}.
While the behavior does not depend significantly on $d$, the extension to the PWIT was only investigated later in \cite{addario-berry_invasion_2012}.
In \cite{angel_scaling_2013}, the tree generated by the invasion percolation process started at the root is studied using its encoding via functions such as the contour function, height function or Lukasiewicz path. 
The scaling limit is then described as well by its encoding by a height function that is described as the solution to some stochastic differential equation. 
Heuristically, the obtained tree should have the same distribution as $\cT^\infty$, 
up to some deterministic scaling, 
even though our description may seem quite different at first glance.

\paragraph{Construction of $\cM^\bullet$ from $\cT^\bullet$: the Brownian tree aggregation process.}
We now describe in more details the distribution of the pair $(\cT^\bullet, \cM^\bullet)$ used to construct $\cT^\infty$ and $\cM^\infty$. 
They are seen as (nested) random metric spaces, endowed with a common root $\rho$ and distinguished point $v$, and measures (one for each) that are not necessarily normalized to  be probability measures. 

First, $\cT^\bullet$ has the distribution of a Brownian tree with random mass $X$, where $X\overset{(d)}{=}\sqrt{2}\cdot \Gamma$  where $\Gamma$ has distribution $ \mathrm{Gamma}(\frac{1}{2})$, endowed with a randomly chosen point $v$ (on top of being rooted at some point $\rho$). 
The construction of $\cM^\bullet$ from  $\cT^\bullet$ then uses a process that we call the \emph{Brownian tree aggregation process} (which is analogous to the Poisson-Galton-Watson aggregation process used in the discrete setting to describe the construction of $\rM_\alpha^\bullet$ for any $\alpha>1$).
A useful analogy for the reader is the stick-breaking construction of the CRT, where line segments of random lengths are glued at uniform points on the union of all past segments.
However, unlike that process,
 there are infinitely many attached trees being added in any time interval.

Here is a first intuitive way of understanding the construction of $\cM^\bullet$.
This is a growth process for a tree, started at time $t=1$ (which turns out to be more convenient) with the \emph{seed} $\cT^\bullet$. 
Immediately after time $t=1$, independent Brownian trees arrive and are attached to the existing tree by having their root identified to a point already present.
At time $t$, attachments of trees of mass $y$ to an element of mass $\mathrm d m$ of the current structure arrive at rate
\begin{equation}\label{eq:rate aggregation tree}
\frac{1}{y^{\frac32} \sqrt{2\pi} } \exp\left(-\frac{t^2}{2}y\right)\mathrm d y \mathrm d m,
\end{equation}
where each corresponding tree of mass $y$ is a Brownian tree with the given mass (conditionally on all the rest). 
The object $\cM^\bullet$ is the limit as $t\to\infty$ of the process defined above, still rooted and pointed at respectively $\rho$ and $v$. 

The way that we define such a process formally deviates slightly from this intuition, and starts with the process giving the total mass of the tree at time $t$.
We would like to define a pure jump Markov process $(\mathcal{W}_t)_{t\geq 1}$ in terms of its generator. 
Since this Markov process is inhomogeneous in time, we actually give the generator of the process $((t,\cW_t))_{t\geq 1}$, which is time-homogeneous.
The generator $\mathbf{A}$ of the process $(t,\cW_t)$ is such that for any smooth and compactly supported function $f:(0,\infty) \times \intervalleoo{0}{\infty} \rightarrow \R$ we have
\begin{align}\label{eq:definition of continuous W}
\mathbf{A}f(t,x)= \partial_t f(t,x) + x\cdot \int_{0}^{\infty}\left(f(t,x+y)-f(t,x)\right) \frac{\exp\left(-\frac{t^2}{2}y\right)}{y^\frac{3}{2}\sqrt{2\pi}}\mathrm d y.
\end{align} 
This process represents the evolution of the total mass that is present at time $t$.
The process $(\cW_t)_{t\ge 1}$ is a subordinator and a pure jump process, and so it has countably many jumps.
For every jump time $t$ of the process, take an independent Brownian tree
\begin{align*}
	\cT_t
\end{align*}
with mass given by the size of the jump $w_t:=\mathcal{W}_t-\mathcal{W}_{t-}$.
For every such time $t$, a point $X_t$ is sampled proportionally to the mass measure on 
\begin{align}
  \cT^\bullet \sqcup \bigsqcup_{1<s<t} \cT_t.
\end{align}
The tree at time $t$ is obtained by gluing the root of every $\cT_s$, for all jump times $s\le t$ of the process $\mathcal W$, to the corresponding point $X_s$.
For $t=\infty$, we take the completion of the obtained space.
The fact that this construction makes sense and almost surely yields a compact metric space is proved in a more general setting in the corresponding section.

We should also mention how the measure on $\cM^\bullet$ is obtained:
The total mass of the tree constructed up to time $t$ is $\cW_t$ and tends to infinity.
The measure on $\cM^\bullet$ arises as the scaling limit as $t\to\infty$ of the mass-measure on the trees that have index less than $t$. 
As it turns out, the correct normalization is $\frac{1}{t}$.
The measure $\mu$ on $\cM^\bullet$ is defined as $\mu := \lim_{t\rightarrow \infty} \frac1t \mu_t$, where $\mu_t$ is the measure on the tree at time $t$, and the limit is a weak limit of measures.
The limiting measure on $\cM^\bullet$ has total mass given by
\begin{align*}
  Y:=\lim_{t\to\infty} \frac{\mathcal{W}_t}{t}.
\end{align*}
More about that in Section~\ref{sec:the tree Minfty}.

Thanks to the construction of $\cM^\bullet$ as an aggregation process, we can determine some of its geometric properties.
 We call $\bL^{\bullet,c}$ the set of GHP-isometry classes of rooted, measured, compact, pointed length space. 
 This set is endowed with the so-called GHP distance, which makes it Polish (more details in Section~\ref{subsec:reminder about GHP}).

\begin{theorem}\label{thm:properties of Nbullet}
	The random metric space $\cM^\bullet$ is well-defined as a random variable in $\bL^{\bullet,c}$. 
	Furthermore
	\begin{enumerate}[label=(\roman*)]
	\item\label{it:mass and diameter Nbullet moments of all orders} Its total mass and its diameter have moments of all orders.
	\item\label{it:dimension Nbullet is 3} It almost surely has Minkowski and Hausdorff dimension $3$.
	\item\label{it:measure Nbullet is diffuse with full support} Its measure is almost surely non-zero, diffuse, and has full support.
	\end{enumerate}
      \end{theorem} 
Note that, as a result, the non-compact random metric space $\cM^\infty$ which is a countable union of copies of $\cM^\bullet$ also has Hausdorff dimension $3$ and diffuse mass measure with full support.

\subsection{Organization of this paper}
We start in Section~\ref{sec:constructions and convergences} by recalling some facts about the  Gromov--Hausdorff--Prokhorov topology and introducing the two constructions called $\mathrm{Chain}$ and aggregation process ($\mathrm{Aggreg}$) respectively. 
These constructions are particularly important because we interpret all our objects in this framework. 
In particular, we state some sufficient conditions for the convergence of such structures.
In Section~\ref{sec:the tree Minfty} we then use these constructions to rigorously construct the trees $\cM^\bullet$ and $\cM^\infty$. In particular, we prove Theorem~\ref{thm:properties of Nbullet} in Section~\ref{subsec:properties of the weight process}.
Then in Section~\ref{sec:scaling limit result}, we study the tree $\mathrm{M}^\infty$ and decompose it in a way that use again the $\mathrm{Chain}$ and $\mathrm{Aggreg}$ constructions in an analogous way. 
We then prove Theorem~\ref{thm:convergence of the MST of the PWIT to Minfty} in Section~\ref{subsec:proof of main theorem} by using results from Section~\ref{sec:constructions and convergences} that ensure that the limit of objects constructed as a $\mathrm{Chain}$ (resp. $\mathrm{Aggreg}$) of discrete objects can be described as a $\mathrm{Chain}$ (resp. $\mathrm{Aggreg}$) of continuous ones.

\subsection*{Acknowledgements.}
O.A.'s research was supported in part by NSERC. 
The authors are grateful to an anonymous referee for their thorough reading of the manuscript and their numerous remarks and comments.  
\begin{table}[htbp]\caption{Table of notation of the main objects of the paper}
	\centering 
	\begin{tabular}{l c p{13cm} }
		\toprule
			\multicolumn{3}{c}{}\\
		\multicolumn{3}{c}{\underline{Discrete objects}}\\
		\multicolumn{3}{c}{}\\
		$\sA$ & : & set of weights of forward maximal edges of $\rT^\infty$ \\
		$\rT^\infty$ & : & invasion percolation cluster of the root of the PWIT\\
		$\rT_\alpha^\bullet$ & : & defined for $\alpha\in \sA$; it is the pond located just before the forward maximal edge of weight $\alpha$\\
		$M_n$ & : & minimal spanning tree of $K_n$ endowed with independent random weights \\  
		$\rM^\infty$ & : & the root component of the WMSF of the PWIT\\
		$\rM_\alpha^\bullet$ & : & the connected component of $\rM^\infty$ after removing the forward maximal edges of $\rT^\infty$ that contains $\rT_\alpha^\bullet$\\
        \multicolumn{3}{c}{}\\
        \multicolumn{3}{c}{\underline{Continuous objects}}\\
       	\multicolumn{3}{c}{}\\
        $\mathscr{P}$ & : & Poisson point process on $\intervalleoo{0}{\infty}$ with intensity measure $\frac{\dd t}{t}$\\
        $\cT^\infty$ & : & scaling limit of $\rT^\infty$\\
       	$\cT^\bullet$ & : & scaling limit of $\rT_\alpha^\bullet$ \\
       	$\cM^\infty$ & : & scaling limit of $\rM^\infty$\\
       	$\cM^\bullet$ & : & scaling limit of $\rM_\alpha^\bullet$ \\	
       	$(\cW_{t})_{t\geq 1}$ & : & weight process used in the construction of $\cM^\bullet$ \\
       	$X$ & : & total mass of $\cT^\bullet$ and starting value of $(\cW_{t})_{t\geq 1}$, i.e.\ $X=\cW_1$\\
       	$Y$ & : & total mass of $\cM^\bullet$, obtained as $Y=\lim_{t\rightarrow \infty} \frac{\cW_t}{t}$ \\
       	$\mathscr{J}$ &: & set of jumps of the weight process $(\cW_{t})_{t\geq 1}$\\
       	$(w_t)_{t\in \mathscr{J}}$ & : & jumps $w_t=\cW_t - \cW_{t-}$ of the weight process \\
       	$\cT_t$ & : & Brownian CRT with total mass $w_t$ \\
		\multicolumn{3}{c}{}\\
		\multicolumn{3}{c}{\underline{Constructions}}\\
		\multicolumn{3}{c}{}\\
		$\mathrm{mass}(\mathsf{X})$ & : & total mass of the measure carried by the metric measured space $\mathsf{X}$\\
		$\mathrm{diam}(\mathsf{X})$ & : & diameter of the metric (possibly measured) space $\mathsf{X}$\\
		$\mathrm{Scale}(a,b;\mathsf{X})$ & : & version of the metric measured space $\mathsf{X}$ where the distances are scaled by $a$ and the mass is scaled by $b$ \\
		$\mathrm{Chain}$ & : & construction taking a sequence of length spaces and linking them, each one to the next\\
		$\mathrm{Aggreg}$ & : & construction that yields a metric space from the result of an aggregation process\\
		$\mathbb{L}$ & : & set of equivalence classes of rooted, measured, locally compact length spaces\\
		$\mathbb{L}^c$ & : & set of equivalence classes of rooted, measured, compact length spaces\\
		$\mathbb{L}^{\bullet,c}$ & : & set of equivalence classes of rooted, pointed, measured, compact length spaces\\
		\multicolumn{3}{c}{}\\
		\multicolumn{3}{c}{\underline{Other symbols}}\\
		\multicolumn{3}{c}{}\\
		$C_0$ & : & some positive constant defined in \eqref{eq:def C0} \\
		$r$  & : & positive parameter destined to tend to $0$\\
		$\cst$ & : & a positive constant whose value may change along the lines \\
		\bottomrule
	\end{tabular}
	\label{ton:introduction}
\end{table}

 \section{Constructions and convergences}\label{sec:constructions and convergences}
 In this section, we introduce two constructions from sequences of measured metric space. One of them is fairly simple and consists in linking a sequence of metric spaces in a linear manner: we call it the \emph{chain} construction. 
 The other one is slightly more involved and consists of gluing some ordered family of metric spaces on top of each other in a random way: this is called an \emph{aggregation process}. It is a natural extension of the so called "line-breaking" construction of random trees \cite{aldous_continuum_1991,aldous_inhomogeneous_2000,goldschmidt_line_2015}, already extended to more general metric spaces \cite{goldschmidt_stable_2022,senizergues_random_2019}, where we do not constrain the blocks that we glue together to arrive in a discrete order anymore. 
 
 We state some properties of these two constructions (in particular convergence results) and we explain how to interpret our objects of interest in this setting. 
 These results will be the base of the rest of our arguments.  
 
 In this section we only expose results and give an idea of how they are used later in the paper. 
 The proofs of the results stated here will be postponed to the Appendix.
 \subsection{Reminder about the pointed GH and GHP topologies and local versions}\label{subsec:reminder about GHP}
 In this section, we give a brief reminder about the definition of the Hausdorff, Gromov--Hausdorff and Gromov--Hausdorff--Prokhorov topologies on compact metric spaces, and their local equivalent on boundedly finite rooted metric spaces. We will mainly follow the exposition of \cite{abraham_note_2013}. 
 This section can be skipped at first reading.

 Let $(X,d^X)$ be a Polish metric space and $\cb(X)$ be its set of Borel measurable subsets. 
 The diameter of $A\in \cb(X)$ is given by:
 \[
 \diam(A)=\sup\enstq{d^X(x,y)}{x,y\in A}.
 \]
 For $A,B\in
 \cb(X)$, we set:
 \[
 d_\text{H}^X(A,B)= \inf \{ \epsilon >0;\ A\subset B^\epsilon\
 \mathrm{and}\ B\subset 
 A^\epsilon \}, 
 \]
 the Hausdorff distance between $A$ and $B$, where 
 \begin{equation}
 \label{eq:espilon fattening}
 A^\epsilon = \enstq{x\in
 X}{\inf_{y\in A} d^X(x,y) < \epsilon}
 \end{equation}
 is the $\epsilon$- fattening of $A$. 
 If $(X,d^X)$ is compact, then the space of
 compact subsets of $X$, endowed with the Hausdorff distance, is compact.

 Let $\mathrm{Mes}_{f}(X)$ denote the set of all finite Borel measures on $X$. If $\mu,\nu \in
 \mathrm{Mes}_f(X)$, we set:
 \[
 d_\text{P}^X(\mu,\nu) = \inf \enstq{ \epsilon >0}{\mu(A)\le \nu(A^\epsilon) +
 \epsilon 
 \text{ and } 
 \nu(A)\le \mu(A^\epsilon)+\epsilon\ 
 \text{ for any closed set } A }, 
 \]
 the Prokhorov distance between $\mu$ and $\nu$. 
 It is well known, that
 $(\mathrm{Mes}_f(X), d_\text{P}^X)$ is a Polish metric space, and that the
 topology generated by $d_\text{P}^X$ is exactly the topology of weak
 convergence. 
 
 The notion of Prokhorov distance can be extended in the following way. 
 A Borel measure is said to be locally finite if the measure of any
 bounded Borel set is finite. 
 Let $\mathrm{Mes}(X)$
 denote the set of all locally finite Borel measures on $X$. Let
 $\rho$ be a distinguished element of $X$, which we call the
 root. 
 We consider the closed ball of radius $R$ centered at
 $\rho$: 
 \begin{equation}
 \label{eq:X(r)}
 X^{(R)}=\enstq{x\in X}{d^X(\rho,x)\leq R},
 \end{equation}
 and for $\mu\in \mathrm{Mes}(X)$ its restriction $\mu^{(R)}$ to $X^{(R)}$:
 \begin{equation}
 \label{eq:mu(r)}
 \mu^{(R)} (dx)=\mathbf{1}_{X^{(R)}} (x)\; \mu(dx).
 \end{equation}
 If $\mu,\nu \in
 \mathrm{Mes}(X)$, the generalized Prokhorov distance between $\mu$ and $\nu$ is defined as:
 \begin{equation}
 \label{eq:dgP}
 d_\text{gP}^X(\mu,\nu) = \int_0^\infty \exp{-R} \left(1 \wedge
 d^X_{\text{P}}\left(\mu^{(R)},\nu^{(R)}\right)
 \right) \ \dd R.
 \end{equation}
 The function $d_\text{gP}^X$ is well defined and is a
 distance. 
 Furthermore $(\mathrm{Mes}(X), d_\text{gP}^X)$ is a Polish metric
space, and the topology generated by $d_\text{gP}^X$ is exactly the
 topology of vague convergence. 
 When there is no ambiguity on the metric space $(X, d^X)$, we may write $d$,
 $d_\text{H}$, and $d_\text{P}$ instead of $d^X$, $d^X_\text{H}$ and
 $d^X_\text{P}$. 

 If $\Phi:X\rightarrow X'$ is a Borel-measurable map between two Polish metric
 spaces and if $\mu$ is a Borel measure on $X$, we will note $\Phi_*\mu$
 the image measure on $X'$ defined by $\Phi_*\mu(A)=\mu(\Phi^{-1}(A))$,
 for any Borel set $A\subset X$. 
 
 \begin{definition}
 	\label{defi:rwms} 
 	$ $
 	\begin{itemize}
 		\item A rooted measured metric space $\sfX = (X,d, \rho,\mu)$ is a
 		metric space $(X , d)$ with a distinguished element $\rho\in X$,
 		called the root, and a locally finite Borel measure $\mu$.
 		\item Two rooted measured metric spaces $\sfX=(X,d,\rho,\mu)$
 		and $\sfX'=(X',d',\rho',\mu') $ are said to be
 		GHP-isometric if there exists an isometric one-to-one map $\Phi:X
 		\rightarrow X'$ such that $\Phi(\rho)= \rho'$ and 
 		$\Phi_* \mu = \mu'$. In that case, $\Phi$ is called a GHP-isometry. 
 	\end{itemize}
 \end{definition}
 
 Notice that if $(X, d)$ is compact, then a locally finite measure on
 $X$ is finite and belongs to $\mathrm{Mes}_f(X)$. 
 We will now use a
 procedure due to Edwards \cite{edwards_structure_1975} and developed by Gromov \cite{gromov_groups_1981} to compare any two
 compact rooted measured metric spaces, even if they are not subspaces of
 the same Polish metric space.

 \subsubsection{Gromov-Hausdorff-Prokhorov metric for compact spaces}
 Let $(X,d)$ and
 $(X',d')$ be two compact metric spaces. The Gromov-Hausdorff metric
 between $(X,d)$ and $(X',d')$ is given by:
 \begin{equation}
 \label{eq:d-GH}
 d_{\text{GH}}^c((X,d), (X',d')) = \inf_{\varphi,\varphi',Z} 
 d_\text{H}^Z(\varphi(X),\varphi'(X')) , 
 \end{equation}
 where the infimum is taken over all isometric embeddings
 $\varphi:X\hookrightarrow Z$ and $\varphi':X'\hookrightarrow Z$ into some
 common Polish metric space $(Z,d^Z)$. 
The last display actually defines a metric on the set of isometry classes of compact metric spaces.\\
 
 Now, we introduce the Gromov--Hausdorff--Prokhorov distance for compact
 spaces. 
 Let $\sfX=(X,d,\rho,\mu)$ and
 $\sfX'=(X',d',\rho',\mu')$ be two compact rooted measured metric
 spaces, and define:
 \begin{equation} 
 \label{f:def}
 d_{\mathrm{GHP}}^c(\sfX,\sfX') = \inf_{\Phi,\Phi', Z} \left(
 d^Z(\Phi(\rho),\Phi'(\rho')) + d_\text{H}^Z(\Phi(X),\Phi'(X')) + 
 d_\text{P}^Z(\Phi_* \mu,\Phi_*'
 \mu') \right), 
 \end{equation}
 where the infimum is taken over all isometric embeddings
 $\Phi:X\hookrightarrow Z$ and $\Phi':X'\hookrightarrow Z$ into some
 common Polish metric space $(Z,d^Z)$. 
 
 Note that equation \eqref{f:def} does
 not actually define a distance, as
 $d_{\mathrm{GHP}}^c(\sfX,\sfX')=0$ if $\sfX$ and $\sfX'$ are GHP-isometric.
 Therefore, we shall consider $\K$, the set of GHP-isometry classes of
 compact rooted measured metric space and identify a compact rooted
 measured metric space with its class in $\K$. Then the function
 $d_{\mathrm{GHP}}^c$ is finite on $\K^2$. 
 
 \begin{theorem}[Theorem~2.5 of \cite{abraham_note_2013}]
 	\label{theo:dcGHP}
 	$ $ 
 	\begin{itemize}
 		\item[(i)] The function $d_{\mathrm{GHP}}^c$ defines a distance on $\K$.
 		\item[(ii)] The space
 		$(\K, d_{\mathrm{GHP}}^c)$ is a Polish metric space.
 	\end{itemize} 
 \end{theorem}
 
The function $d_{\mathrm{GHP}}^c$ is called the Gromov--Hausdorff--Prokhorov distance. 
The  Gromov--Hausdorff--Prokhorov distance could be defined without reference to any root. 
However, the introduction of the root is necessary to extend the definition of the Gromov--Hausdorff--Prokhorov distance to spaces that are not necessarily compact.

We sometimes need to endow our spaces with extra structure, like an additional distinguished point for example.
In that case, it is possible to modify the definition of the Gromov--Hausdorff--Prokhorov distance to account for this extra piece of data and the corresponding space $\mathbb K^\bullet$ of equivalence classes of root \emph{pointed} compact metric spaces is also Polish.
We denote by $d_{\mathrm{GHP}}^{\bullet, c}$ the corresponding distance.
 \subsubsection{Gromov--Hausdorff--Prokhorov distance for locally compact length spaces}
A metric space $(X, d)$ is a length space if for every $x, y \in X$ , we have $d(x, y) = \inf\{\mathrm{Len}(\gamma)\}$,
- where the infimum is taken over all rectifiable curves $\gamma:\intervalleff{0}{1} \rightarrow X$ such that $\gamma(0) = x$
 and $\gamma(1) = y$ , and where $\mathrm{Len}(\gamma)$ is the length of the rectifiable curve $\gamma$.

 Let $\mathbb L$ be the set of GHP-isometry classes of rooted, measured, complete
	and locally compact length spaces and identify a rooted, measured, complete
 	and locally compact length space with its class in $\mathbb L$. 
 	We will also denote by $\mathbb L^c=\mathbb K \cap \mathbb L$, and $\mathbb L^{\bullet,c}$ for the space of their pointed counterparts.
 
 If $\sfX=(X,d,\rho,\mu)\in \mathbb L$, then for $R\geq 0$ we will consider
 its restriction to the closed ball of radius $R$ centered at
 $\rho$, denoted by $\sfX^{(R)}=(X^{(R)}, d^{(R)}, \rho^{(R)}, \mu^{(R)})$, where
 $X^{(R)} $ is defined by \eqref{eq:X(r)}, the distance $d^{(R)}$ is the
 restriction of $d$ to $X^{(R)}$, and the measure $\mu^{(R)}$ is defined
 by \eqref{eq:mu(r)}. 
If $\sfX$ belongs to $\mathbb L$ , then $\sfX^{(R)}$ belongs to $\K$ for all $R\geq 0$.
 
The following distance is well defined on $\mathbb L \times \mathbb{L}$:
 \[
 d_{\mathrm{GHP}}(\sfX,\sfY) = \int_0^\infty e^{-R} \cdot  \left(1 \wedge
 d^c_{\mathrm{GHP}}\left(\sfX^{(R)},\sfY^{(R)}\right) 
 \right)  \dd R.
\]
The weight $e^{-R}$ and truncation are arbitrary and their choice does not affect the resulting topology.

\begin{theorem}[Theorem~2.9 and Proposition~2.10 of \cite{abraham_note_2013}]
	We have the following:
 	\begin{itemize}
 		\item[(i)] The function $d_{\mathrm{GHP}}$ defines a distance on $\mathbb L$.
 		\item[(ii)] The space $(\mathbb L, d_{\mathrm{GHP}})$ is a Polish metric space.
 		\item[(iii)] The two distances $d_{\mathrm{GHP}}$ and $d_{\mathrm{GHP}}^c$ define the same topology on $\mathbb K \cap \mathbb L$.
 	\end{itemize}
\end{theorem}
All the objects that we study in this paper can be understood in that context of length spaces. 
The discrete objects that we are considering (discrete trees) can be seen as length spaces by considering their \emph{cable graph}, meaning that we consider every edge in the tree as a segment of length $1$. 
The continuous objects that we are considering are all weak limits of such discrete objects that are indeed length spaces. 
This makes them length spaces as well by the fact that taking Gromov--Hausdorff limits preserves the property of being a length space, see \cite[Theorem~7.5.1]{burago_course_2001}. 
 
\subsection{The chain construction}\label{subsec:chain construction}

In this section, we present a construction that takes a family of compact measured metric spaces and glues them sequentially along a path. 

\subsubsection{Construction}

Suppose that we are given a discrete set $\mathscr{N}\subset \intervalleoo{0}{\infty}$ and a family $(\mathsf L_t^\bullet, t\in \mathscr{N})$ of rooted pointed measured compact length spaces indexed by the set $\mathscr{N}$, where
\begin{align*}
  \mathsf L_t^\bullet=(L_t,d_t,\rho_t,v_t,\nu_t).
\end{align*} 
We would like to define the object obtained by linking every $\mathsf L_t^\bullet$ by identifying the distinguished point $v_t$ to the root $\rho_{t'}$ of the next block in ascending order.  
For the result to be well defined we make the following assumptions.

Assume that
\begin{enumerate}[label=(Ch\arabic*)]
 	\item\label{assum:chain1} For any $A>0$ we have 
 	\begin{align*}
 	\sum_{t\in \mathscr{N}\cap \intervalleoo{A}{\infty}} d_t(\rho_t,v_t)<\infty \quad \text{ and } \quad \sum_{t\in \mathscr{N}\cap \intervalleoo{A}{\infty}} \nu_t(L_t)<\infty
 	\end{align*}
 	\item \label{assum:chain2}We have
 	\begin{align*}
 	\sup_{t\in \mathscr{N}\cap \intervalleoo{A}{\infty}} \diam(\mathsf L_t^\bullet) \underset{A \rightarrow\infty}{\rightarrow} 0.
 	\end{align*}
 	\item\label{assum:chain3} For any $\epsilon>0$, one of the two following conditions is satisfied:
 	\begin{enumerate}[label=(Ch3\alph*)]
 	 \item\label{assum:chain3a} $\sum_{t\in \mathscr{N}\cap \intervalleoo{0}{\epsilon}} d_t(\rho_t,v_t)=\infty$,
 	\item\label{assum:chain3b} $\mathscr{N}\cap \intervalleoo{0}{\epsilon}$ is finite.
 	\end{enumerate}
\end{enumerate}
We define $L_\infty=\{\rho_\infty\}$ and  consider
\begin{align*}
  \bigsqcup_{t\in \mathscr N\cup \{\infty\}} L_t
\end{align*}
which we endow with the pseudo-distance $d$ defined as
\begin{equation}
  d(x,y)=d(y,x)=
  \begin{cases}
    d_t(x,y) &\text{if } x,y\in L_t,\\
    d_t(x,v_t) + \displaystyle\sum_{r \in\mathscr{N}\cap(s,t)}d_r(\rho_r,v_r)  + d_s(\rho_s,y) &\text{if } (x,y)\in L_s\times L_t,\ t<s.
  \end{cases}
\end{equation}
Then $\mathrm{Chain}(\mathsf L_t^\bullet, t\in \mathscr{N})$ is defined as
\begin{align*}
  \left(\bigsqcup_{t\in \mathscr N\cup \{\infty\}} L_t\right)/\sim
\end{align*}
where $\sim$ is the equivalence relation generated by $x\sim y$ if $d(x,y)=0$. 
The space $\mathrm{Chain}(\mathsf L_t^\bullet, t\in \mathscr{N})$ is seen as a rooted measured metric space after endowing the space with the projection of the distance $d$ and the measure $\sum_{t\in \mathscr N} \nu_t$ and the root $\rho_\infty$.
 
\begin{lemma}\label{lem:chain is well defined}
  If assumptions \ref{assum:chain1}, \ref{assum:chain2}, \ref{assum:chain3} hold then $\mathrm{Chain}(\mathsf L_t^\bullet, t\in \mathscr{N})$ is an element of $\mathbb L$.
\end{lemma}
The proof of this lemma can be found in Appendix~\ref{app:chain}.

\begin{remark}
If only \ref{assum:chain3} fails, then a length space in $\mathbb L$ can still be constructed from $\mathrm{Chain}(\mathsf L_t^\bullet, t\in \mathscr{N})$, except that it needs to be completed with a second point $\rho_0$ on the other side.
Also, note that in our definition, we exclude the possibility that the set $\mathscr N$ could be dense. 
The construction could be made more general to accommodate for such situations, but this would require the addition of more (possibly an uncountable infinity) completion points in between the links of the chain. 
Such a construction appears for example in \cite[Section~6.1]{senizergues_decorated_2023}.
\end{remark}

\subsubsection{Convergence of chains}

In order to prove scaling limit convergence results for the discrete objects mentioned above to the continuous ones, we provide here sufficient conditions for convergence. 
Since we will consider the convergence of random such objects, we state the condition for a random object constructed as $\mathrm{Chain}(\mathsf L_t^{\bullet,r}, t\in \mathscr{N}^r)$ to converge to a limit expressed as $\mathrm{Chain}(\mathsf L_t^{\bullet,0}, t\in \mathscr{N}^0)$ as $r\rightarrow 0$.
We assume that for any $r\geq 0$, so in particular for $r=0$ as well, the point process $\mathscr{N}^r$ and the family $(\mathsf L_t^{\bullet,r}, t\in \mathscr{N}^r)$ satisfy  \ref{assum:chain1}, \ref{assum:chain2}, \ref{assum:chain3} almost surely. 
Additionally we assume the following:
 \begin{enumerate}[label=(ChConv\arabic*)]
 	\item \label{assum:convergence chains 1}
 	The process $((t,\mathsf L_t^{\bullet,r}),\ t\in \mathscr N^r)$ converges in distribution as $r\rightarrow0$ to $((t,\mathsf L_t^{\bullet,0}),\ t\in \mathscr N^0)$ as point processes on $\intervalleoo{0}{\infty}$ with marks in  $\mathbb{L}^{\bullet,c}$.
 	We mean by that that for any compact set $K\subset \intervalleoo{0}{\infty}$, the random measure 
 	\begin{align*}
\sum_{t \in \mathscr N^r \cap K} \delta_{(x, \mathsf L_t^{\bullet,r})} \overset{(d)}{\rightarrow} \sum_{t \in \mathscr N^0 \cap K} \delta_{(x, \mathsf L_t^{\bullet,0})}
 	\end{align*}
 as $r\rightarrow 0$ for the topology of the weak convergence of measures on $K\times \mathbb{L}^{\bullet,c}$.
 	\item \label{assum:convergence chains 2} For any $\epsilon>0$
 	\begin{align*}
 	\limsup_{A\rightarrow\infty }\limsup_{r\rightarrow0 }\Pp{ \mathrm{mass}\left(\mathrm{Chain}(\mathsf L_t^{\bullet,r}, t\in \mathscr{N}^r\cap \intervalleoo{A}{\infty})\right) >\epsilon}=0.
 	\end{align*}
 	\item \label{assum:convergence chains 3} For any $\epsilon>0$
 	\begin{align*}
 	\limsup_{A\rightarrow\infty }\limsup_{r\rightarrow0 }\Pp{ \mathrm{diam}\left(\mathrm{Chain}(\mathsf L_t^{\bullet,r}, t\in \mathscr{N}^r\cap \intervalleoo{A}{\infty})\right) >\epsilon}=0.
 	\end{align*}
 \end{enumerate}
 Then we have the following result.
 \begin{proposition}\label{prop:convergence of chains}
Suppose that for all $r> 0$, the family $(\mathsf L_t^{\bullet,r}, t\in \mathscr{N}^r)$ satisfies \ref{assum:chain1}, \ref{assum:chain2} and \ref{assum:chain3} almost surely. 
Additionally, assume that $(\mathsf L_t^{\bullet,0}, t\in \mathscr{N}^0)$ satisfies \ref{assum:chain1}, \ref{assum:chain2} and \ref{assum:chain3a} almost surely. 
Then, under \ref{assum:convergence chains 1},\ref{assum:convergence chains 2}, \ref{assum:convergence chains 3}, 
we have the following convergence in distribution in the space $\mathbb L$,
 	\begin{align*}
 	\mathrm{Chain}(\mathsf L_t^{\bullet,r}, t\in \mathscr{N}^r) \underset{r\rightarrow0}{\rightarrow} \mathrm{Chain}(\mathsf L_t^{\bullet,0}, t\in \mathscr{N}^0).
 	\end{align*}
 \end{proposition}
The proof of this proposition can be found in Appendix~\ref{app:chain}.

\subsection{Aggregation processes}\label{subsec:presentation aggregation process}

The second construction that we will use is that of \emph{aggregation processes}. 
In this case, we create a metric space by gluing together a (potentially infinite) collection of measured metric spaces $\left(\mathsf{B}_t, \ t\in \mathscr J\right)$ in a \emph{random} manner.
The reader familiar with the stick-breaking construction of the CRT may wish to keep that in mind as a useful analogy.
Indeed, the stick breaking construction can be defined as a special case of an aggregation process as defined here.
However, unlike the stick breaking construction, here it is possible that infinitely many blocks are being glued within a finite time interval.
(For any $t<\infty$ the blocks used up to time $t$ still have finite total mass.)

\subsubsection{Construction}

Since the definition of this process is a bit more involved, we express it for a particular case where the blocks $\left(\mathsf{B}_t, t\in \mathscr J\right)$ that we glue together are random and satisfy some \emph{ad hoc} assumptions that will be satisfied by the objects studied in this paper. 
We start with 
\begin{itemize}
\item a \emph{weight process} $(W_t)_{t\geq t_0}$, which is a càdlàg increasing pure-jump process, such that $W_{t_0}>0$, (we denote by $\mathscr J$ its set of jump times),
\item a \emph{seed} 
$\mathsf S^\bullet=(S,\rho,v,d_S,\nu_S)$
which is a compact, rooted, pointed, measured length space (\emph{i.e.} an element of $\mathbb L^{\bullet,c}$) and has total mass equal to $W_{t_0}$. 
\item a \emph{block-distribution family} $\eta=\left(\eta(t,w): \ t>t_0, w>0\right)$ of distributions on $\mathbb L^{c}$ which is such that for any $t>t_0,w>0$, an object $\mathsf{B}(t,w)$ sampled under the measure $\eta(t,w)$ almost surely has mass $w$. 
\end{itemize}
From these, we construct below some rooted (pointed) random metric space
\begin{align*} 
	\mathrm{Aggreg}((W_t)_{t\geq t_0},\mathsf S^\bullet,\eta).
\end{align*}
For each jump time $t\in \mathscr J$ we let $w_t=W_t-W_{t-}$ and we let $w_{t_0}=W_{t_0}$.
Then we sample a random \emph{block} $\mathsf B_t=(B_t,d_t,\rho_t,\nu_t)$, whose distribution is $\eta(t,w_t)$, independently for all $t \in \mathscr J$. 
Then the object of interest in constructed from the collection $\left(\mathsf B_t, t\in \mathscr J\right)$, by quotienting the set
\begin{align*}
  S\sqcup \bigsqcup_{t > t_0} B_t,
\end{align*}
by the appropriate identification of pairs of points:
For every jump time $t\in \mathscr J$, we pick a random point $X_t$ on the set $\bigsqcup_{t_0 \leq s <t} B_s$ using a normalized version of the measure $\sum_{t_0 \leq s <t} \nu_s$.
Define the equivalence relation $\sim$ as generated by the relations $X_t\sim \rho_t$ for each $t\in \mathscr J$.
The object of interest is then the \emph{metric gluing}
\begin{align*}
  (A_\infty,\rho,v,d) :=\overline{\left(\mathsf S^\bullet \sqcup \bigsqcup_{t> t_0} \mathsf B_t\right)/ \sim}
\end{align*}
in the sense of \cite{burago_course_2001}, where the root $\rho$ and distinguished point $v$ are inherited from the seed $\mathsf S^\bullet$.
The overline in the last display represent the operation of completion with respect to the distance defined by the metric gluing.
Note that the completion and the gluing operation both preserve the property of being a length space.
We will see below that under some reasonable assumptions, this space can be shown to be almost surely compact.

For any $t\geq t_0$, we can consider the probability measure $\mu_t:=W_t^{-1} \sum_{t_0\leq s \leq t} \nu_s$. 
Again, under some reasonable assumptions, this sequence of measures almost surely converge for the weak topology as $t\rightarrow\infty$ so that $A_\infty$ can be endowed with a probability measure $\mu_\infty:=\lim_{t\rightarrow \infty} \mu_t$.
This makes 
\begin{align*}
	\mathrm{Aggreg}((W_t)_{t\geq t_0},\mathsf S^\bullet,\eta):=\mathsf A_\infty^\bullet=(A_\infty,\rho,v,d,\mu_\infty)
\end{align*}
an element of $\mathbb L^{\bullet,c}$ almost surely.
In what comes next we will work under the following two assumptions, which are satisfied in all the cases studied in this paper.
 \begin{enumerate}[label=(AG\arabic*)]
 	\item \label{assum:weight process} For some deterministic non-increasing function $\delta: \intervallefo{t_0}{\infty} \rightarrow\intervalleff{0}{1}$ that tends to $0$ as $t\rightarrow \infty$, for some random variables $Z_1$ and $Z_2$ we have for all $t\geq t_0$,
 	\begin{align}\label{eq:estimate on the weight process and its jumps}
 	W_t\leq Z_1 \cdot t \qquad \text{ and } \qquad w_t\leq Z_2 \cdot t^{-2+\delta(t)}. 
 	\end{align}
 	\item \label{assum:blocks}
 	We have
 	\begin{align*}
 	\sup_{t\geq t_0, w>0}\Pp{\frac{\diam \mathsf B(t,w)}{\sqrt{w}}\geq x}<C_1\cdot \exp\left(-C_2 x\right)
 	\end{align*} 
\end{enumerate}

The first assumption ensures that the jumps of the process get indeed smaller and smaller in a quantifiable way. 
The second one ensures that the scaling of the distances in the blocks are of the order of the square root of their mass. 
Those assumptions above have been taken so that they apply to our setting in a way that give us some quantitative results that depend on the specific exponents $2$ and $\frac{1}{2}$ that appear in \ref{assum:blocks} and \ref{assum:weight process}.
The arguments that we use in the proofs are quite robust and can be modified to account for other assumptions of the same type. 

\begin{proposition}\label{prop:aggregate is well-defined}
  Under assumptions \ref{assum:weight process} and \ref{assum:blocks}, the object 
  \begin{align*}
   \mathsf {A}_\infty^\bullet=\mathrm{Aggreg}((W_t)_{t\geq t_0},\mathsf S^\bullet,\eta) 
  \end{align*}
 is almost surely well-defined as an $\mathbb L^{\bullet,c}$-valued random variable.
  Furthermore, if the random variables $Z_1$ and $Z_2$ have moments of all orders, then the Hausdorff distance  $d_{\rmH}(\mathsf{A}_\infty^\bullet,\mathsf{S}^\bullet)$ also has moments of all order.
\end{proposition}
The proof of this proposition can be found in Section~\ref{app:subsubsec:proof of prop aggregate is well-defined} of the Appendix. 
We remark that without some assumptions similar to the above it is possible to give examples of similar aggregation processes where the distances between some points in $\mathsf{A}_t$ are infinite, even for finite $t$.


\subsubsection{Extra conditions on the aggregation process to get Hausdorff and Minkowski dimension}
Finally, we introduce more precise conditions on the blocks and weight process (which will be satisfied in our applications) that allow us to compute the almost sure Hausdorff and Minkowski dimension of our object $\mathsf{A}_\infty^\bullet$.
We use $N_\epsilon(\mathsf{X})$ to denote the minimal number of balls of radius $\epsilon$ needed to cover a metric space $\mathsf{X}$.
\begin{enumerate}[label=(AGMink)]
	\item \label{assum:Minkowski coverings} There exists some constant $C$ such that for all $t\geq t_0$ and $w>0$ we have for $\mathsf B(t,w)\sim \eta(t,w)$,
	\begin{align}
		\Ec{N_\epsilon(\mathsf B(t,w))} \leq C \cdot (1\vee  w \epsilon^{-2+\petito{1}}).
	\end{align}
\end{enumerate}
\begin{enumerate}[label=(AGHaus\arabic*)]
	\item\label{assum:aghaus:weight process grows linearly} The weight process $(W_t)_{t\geq t_0}$ satisfies
	\begin{align*}
		W_t\sim Z_3 \cdot  t,
	\end{align*} 
	as $t\rightarrow \infty$ for some random variable $Z_3$.
	\item\label{assum:aghaus:small jumps don't contribute to weight process} The weight process $(W_t)_{t\geq t_0}$ is such that for all $\delta>0$,
	\begin{align*}
		\lim_{t\rightarrow\infty} \frac{\sum_{s\leq t} w_s \ind{w_s \geq s^{-2-\delta}}}{W_t}=1.
	\end{align*}
	\item\label{assum:aghaus:random point on block is far from root} The measure $\nu_{t,w}$ carried on $\mathsf B(t,w)=(B_{t,w},\rho_{t,w},d_{t,w},\nu_{t,w})$ has almost surely full support and there exists positive $a$ and $b$ such that uniformly in $t$ and $w$, for $U(t,w)$ a random point in $\mathsf B(t,w)$ sampled under a normalized version of the measure $\nu_{t,w}$ we have
	\begin{align*}
		\Pp{d(\rho_{t,w},U(t,w))\geq a w^{\frac{1}{2}}}\geq b.
	\end{align*}
\end{enumerate}

\begin{proposition}\label{prop:dimension of aggregate is 3}
	Assume \ref{assum:weight process} and \ref{assum:blocks} are satisfied so that $\mathsf{A}_\infty^{\bullet}$ is almost surely well-defined as a random element of $\mathbb L^{\bullet,c}$.
Then
	\begin{enumerate}[label=(\roman*)]
		\item\label{it:minkowski dimension is less than 3} If we further assume that $\mathsf{S}^\bullet$ has upper Minkowski dimension less than $3$ almost surely and that \ref{assum:Minkowski coverings} holds, then a.s.\ the upper Minkowski dimension of $\mathsf{A}_\infty^{\bullet}$ satisfies
		\begin{align*}
			\overline{\dim}_{\mathrm{Mink}}(\mathsf{A}_\infty^{\bullet})\leq 3.
		\end{align*}
		\item\label{it:Hausdorff dimension is more than 3} If we further assume that \ref{assum:aghaus:weight process grows linearly},\ref{assum:aghaus:small jumps don't contribute to weight process} and \ref{assum:aghaus:random point on block is far from root} hold, then almost surely the measure $\mu_\infty$ is diffuse and has full support and 
		\begin{align*}
			\dim_{\mathrm{Haus}}(\mathsf{A}_\infty^{\bullet}) \geq 3.
		\end{align*}
	\end{enumerate} 
	In particular if all those conditions are satisfied, then the Minkowski dimension of $\mathsf{A}_\infty^{\bullet}$ almost surely exists and $\dim_{\mathrm{Mink}}(\mathsf{A}_\infty^{\bullet})= \dim_{\mathrm{Haus}}(\mathsf{A}_\infty^{\bullet}) =3$ a.s..
\end{proposition}
The proof of this proposition is divided in two parts, the proofs of \ref{it:minkowski dimension is less than 3} and \ref{it:Hausdorff dimension is more than 3} can respectively be found in Section~\ref{subsec:app:upper-bound on Minkowski dimension} and Section~\ref{subsec:app:lower-bound on the Hausdorff dimension} of the Appendix. 

\subsubsection{Convergence of aggregation processes}\label{subsubsec:convergence of aggregation process}

Now, assume that for any $r>0$ we have
a weight process $(W_t^r)_{t\geq t_0}$, a seed
$\mathsf S^{\bullet,r}$ and a block-distribution family $\eta^r=\left(\eta^r(t,w): \ t>t_0, w>0\right)$, for which we assume that \ref{assum:weight process} and \ref{assum:blocks} are satisfied.
We make the following assumptions to ensure a convergence result on aggregation processes.
\begin{enumerate}[label=(AGConv\arabic*)]
 	\item\label{assum:convergence weight process} On every compact interval we have the following convergence in distribution for the Skorokhod topology:
 	\begin{align*}
 	(W_t^r)_{t\geq t_0} \rightarrow (W_t)_{t\geq t_0} \quad \text{as} \quad r\rightarrow 0.
 	\end{align*}
 	\item \label{assum:convergence seed} We have $\mathsf \mathsf S^{\bullet,r}\underset{r\rightarrow 0}{\rightarrow} \mathsf \mathsf S^{\bullet}$ in distribution in $\mathbb L^{\bullet,c}$.
 	\item\label{assum:convergence block distribution} For any $t>t_0$ and $w>0$ we have $\eta^r(t,w)\rightarrow \eta(t,w)$ for the Prokhorov distance (on the set of probability measures over $\mathbb{L}^c$) as $r\rightarrow0$, and the mapping $(t,w)\mapsto \eta(t,w)$ is continuous (for the Prokhorov distance as well). 
 	\item\label{assum:tight control on weight process} For some non-increasing function $\delta: \intervallefo{t_0}{\infty} \rightarrow\intervalleff{0}{1}$ that tends to $0$ as $t\rightarrow \infty$, for some random variables $Z_i^r$ with $i\in\{1,2\}$
 	where the family $(Z_i^r)_{0<r\leq 1}$ is tight, 
 	we have for all $t\geq t_0$,
 	\begin{align}
	W_t^r\leq Z_1^r \cdot t \quad \text{ and } \quad w_t^r\leq Z_2^r \cdot t^{-2+\delta(t)}. 
      \end{align}
 	\item\label{assum:uniform control on tail}There exists constant $c,C>0$ such that we have
 	\begin{align*}
 	\sup_{0<r\leq 1}\sup_{t> t_0, w>0}\Pp{\frac{\diam \mathsf B^r(t,w)}{\sqrt{w}}\geq x} \leq C \exp(-cx).
 	\end{align*} 
\end{enumerate}
      
\begin{proposition}\label{prop:convergence of aggregation processes}
  Under the assumptions above, we have the convergence in $\mathbb{L}^{\bullet,c}$ in distribution 
  \begin{align*}
    \mathrm{Aggreg}((W_t^r)_{t\geq t_0},\mathsf S^{\bullet,r},\eta^r) \underset{r\rightarrow 0}{\rightarrow} \mathrm{Aggreg}((W_t)_{t\geq t_0},\mathsf S^\bullet,\eta),
  \end{align*}
\end{proposition}
The proof of this proposition can be found in Section~\ref{subsec:app:convergence of aggregation processes} of the  Appendix.

\section{The tree $\cM^\infty$}\label{sec:the tree Minfty}

In this section, we use the results presented in the previous section to rigorously construct our object $\cM^\infty$.
First, we need to justify that the 
tree $\cM^\bullet$ 
is well-defined as a random variable in $\mathbb L^{\bullet,c}$. 
%
Then, the tree $\cM^\infty$ is constructed as
\begin{align}\label{eq:def cMinfinity}
	\cM^\infty:=\mathrm{Chain}\left(\cM_x^\bullet,\  x\in \mathscr{P} \right),
\end{align}
 where:
\begin{itemize}
	\item  the point process $\mathscr{P}$ is a Poisson point process with intensity $\frac{\dd t}{t}$ on $\intervalleoo{0}{\infty}$,
	\item  conditionally on $\mathscr P$ the sequence $\left(\cM_x^\bullet, x\in \mathscr{P} \right)$ is independent and
	\begin{align*}
	\cM^\bullet_x\overset{(d)}{=} \mathrm{Scale}(x^{-1}, x^{-3}; \cM^\bullet),
	\end{align*}
	for the random measured tree $\cM^\bullet$.
\end{itemize}
We need to check that conditions \ref{assum:chain1}, \ref{assum:chain2}, \ref{assum:chain3} hold almost surely in order for $\cM^\infty$ to be well-defined. 
In fact, checking that $\Ec{\diam(\cM^\bullet)}$ and $\Ec{\mathrm{mass}(\cM^\bullet)}$ are finite will be enough to ensure that those assumptions hold.

On our way to define $\cM^\bullet$, we first introduce $\widetilde\cM^\bullet$ as the result of an aggregation process 
\begin{align}\label{eq:def of tildeNbullet as aggreg}
	\widetilde\cM^\bullet:=\mathrm{Aggreg}((\mathcal W_t)_{t\geq 1},\cT^\bullet, \eta^\mathrm{Br})
\end{align}
where 
\begin{itemize}
	\item the seed $\cT^\bullet$ is given by a Brownian tree of random mass $X$ with  density $\frac{\exp\left(-\frac{x}{2}\right)}{\sqrt{2\pi x}}\mathrm{d}x$, pointed at a uniform random point,
\item the weight process $(\mathcal W_t)_{t\geq 1}$, started at time $1$ with value $X$, is an increasing pure-jump càdlàg process and the process $((t,\cW_t))_{t\geq 1}$  is a Markov Feller process with generator $\mathbf{A}$ so that for any smooth and compactly supported function $f:(0,\infty) \times \intervalleoo{0}{\infty} \rightarrow \R$ we have
\begin{align}
	\label{eq:second definition of continuous W}
	\mathbf{A}f(t,x)= \partial_t f(t,x) +  \int_{0}^{\infty}\left(f(t,x+y)-f(t,x)\right)\cdot  x \cdot \frac{\exp\left(-\frac{t^2}{2}y\right)}{y^\frac{3}{2}\sqrt{2\pi}}\dd y,
\end{align} 
	\item the block-distribution family $\eta^\mathrm{Br}$ is such that for any $t>1,w>0$, the distribution $\eta^\mathrm{Br}(t,w)$ is that of a Brownian tree with mass $w$. 
	This means that a tree under $\eta^\mathrm{Br}(t,w)$ is distributed as $\mathrm{Scale}(w^{\frac{1}{2}},w;\cT)$, where $\cT$ is the Brownian tree of mass $1$\footnote{Here we use Aldous' normalization: the tree $\cT$ can be defined by its contour function, which is given by twice a Brownian excursion of duration $1$, see \cite[Section~2.3]{legall_random_2006}.}.
\end{itemize}
Then, denoting
$
	Y:= \lim\frac{\mathcal W_t}{t}
$
our random tree $\cM^\bullet$ is obtained as 
\begin{align}\label{eq:def Nbullet as scaled tildeNbullet}
	\cM^\bullet:= \mathrm{Scale}(1,Y; \widetilde\cM^\bullet).
\end{align}
In this section, we want to prove that \ref{assum:weight process} and  \ref{assum:blocks} hold for the construction above, in order to apply Proposition~\ref{prop:aggregate is well-defined}.
We also need to check that the random variable $Y$ indeed exists and is non-zero almost surely.

All of that, along with the proof of other properties of the random metric space $\cM^\bullet$, will be the content of the proof of Theorem~\ref{thm:properties of Nbullet}, which is the goal of this section.

\subsection{Properties of the weight process and proof of Theorem~\ref{thm:properties of Nbullet}}
\label{subsec:properties of the weight process}
We start by stating the properties of the weight process $(\cW_t)_{t\geq 1}$ that we need in the proof of Theorem~\ref{thm:properties of Nbullet} and that we prove in the rest of the section. 
Recall that we write $w_t=\cW_t - \cW_{t-}$ for any $t>1$.
\begin{proposition}\label{prop:properties of the continuous weight process}
	The process $(\mathcal{W}_t)_{t\geq 1}$ introduced above satisfies the following properties
	\begin{enumerate}[label=(\roman*)]
		\item\label{it:cW over t converges as} Almost surely, 
		\begin{align*}
			\frac{\mathcal{W}_t}{t} \underset{t \rightarrow \infty}\rightarrow Y,
		\end{align*}
		where the limiting random variable is almost surely positive.
		\item\label{it:sup of W over t has exponential moments} The random variable 
		\begin{align*}
			Z_1 :=\sup_{t\geq 1} \left(\frac{\mathcal{W}_t}{t}\right)
		\end{align*}
		is almost surely finite and the tail of its distribution decays at most exponentially fast.
		\item\label{it:sup of jumps} The random variable 
		\begin{align*} 
			Z_2 :=\sup_{t\geq 1} \left(\frac{w_t}{t^{-2}(1\vee  \log^2t)}\right)
		\end{align*}
		is almost surely finite and the tail of its distribution decays at most exponentially fast.
		\item\label{it:small jumps dont contribute} For any $\delta>0$ we almost surely have 
		\begin{align*}
			\sum_{1\leq s \leq t}w_s \ind{w_s \leq s^{-2-\delta}} = o(t).
		\end{align*}
	\end{enumerate}
\end{proposition}
We now prove Theorem~\ref{thm:properties of Nbullet} from the definition of $\cM^\bullet$ given by \eqref{eq:def of tildeNbullet as aggreg} and \eqref{eq:def Nbullet as scaled tildeNbullet}, using the properties of the weight process given in Proposition~\ref{prop:properties of the continuous weight process}. 
\begin{proof}[Proof of Theorem~\ref{thm:properties of Nbullet}]
	We first need to apply Proposition~\ref{prop:aggregate is well-defined} to get that $\widetilde{\cM}^\bullet$ is well-defined as a random variable in $\bL^{\bullet,c}$. 
	The fact that $\cM^\bullet$ is well-defined then just follows.
	For that we can check conditions \ref{assum:weight process} and \ref{assum:blocks}. 
	Assumption~\ref{assum:blocks} is actually immediate in this case using the fact that the diameter of the Brownian tree of mass $1$ has an exponential moment (use for example \cite{kennedy_distribution_1976} and the construction of the CRT from the Brownian excursion, see \cite[Section~2.3]{legall_random_2006}). 
	Then, assumption \ref{assum:weight process} is satisfied thanks to items \ref{it:sup of W over t has exponential moments} and \ref{it:sup of jumps} of  Proposition~\ref{prop:properties of the continuous weight process}.
	This entails that $\widetilde{\cM}^\bullet$ is well-defined as a random variable in $\bL^{\bullet,c}$ and the Hausdorff distance between $\widetilde{\cM}^\bullet$ and its seed $\cT^\bullet$ has moments of all orders. Since $\diam(\cT^\bullet)$ has moments of all orders as well, this entails that the same is true for $\diam(\widetilde{\cM}^\bullet)$.
	Now, since $\cM^\bullet=\mathrm{Scale}(1,Y;\widetilde{\cM}^\bullet)$ the above discussion also shows that $\cM^\bullet$ is well-defined and its diameter has moments of all orders. 
	Since Proposition~\ref{prop:properties of the continuous weight process}.\ref{it:sup of W over t has exponential moments} ensures that $Y$ has moments of all order, this allows us to conclude that Theorem~\ref{thm:properties of Nbullet}.\ref{it:mass and diameter Nbullet moments of all orders} holds.
	
	For Theorem~\ref{thm:properties of Nbullet}.\ref{it:dimension Nbullet is 3}, we use Proposition~\ref{prop:dimension of aggregate is 3}. For that it suffices to check that the assumptions \ref{assum:Minkowski coverings} and \ref{assum:aghaus:weight process grows linearly},\ref{assum:aghaus:small jumps don't contribute to weight process} and \ref{assum:aghaus:random point on block is far from root} hold for the aggregation process that defines $\widetilde{\cM}^\bullet$. 
	Conditions \ref{assum:Minkowski coverings} and \ref{assum:aghaus:random point on block is far from root}  follow easily from the fact that $\eta^{\mathrm{Br}}(t,w)$ is the law of a Brownian tree of mass $w$. 
	Conditions \ref{assum:aghaus:weight process grows linearly},\ref{assum:aghaus:small jumps don't contribute to weight process} follow from Proposition~\ref{prop:properties of the continuous weight process}.
	This concludes the proof of Theorem~\ref{thm:properties of Nbullet}.
\end{proof}
The rest of the section is devoted to proving Proposition~\ref{prop:properties of the continuous weight process}.
\subsection{Definition of the weight process from a Poisson Point Process}
\label{subsec:definition of the continuous weight process from PPP}
In this section, we first present an alternative construction of the process $(\cW_t)_{t\geq 1}$ from a Poisson point process. 
This will allow for some easier computations in the next section. 

\paragraph{Some stochastic differential equation driven by a Poisson point process.}
In order to match the setting of \cite{ikeda_stochastic_1989}, we assume in this section that we are working on a probability space $(\Omega,\mathcal F, \P)$ on which we can define the random variable $X$, the total mass of the tree $\cT^\bullet$, together with an independent Poisson point process $\sQ$ on $\intervalleoo{1}{\infty}\times \intervalleoo{0}{\infty}^2$ with intensity
\begin{align*}
	 u(t,y)\cdot \dd t \otimes \dd m \otimes \dd y\qquad \text{where} \qquad u(t,y)=\frac{\exp\left(-\frac{t^2}{2}y\right)}{y^\frac{3}{2}\sqrt{2\pi}}.
\end{align*}
Note already that 
\begin{align}\label{eq:moments of u(t,y)}
	\int_{0}^{\infty}y \cdot u(t,y) \dd y= \frac{1}{t} \qquad \text{and} \qquad  \int_{0}^{\infty}y^2 \cdot u(t,y) \dd y = \frac{1}{t^3},
\end{align}
as this will be useful later on.

Denote by $\mathcal{B}\left(\intervalleoo{0}{\infty}^2\right)$ the set of Borel-measurable subsets of $\intervalleoo{0}{\infty}^2$.
For any $U\in \mathcal{B}\left(\intervalleoo{0}{\infty}^2\right)$ introduce $N_\sQ(t,U):=\#\sQ \cap (\intervalleof{0}{t} \times U)$ and consider $(\mathcal{F}_t)_{t\geq 1}$ the usual augmentation of the filtration generated by those processes together with the random variable $X$, i.e.
\begin{align*}
\sigma\left(X, \left(N_\sQ(s,U), \ 1\leq s \leq t,  \ U\in \mathcal{B}\left(\intervalleoo{0}{\infty}^2\right) \right)\right),
\end{align*}
 so
that $(t\mapsto N_\sQ(t,U))$ is $(\mathcal F_t)$-adapted, for every $U\in \mathcal{B}\left(\intervalleoo{0}{\infty}^2\right)$, and $X$ is $\mathcal F_1$-measurable.
Following \cite[Chapter III, Definition~3.1]{ikeda_stochastic_1989}, it is natural to introduce the corresponding \emph{compensator} measure $\hat{N}_\sQ(t,\cdot)=\Ec{N_\sQ(t,\cdot)}$ where for any $U\in \mathcal{B}\left(\intervalleoo{0}{\infty}^2\right)$ we have $\hat{N}_\sQ(t,U)=\int_{1}^{t}\int_{\intervalleoo{0}{\infty}^2}  \ind{(m,y)\in U} \cdot u(s,y) \dd s \dd m \dd y$.

From there,
for any starting value $x_1$
 we introduce the following equation on an unknown function $(\cX(t))_{t\geq 1}$
\begin{align*}
	\cX(t) = x_1 + \sum_{(s,m,y)\in \sQ} y \cdot \ind{m< \cX(s-),\ s\leq  t}.
\end{align*}
which can also be expressed as the following jump-type SDE
\begin{align}\label{eq:SDE}
	\cX(t) = x_1 + \int_{1}^{t}\int_{\intervalleoo{0}{\infty}^2} g(\cX(s-),m,y) N_\sQ(\dd s\dd m \dd y )
\end{align}
where $g(x,m,y)=y\ind{m<x}$.

We can check that for any $t\geq 1$, for any two values $x',x$ we have 
\begin{align}\label{eq:lipschitz condition}
 	\int_{\intervalleoo{0}{\infty}^2} |g(x',m,y) - g(x,m,y)| u(t,y) \dd m \dd y 
 	\leq  |x'-x|,
 \end{align}
which we call the \emph{Lipschitz} condition.  
We would like to deduce from this condition that the SDE \eqref{eq:SDE} admits a unique strong solution. 
Unfortunately we are not quite in the the classical setting of \cite[Theorem~IV.9.1]{ikeda_stochastic_1989} or \cite[Theorem~III.2.32]{jacod_limit_2003}.
The first reference is aimed at being used for \emph{compensated} sums of jumps and includes a square, i.e.\ requires a condition of the form 
\begin{align*}
	\int_{\intervalleoo{0}{\infty}^2} |g(x',m,y) - g(x,m,y)|^2 u(t,y) \dd m \dd y 
	\leq K |x'-x|^2,
\end{align*}
which does not hold in our setting. 
The second reference asks for a pointwise condition of the form 
\begin{align*}
	|g(x',m,y) - g(x,m,y)| \leq \rho(m,y) \cdot |x-x'|,
\end{align*} 
which again does not hold in our case. 
It would be possible to use the more general results of \cite{xi_jump_2019}, but in the end it seems more straightforward to just prove existence and pathwise uniqueness of a solution from elementary considerations, following the proof of \cite[Theorem~IV.3.1]{ikeda_stochastic_1989}. 

We first check pathwise  uniqueness. 
Let $\cX$ and $\cY$ be two solutions of our SDE with initial condition $x_1$ at time $t=1$, by which we mean that $\cX$ and $\cY$ are two càdlàg, $(\cF_t)_{t\geq 1}$-adapted processes for which \eqref{eq:SDE} holds.
For any $t\geq 1$ we have 
\begin{align*}
	\sup_{1\leq s \leq t}|\cX(s) - \cY(s)| 
	&\leq  \int_{1}^{t}\int_{\intervalleoo{0}{\infty}^2} \left| g(\cX(s-),m,y)-g(\cY(s-),m,y)\right| N_\sQ(\dd s\dd m \dd y ).
\end{align*}
Then, for any $N\geq 1$ we let $\sigma_N=\inf\enstq{t\geq 1}{\cX_t \geq N}$ and $\tau_N=\inf\enstq{t\geq 1}{\cY_t \geq N}$. 
Applying the above inequality to the time $t\wedge \sigma_N\wedge \tau_N$ and taking expectations we get 
\begin{align*}
	&\Ec{\sup_{1\leq s \leq t}|\cX(s\wedge \sigma_N\wedge \tau_N) - \cY(s\wedge \sigma_N\wedge \tau_N)|} \\
	&\leq  \Ec{\int_{1}^{t} \ind{s \leq \sigma_N\wedge \tau_N }\int_{\intervalleoo{0}{\infty}^2} \left| g(\cX(s-),m,y)-g(\cY(s-),m,y)\right| N_\sQ(\dd s\dd m \dd y )}\\
	&\leq  \Ec{\int_{1}^{t} \left| \cX(s-) -\cY(s-)\right| \ind{s \leq \sigma_N\wedge \tau_N } \dd s}\\
	&\leq  \int_{1}^t \Ec{\sup_{1\leq u \leq s}\left| \cX(u\wedge \sigma_N\wedge \tau_N) -\cY(u\wedge \sigma_N\wedge \tau_N)\right|},
\end{align*}
where we used the Lipschitz condition \eqref{eq:lipschitz condition} in the second inequality.
This implies using Gronwall's lemma that for any fixed $t\geq 1$ we have $\Ec{\sup_{1\leq s \leq t}|\cX(s\wedge \sigma_N\wedge \tau_N) - \cY(s\wedge \sigma_N\wedge \tau_N)|}=0$. 
Taking $N\rightarrow \infty$ and using monotone convergence we get $\Ec{\sup_{1\leq s \leq t}|\cX(s) - \cY(s)|}=0$.
This ensures pathwise uniqueness for solutions of \eqref{eq:SDE}.

For the existence of a strong solution, we can use the method of successive approximations by defining $\cX^0$ as the constant function equal to $0$ on $\intervallefo{1}{\infty}$ and defining recursively $\cX^n$ as 
\begin{align*}
\cX^n(t) = x_1 + \int_{1}^{t}\int_{\intervalleoo{0}{\infty}^2} g(\cX^{n-1}(s-),m,y) N_\sQ(\dd s\dd m \dd y).
\end{align*}
Note that from this definition $\cX^1$ is constant equal to $x_1$.
Using the same line of reasoning as for the proof of pathwise uniqueness, one can check that we have, for $n\geq 2$, using the Lipschitz condition \eqref{eq:lipschitz condition},
\begin{align*}
	&\Ec{\sup_{1\leq s \leq t}|\cX^n(s) - \cX^{n-1}(s)|} \\
	&\leq	\Ec{\int_{1}^{t}\int_{\intervalleoo{0}{\infty}^2} |g(\cX^{n-1}(s-),m,y)-g(\cX^{n-2}(s-),m,y)| N_\sQ(\dd s\dd m \dd y ) }\\
	&\leq  \int_{1}^t  \Ec{ |\cX^{n-1}(s) -\cX^{n-2}(s)|}\\
	&\leq  \int_{1}^t \Ec{ \sup_{1\leq u \leq s}|\cX^{n-1}(u) -\cX^{n-2}(u)|}.
\end{align*}
From there, by iterating the above and using that $\cX^1-\cX^0$ is identically equal to $x_1$ we get that for any $n\geq 1$ and for any $t\geq 1$, 
\[\Ec{\sup_{1\leq s \leq t}|\cX^n(s) - \cX^{n-1}(s)|} \leq \frac{(t-1)^{n-1}}{(n-1)!}x_1,\]
so the sequence $(\Ec{\sup_{1\leq s \leq t}|\cX^n(s) - \cX^{n-1}(s)|})_{n\geq 1}$ is summable. 
This ensures that $(t\mapsto\cX^n(t))$ almost surely converges uniformly on every compact interval as $n\rightarrow\infty$ and we can check that the limit $t\mapsto \cX(t)$ is indeed a strong solution to the SDE.

\paragraph{Definition and properties of $(\cW_t)_{t\geq 1}$.}
The process $(\mathcal W_t)_{t\geq 1}$ is then defined as the solution of \eqref{eq:SDE}, started from the value $X$ at time $1$, where $X$ is the total mass of the tree $\cT^\bullet$.
Let us already mention that, from Proposition~\ref{prop:properties of the continuous weight process}\ref{it:sup of W over t has exponential moments}, which we prove in the next subsection, 
for any $t\geq 1$ we have $\Ec{(\cW_t)^2} \leq C t^2$ with $C$ a positive constant. 
Note that by construction, $(\mathcal W_t)_{t\geq 1}$ is almost surely positive and increasing.
Now, writing $\cW_t$ as 
\begin{align}
	\cW_t = X + \int_{1}^{t}\int_{\intervalleoo{0}{\infty}^2} g(\cW_{s-},m,y) N_\sQ(\dd s \dd m \dd y)
\end{align}
ensures that $(\cW_t)_{t\geq 1}$ is a semi-martingale in the sense of \cite[Definition~II.4.1]{ikeda_stochastic_1989}, as the term $g(\cW_{s-},m,y)$ appearing in the integral is $(\mathcal F_s)$-predictable. 

Applying the Itô formula \cite[Theorem~II.5.1]{ikeda_stochastic_1989} with any twice differentiable function $F$ we get that, for $t\geq 1$,
\begin{align*}
F(\cW_{t})-F(\cW_{1})  &=\int_{1}^{t}\int_{\intervalleoo{0}{\infty}^2}\left( F(\cW_{s-}+g(\cW_{s-},m,y))-F(\cW_{s-})\right) N_\sQ(\dd s \dd m \dd y).
\end{align*}
Now, we aim at subtracting a term from the last display in a way that we get a martingale. 
The standard way to do that is to subtract a term containing the same integral but against the compensator $\hat{N}_\sQ(\dd s \dd m \dd y)$ instead of the point process $N_\sQ(\dd s \dd m \dd y)$. 
For this we first need to check some integrability conditions:
We can check that when we integrate the absolute value of the integrand against the compensator $\hat{N}_\sQ(\dd s \dd m \dd y)$ we get
\begin{align*}
	&\int_{1}^{t}\int_{\intervalleoo{0}{\infty}^2}\left| F(\cW_{s-}+g(\cW_{s-},m,y))-F(\cW_{s-})\right| \hat{N}_\sQ(\dd s \dd m \dd y)\\
	 &= \int_{1}^{t}\int_{\intervalleoo{0}{\infty}^2}\left| F(\cW_{s-}+y \cdot \ind{m<\cW_{s-}})-F(\cW_{s-})\right| u(s,y) \dd s \dd m \dd y\\
	 &= \int_{1}^{t}\int_{\intervalleoo{0}{\infty}}\left| F(\cW_{s-}+y)-F(\cW_{s-})\right|\cdot  \cW_{s-} \cdot  u(s,y) \dd s \dd y.
\end{align*}
Now suppose that $F$ is taken so that it additionally satisfies that for any $x,y\geq 0$ we have $|F(x+y)-F(x)|\leq C \cdot (1+x+y)\cdot y$. Note that this holds for any differentiable $F$ with bounded derivative for example, as well as for the function $x\mapsto x^2$. 
For such a function $F$, the above integral is bounded above by 
\begin{align*}
	\int_{1}^{t}\int_{\intervalleoo{0}{\infty}} C \cdot (1+\cW_{s-}+y)\cdot y \cdot  \cW_{s-}\cdot  u(s,y) \dd s \dd y = C \int_{1}^{t} \left(\frac{\cW_{s-}}{s} + \frac{(\cW_{s-})^2}{s} + \frac{\cW_{s-}}{s^3} \right) \dd s,
\end{align*}
and now the expectation of the right-hand-side is finite for any $t\geq1$, thanks to the bound $\Ec{(\cW_s)^2} \leq C s^2$. 
This allows us to use \cite[(3.8) in Chapter~II]{ikeda_stochastic_1989}, which ensures that the compensated process
\begin{align}\label{eq:dynkin formula from stochastic calculus}
	F(\cW_{t})-F(\cW_{1}) - \int_{1}^{t}\int_{\intervalleoo{0}{\infty}}\left( F(\cW_{s-}+y)-F(\cW_{s-})\right)\cdot  \cW_{s-} \cdot  u(s,y) \dd s \dd y
\end{align}
is an $(\mathcal F_t)$-martingale. This will be useful later on.
Note in particular that
\begin{align*}
\Ecsq{F(\cW_{t+h})-F(\cW_{t})}{\mathcal F_t} 
&= \int_{t}^{t+h}\int_{\intervalleoo{0}{\infty}}\Ecsq{\left( F(\cW_{s-}+y)-F(\cW_{s-})\right) \cW_{s-})}{\mathcal F_t} u(s,y) \dd s \dd y\\
&\underset{h\rightarrow 0+}{=} h\cdot \int_{\intervalleoo{0}{\infty}}\left(F(\cW_t+y)-F(\cW_t)\right) \cdot \cW_t\cdot  u(t,y) \dd y + o(h).
\end{align*}
This entails that $(\cW_t)_{t\geq 1}$ is a time-inhomogeneous Feller process whose generator $\bfA_t$ at time $t$ acts on functions $f$ as 
\begin{align*}
	\bfA_t f (x) = \int_{\intervalleoo{0}{\infty}}\left(f(x+y)-f(x)\right) \cdot x \cdot  u(t,y) \dd y.
\end{align*}
	
The reason for describing in \eqref{eq:second definition of continuous W} the generator of the process $((t,\cW_t))_{t\geq 1}$ instead of $(\cW_t)_{t\geq 1}$ is just a matter of technicality in the proofs, the former is time-homogeneous whereas the latter is not. 
From the above we get the following lemma.
\begin{lemma}\label{lem:regular dynkin formula}
	The process $((t,\cW_t))_{t\geq 1}$ defined above is a time-homogeneous Markov Feller process with generator $\mathbf A$ characterized by its action on, say, the smooth functions $f:\intervallefo{1}{\infty}\times \R_+ \rightarrow \R$ with compact support, with 
	\begin{align}\label{eq:definition generator weight process}
		\mathbf A f(t,x)= \partial_tf(t,x)+ \int_{0}^{\infty}\left(f(t,x+y)-f(t,x)\right) \cdot  x \cdot \frac{\exp\left(-\frac{t^2}{2}y\right)}{y^\frac{3}{2}\sqrt{2\pi}} \cdot \dd y.
	\end{align}
\end{lemma}

\paragraph{Intuition behind this description of $(\cW_t)_{t\geq 1}$.}
From all of the above, the process $(\mathcal W_t)_{t\geq 1}$ satisfies $\cW_1=X$ and for all $t\geq 1$,
\begin{align}\label{eq:def Wtx}
	\mathcal W_t = X + \sum_{(s,m,y)\in \sQ} y \cdot \ind{m< \mathcal W_{s-},\ s\leq  t}.
\end{align} 
A way to imagine this process informally is to think of the first dimension as time, the second dimension as space, and the third one as a count of mass, so that every point $(s,m,y)$ is seen as a point at time $s$, at position $m$, with mass $y$.
This way we can imagine letting the time run starting from $1$ and following the value $\mathcal W_t$ of the function started at height $X$. 
Then at any time $t$ that an atom $(t,m,y)$ is present, if the point is below $\mathcal W_{t-}$ then the function grows by an amount  $w_t=\mathcal W_t-\mathcal W_{t-}$, which corresponds to the mass of that point.
\subsection{A priori estimates on the weight process}\label{subsec:a priori estimates on the weight process}
For any measurable subset $E$ of $\intervalleoo{1}{\infty} \times  \intervalleoo{0}{\infty}^2$ we define
\begin{align}\label{eq:definition MQ(E)}
	M^{\sQ}(E):= \sum_{(s,m,y)\in \sQ} y \cdot \ind{(s,m,y)\in E}.
\end{align}
In the time-space-mass interpretation of the construction appearing at the end of Section~\ref{subsec:definition of the continuous weight process from PPP}, the quantity $M^{\sQ}(E)$ corresponds to the total mass of points present in the set $E$.
We are interested in this quantity because of the result stated below, which states that if the process $\cW$ is below some level $A$ at time $t$ and the total mass $M^{\sQ}(\intervalleof{t}{t'}\times \intervalleof{0}{A+B} \times \intervalleoo{0}{\infty} )$ of points present below level $A+B$ in the time interval $\intervalleof{t}{t'}$ is smaller than $B$, then the process cannot reach level $A+B$ by time $t'$.

\begin{lemma}\label{lem:control of the weight process through its increments}
For two different times $t<t'$ and any two constants $A,B>0$, on the event where
\begin{align*}
	\mathcal W_t\leq A \qquad \text{and} \qquad M^{\sQ}(\intervalleof{t}{t'}\times \intervalleof{0}{A+B} \times \intervalleoo{0}{\infty} )< B
\end{align*}
we have 
\begin{align*}
	\mathcal W_{t'}< A + B.
\end{align*}
\end{lemma}
\begin{proof}
	Assume that we are on the event where the first condition is satisfied and where $\mathcal W_{t'}> A + B$. 
Then there exists some $t''= \inf\enstq{t\leq s\leq t'}{\mathcal W_s \geq A+B}$. We then have 
\begin{align*}
	A+B\leq \mathcal W_{t''} &= \mathcal W_{t}+ \sum_{(s,m,y)\in \sQ} y \cdot \ind{m< \mathcal W_{s-},\ t< s\leq t''} \\
	&\leq A+ M^{\sQ}(\intervalleof{t}{t'}\times \intervalleof{0}{A+B} \times \intervalleoo{0}{\infty} ) < A+B,
\end{align*}
which is a contradiction, so the event that we considered in the first place is empty. 
This concludes the proof by contradiction.
\end{proof}

Using Campbell's theorem for Poisson point processes (see for example \cite[Section~3.2]{kingman_poisson_1993} for a reference), we get that for any $\lambda \in \mathbb C$ for which the right-hand-side makes sense we have
\begin{align*}
	\Ec{\exp(\lambda \cdot M^{\sQ}(\intervalleof{c}{d}\times \intervalleof{a}{b} \times \intervalleoo{0}{\infty} ))} 
	&= \int_c^d \dd{t} \int_a^b\dd{m} \int_{0}^{\infty} (e^{\lambda y}-1) \cdot \frac{e^{-\frac{t^2}{2}y}}{y^\frac{3}{2}\sqrt{2\pi}} \dd{y}\\
	 &= (b-a) \int_c^d F(\lambda,t)\dd{t},
\end{align*}
where we write $F(\lambda,t)= \int_{0}^{\infty} (e^{\lambda y}-1) \cdot \frac{e^{-\frac{t^2}{2}y}}{y^\frac{3}{2}\sqrt{2\pi}} \dd{y}$, again, for values of $\lambda$ for which this integral makes sense. 
Now we have the following control on the values of $F(\lambda,t)$.
\begin{lemma}\label{lem:expansion for F(lambda,t)}
For $\lambda<\frac{t^2}{2}$ the value $F(\lambda,t)$ is well-defined and 
\begin{align*}
	F(\lambda,t)\underset{\substack{\lambda,t\rightarrow\infty \\ \lambda \ll t^2}}{=} \frac{\lambda}{t} + O\left(\frac{\lambda^2}{t^3}\right)
\end{align*}
\end{lemma}
\begin{proof}
Using an integration by parts we get 
\begin{align*}
F(\lambda,t)&= \int_{0}^{\infty} (e^{\lambda y}-1) \cdot \frac{e^{-\frac{t^2}{2}y}}{y^\frac{3}{2}\sqrt{2\pi}} \dd{y} \\
&= \left[(-2y^{-\frac{1}{2}}) \cdot  \frac{e^{-(\frac{t^2}{2} - \lambda)y} - e^{-\frac{t^2}{2} y}}{\sqrt{2\pi}}\right]_{y\rightarrow 0}^{y \rightarrow \infty} -  \int_0^\infty (-2y^{-\frac{1}{2}})  \frac{- (\frac{t^2}{2} - \lambda) e^{-(\frac{t^2}{2} - \lambda)y} + \frac{t^2}{2}e^{-\frac{t^2}{2} y}}{\sqrt{2\pi}} \dd y.
\end{align*}
The first term is just 0. For the second term, we use the fact that for any $z>0$ we have $\int_0^\infty y^{-\frac{1}{2}} e^{-zy}\dd y = \sqrt{\frac{\pi }{z}}$ to compute the integral and get that
\begin{align*}
F(\lambda,t)&= \sqrt{2} \cdot \left( \sqrt{\frac{t^2}{2}} - \sqrt{\frac{t^2}{2} - \lambda}\right) = t \cdot \left(1 - \sqrt{1-\frac{2\lambda }{t^2}}\right). 
\end{align*}
The proof of the lemma follows easily from this expression for $F(\lambda,t)$.
\end{proof}
Using those two results (Lemma~\ref{lem:control of the weight process through its increments} and Lemma~\ref{lem:expansion for F(lambda,t)}) repeatedly, together with some Chernoff bounds, we can obtain some tail bound for the values of the process $(\mathcal W_t)_{t\ge 1}$ and its jumps.
Below, we prove points \ref{it:sup of W over t has exponential moments}, \ref{it:sup of jumps} and \ref{it:small jumps dont contribute} of Proposition~\ref{prop:properties of the continuous weight process}.
\begin{proof}[Proof of Proposition~\ref{prop:properties of the continuous weight process}.\ref{it:sup of W over t has exponential moments}]
Let $(A_k)_{k\geq 1}$ be an increasing sequence of positive real numbers that we will specify later on. 
Using a union bound and the Markov property at times $k\geq 1$, we get
\begin{align*}
	\Pp{\exists k \geq 1,\ \mathcal W_k > A_k}\leq \Pp{\mathcal W_1> A_1}+ \sum_{k=1}^\infty\Ppsq{\mathcal W_{k+1} > A_{k+1}}{\mathcal W_k \leq A_k}.
\end{align*}
From Lemma~\ref{lem:control of the weight process through its increments} we get that 
\begin{align*}
	\Ppsq{\mathcal W_{k+1} > A_{k+1}}{\mathcal W_k \leq A_k}\leq  \Pp{M^{\sQ}(\intervalleof{k}{k+1}\times \intervalleof{0}{A_{k+1}} \times \intervalleoo{0}{\infty})> (A_{k+1}-A_k)}.
\end{align*}
Using a Chernoff bound and Lemma~\ref{lem:expansion for F(lambda,t)}, for any $k\geq 1$ and $\lambda<k^2/2$,  
\begin{align}\label{eq:increment process W chernoff bound}
	&\Pp{M^{\sQ}(\intervalleof{k}{k+1}\times \intervalleof{0}{A_{k+1}} \times \intervalleoo{0}{\infty})> (A_{k+1}-A_k)}\notag \\
	&\leq \exp \left(-\lambda(A_{k+1}-A_k) + A_{k+1}\cdot \int_{k}^{k+1} F(\lambda,t)\dd{t}\right) \notag \\
	&\leq \exp \left(-\lambda(A_{k+1}-A_k) + A_{k+1}\cdot \left(\lambda \log(1+\frac{1}{k})+ O\left(\frac{\lambda^2}{k^3}\right)\right)\right).
\end{align}
For our purposes, we fix some $A>0$ and use the sequence $(A_k)$ defined in such a way that
\begin{align*}
 A_1=A  \qquad \text{and} \qquad A_{k+1}=A_k \left(1+\frac{1}{k} +\frac{1}{k^{\frac{5}{4}}}\right).
\end{align*}
From this, remark that we have $A\cdot k\leq A_k\leq C \cdot A \cdot k$ for all $k\geq 1$ for some constant $C$ and then 
\begin{align*}
A_{k+1}-A_k=A_k\cdot \left(\frac{1}{k}+ \frac{1}{k^{\frac{5}{4}}}\right) \geq \frac{A_k}{k} + \frac{A}{k^{\frac14}} \qquad  \text{and} \quad \frac{A_{k+1}}{k} \leq \frac{A_k}{k} +\frac{C'}{k},
\end{align*}
for some constant $C'$.
Using the above and writing $\log(1+\frac{1}{k})=\frac{1}{k} + O\left(\frac{1}{k^2}\right)$ we can bound the term that appears in the exponential in \eqref{eq:increment process W chernoff bound} as
\begin{align*}
&-\lambda(A_{k+1}-A_k) + A_{k+1}\cdot \left(\lambda \log(1+\frac{1}{k})+ O\left(\frac{\lambda^2}{k^3}\right)\right) \\
&\leq - \lambda \left(\frac{A_k}{k} + \frac{A}{k^{\frac14}} \right) + \lambda A_{k+1}\cdot \left( \frac{1}{k} + O\left(\frac{1}{k^2}\right)\right) +  O\left(A_{k+1} \cdot \frac{\lambda^2}{k^3}\right)\\
&\leq  \lambda\cdot \left( -\frac{A}{k^{\frac{1}{4}} }+ O\left(\frac{1}{k}\right)\right) + O\left(\frac{\lambda^2}{k^{2}}\right).
\end{align*}
Now taking $\lambda= k^{\frac{5}{4}}$ and $A$ large enough, the bounds \eqref{eq:increment process W chernoff bound} are summable in $k$ and the obtained sum is exponentially small in $A$. 
This proves that with complementary probability we have $\cW_k \leq A_k \leq C\cdot A \cdot k$ for all $k\geq 1$; otherwise stated, the quantity
\begin{align*}
	\Pp{\sup_{k\geq 1} \frac{\cW_k}{k} \geq x}
\end{align*}
decays exponentially in $x$.
The analogous result where the supremum is taken on all the real numbers $t\geq 1$ follows from the monotonicity of $\cW$.
\end{proof}
\begin{proof}[Proof of Proposition~\ref{prop:properties of the continuous weight process}.\ref{it:sup of jumps}]
First, we use a union bound to write
\begin{align*}
	\Pp{Z_2\geq A} \leq   \Pp{Z_1 \leq A,\ Z_2\geq A} + \Pp{Z_1 > A}.
\end{align*}
Then 
\begin{align*}
\Pp{Z_1 \leq A,\ Z_2\geq A}&\leq \Pp{\sQ \cap \enstq{(t,m,y)}{t\geq 1, \ y \geq A (1\vee \log^2t) / t^2, \ m \leq A t}\neq \emptyset}\\
&= 1- \exp \left(- \int_{1}^{\infty} \dd t  \int_{A (1\vee \log^2t) / t^2}^{\infty} \dd y \int_{0}^{At} \dd m  \frac{\exp\left(-\frac{t^2}{2}y\right)}{y^\frac{3}{2}\sqrt{2\pi}} \right)\\
&\leq \int_{1}^{\infty} \dd t  \int_{A (1\vee \log^2t) / t^2}^{\infty} \dd y \int_{0}^{At} \dd m  \frac{\exp\left(-\frac{t^2}{2}y\right)}{y^\frac{3}{2}\sqrt{2\pi}}\\
&\leq \int_{1}^{\infty} \dd t \frac{A t }{\sqrt{2\pi}}  \exp\left(-\frac{A}{2}(1\vee \log^2 t) \right) \int_{A (1\vee \log^2t) / t^2}^{\infty}  \frac{1}{y^\frac{3}{2}} \dd y \\
&\leq  \int_{1}^{\infty} \dd t \frac{2 \sqrt{A}  t }{(1\vee \log t)\sqrt{2\pi}}  \exp\left(-\frac{A}{2}(1\vee \log^2 t) \right)\\
&\leq C \exp(- c A),
\end{align*}
for some constants $C, c>0$. 
Since the second term $\Pp{Z_1 > A}$ also decays exponentially in $A$ thanks to Proposition~\ref{prop:properties of the continuous weight process}.\ref{it:sup of W over t has exponential moments}, this concludes the proof.
\end{proof}
\begin{proof}[Proof of Proposition~\ref{prop:properties of the continuous weight process}.\ref{it:small jumps dont contribute}]
We write
\begin{align*}
	\Ec{\ind{Z_1\leq A} \sum_{t< s \leq 2t} w_s\ind{w_s\leq s^{-2-\delta}}}
	&\leq \Ec{M^\sQ(\intervalleof{t}{2t}\times \intervalleof{0}{2At} \times \intervalleff{0}{t^{-2-\delta}})}\\
	&\leq \int_{t}^{2t} \dd s \int_{0}^{2At} \dd m \int_{0}^{t^{-2-\delta}} y \frac{\exp(-\frac{s^2}{2}y)}{\sqrt{2\pi} y^{3/2}} \dd y\\
	&\leq \cst A t \int_{t}^{2t} t^{-1-\delta/2} \dd s\\
	&\leq \cst A t^{1-\delta/2}.
\end{align*}
Using Markov's inequality, summing over a sequence $t_n=2^n$ and using the Borel-Cantelli lemma we get that almost surely for any choice of $A$ we have   
\begin{align*}
	\ind{Z_1\leq A}\sum_{1< s \leq t} w_s\ind{w_s\leq s^{-2-\delta}} =O(t^{1-\delta/4})
\end{align*}
This proves the claim because the random variable $Z_1$ takes only finite values almost surely.
\end{proof}

\subsection{Martingale associated to the weight process}
We prove the following lemma, which directly ensures that Proposition~\ref{prop:properties of the continuous weight process}.\ref{it:cW over t converges as} holds.

\begin{lemma}\label{lem:continuous weight process divided by t converges}
The process $\left(\frac{\cW_t}{t} \right)_{t\geq 1}$ is a positive $L^2$ martingale. We have 
	\begin{align}
	\frac{\cW_t}{t} \underset{t\rightarrow \infty}{\rightarrow} Y,
	\end{align}
	almost surely and in $L^2$. 
	The limiting random variable $Y$ is positive almost surely. 
\end{lemma}

\begin{proof}[Proof of Lemma~\ref{lem:continuous weight process divided by t converges}]
The proof is made of two parts: a first and a second moment computation. 
 \\
\textbf{First moment computation.}
Using that the expression \eqref{eq:dynkin formula from stochastic calculus} with $F(x)=x$ is a martingale we get for any $t\geq s$,
\begin{align*}
\Ecsq{\cW_t - \cW_s}{\mathcal F_s}
= \Ecsq{\int_s^t \cW_{u-} \cdot \int_0 ^\infty v \cdot  \frac{\exp\left(-\frac{u^2}{2}v\right)}{v^\frac{3}{2}\sqrt{2\pi}}\dd v \dd u}{\mathcal F_s}
= \int_{s}^{t} \frac{\Ecsq{\cW_u}{\mathcal F_s}}{u} \dd u,
\end{align*}
where we used \eqref{eq:moments of u(t,y)} in the second equality. 
Taking derivatives, we get that $\frac{\dd }{\dd t} \Ecsq{\cW_t}{\mathcal F_s} = \frac{\Ecsq{\cW_t}{\mathcal F_s}}{t}$, so that
\begin{align*}
\frac{\mathrm d}{\mathrm d t}\left( \frac{\Ecsq{\cW_t}{\mathcal F_s}}{t} \right)= \frac{1}{t^2}\cdot\left(\frac{\Ecsq{\cW_t}{\mathcal F_s}}{t}\cdot t -\Ecsq{\cW_t}{\mathcal F_s}\right)=0.
\end{align*}
This ensures in that for any $t\geq s \geq 1$ we have $\Ecsq{\frac{\cW_t}{t}}{\cW_s}=\frac{\cW_s}{s}$, so the
the process $(\frac{\cW_t}{t})_{t\geq t}$ is a martingale.
Being a positive martingale, it converges almost surely to a limit $Y$.
\\
\textbf{Second moment computation.}
Using that the expression \eqref{eq:dynkin formula from stochastic calculus} with $F(x)=x^2$ is a martingale we get that 
\begin{align*}
\Ecsq{\cW_t^2 - \cW_s^2}{\mathcal F_s} &= \Ecsq{\int_{s}^{t} \dd u \int_0^\infty (2y \mathcal{W}_{u-} +y^2 ) (\cW_{u-}) \frac{\exp\left(-\frac{u^2}{2}y\right)}{y^\frac{3}{2}\sqrt{2\pi}}\mathrm d y}{\mathcal F_s} \\
&= \int_s^t \left(2\Ecsq{\cW_u^2}{\mathcal F_s} \cdot \frac{1}{u} + \Ecsq{\cW_u}{\mathcal F_s} \cdot \frac{1}{u^3} \right) \dd u
\end{align*}
so that $\frac{\dd}{\dd t} \Ecsq{\cW_t^2}{\mathcal F_s}=2\Ecsq{\cW_t^2}{\mathcal F_s} \cdot \frac{1}{t} + \Ecsq{\cW_t}{\mathcal F_s} \cdot \frac{1}{t^3}$ and 
\begin{align*}
\frac{\mathrm d }{\mathrm  d t}\Ecsq{\frac{\cW_t^2}{t^2}}{\mathcal F_s} 
&=\frac{1}{t^2}\left( 2\Ecsq{\cW_t^2}{\mathcal F_s} \cdot \frac{1}{t} + \Ecsq{\cW_t}{\mathcal F_s} \cdot \frac{1}{t^3}\right) - \frac{2}{t^3}\Ecsq{\cW_t^2}{\mathcal F_s} \\
&= \Ecsq{\cW_t}{\mathcal F_s} \cdot \frac{1}{t^5} \\
&= \frac{\cW_s}{s} \cdot \frac{1}{t^4}
\end{align*}
And so integrating between $s$ and $t$, we get
\begin{align*}
\Ecsq{\frac{\cW_t^2}{t^2}}{\mathcal F_s} =\frac{\cW_s^2}{s^2} + \frac{\cW_s}{s}\left(\frac{1}{3s^3}-\frac{1}{3t^3}\right) \underset{t\rightarrow\infty}{\rightarrow} \frac{\cW_s^2}{s^2} + \frac{\cW_s}{3s^4}.
\end{align*}
The process $\left(\frac{\cW_t}{t}, \ t\geq 1\right)$ is a positive martingale that is bounded in $L^2$ and hence converges in $L^2$ to the non-negative limit $Y$.
From what we computed above, we have
\begin{align*}
\Ecsq{Y}{\cW_t}=\frac{\cW_t}{t} \quad \text{ and } \quad \Varsq{Y}{\cW_t}=\frac{\cW_t}{3t^4}.
\end{align*} 
Then using the Chebychev inequality and the fact that $\cW_t$ is non-decreasing in $t$ almost surely we get 
\begin{align*}
\Ppsq{Y=0}{\cW_t}\leq \frac{\Varsq{Y}{\cW_t}}{\Ecsq{Y}{\cW_t}^2}\leq \frac{1}{3 t^3 \cW_t}
\end{align*}
which decreases to $0$ almost surely as $t\rightarrow \infty$.
Hence writing $\Pp{Y=0}=\Ec{\Ppsq{Y=0}{\cW_t}}$ and using monotone convergence we get that $\Pp{Y=0}=0$. 
\end{proof}
Note that it would be possible to do the same computation we $F(x)=x^k$ to obtain information about the value of higher moments of $Y$ recursively. 

%

\section{The root component of the WMSF of the PWIT and its scaling limit}\label{sec:scaling limit result}
In this section, we present in detail the construction of the object $\mathrm M^\infty$ that appears in Theorem~\ref{thm:local convergence louigi} as the local weak limit of $M_n$. 
We then interpret this construction in the framework of chains and aggregation processes, in order to prove the convergence of this object to $\mathcal M^\infty$, which is the content of Theorem~\ref{thm:convergence of the MST of the PWIT to Minfty}. 

\begin{table}[htbp]\caption{Table of Notation for Section~\ref{sec:scaling limit result}}
	\centering 
	\begin{tabular}{l c p{13cm} }
		\toprule
		\multicolumn{3}{c}{}\\
		\multicolumn{3}{c}{\underline{About Poisson--Galton--Watson trees}}\\
		\multicolumn{3}{c}{}\\
		$\PGW(\lambda)$ & : & Poisson--Galton--Watson distribution on trees, with parameter $\lambda$\\
		$\abs{\PGW(\lambda)}$ & : & by abuse of notation, a random variable with the distribution of the total number of vertices of a Poisson--Galton--Watson tree with parameter $\lambda$\\
		$\theta(\lambda)$ & : & the probability $\Pp{\abs{\PGW(\lambda)}=\infty}$\\
		$\lambda^*$ & : & the dual parameter associated to $\lambda>1$,  defined by the relation $\lambda e^{-\lambda}=\lambda^* e^{-\lambda^*}$\\
		\multicolumn{3}{c}{}\\
		\multicolumn{3}{c}{\underline{Objects related to the structure of  $\rM^\infty$}}\\
		\multicolumn{3}{c}{}\\
		$\alpha_1,\alpha_2,\dots$ & : & forward maximal weights in decreasing order \\
		$\sA$ &:& set of forward maximal weights\\
		$\rT_\alpha^\bullet$ & : & pond located just before the forward maximal edge of weight $\alpha$\\
		$\rM_\alpha^\bullet$ & : & obtained as the result of the Poisson--Galton--Watson Aggregation Process (PGWAP) started from $\rT_\alpha^\bullet$ \\
		$(W^\alpha(s))_{s \geq \alpha}$ & : & process counting the number of vertices during the PGWAP started from $\rT_\alpha^\bullet$\\
		$(W(s))_{s \geq \alpha}$ & : & Markov process with the same transitions as $W^\alpha$\\
		$X_\alpha$ & : & number of vertices of $\rT_\alpha^\bullet$ so that $W^{\alpha}(\alpha)=X_\alpha$\\
		$W^\alpha(\infty)$ &:& (finite) limiting value of  $(W^\alpha(s))_{s \geq \alpha}$, it is the number of vertices in $\rM_\alpha^\bullet$\\
		$\hat{T}$ & : & version of a tree $T$ endowed with a phantom root \\
		$\hat{\rM}^\bullet_\alpha$ & : & version of $\rM^\bullet_\alpha$ with an added phantom root\\
		$\widetilde{\rM}^\bullet_\alpha$ & : & version of $\hat{\rM}^\bullet_\alpha$ with normalized measure\\
		$\eta^{\mathrm{Poi}}$ & : & family $(\eta^{\mathrm{Poi}}(t,w), \ t>0, w>0)$ of distributions where a tree under $\eta^{\mathrm{Poi}}(t,w)$ is distributed as $\hat{T}$ where $T$ is a Poisson-Galton-Watson tree conditioned on having exactly $\lfloor w \rfloor$ vertices, endowed with a phantom root.\\
		$L_k$ & : & type of the $k$-th vertex along the spine of $\rT^\infty$\\
	
		\multicolumn{3}{c}{}\\
		\multicolumn{3}{c}{\underline{Scaled versions of the discrete objects and their corresponding limits}}\\
		\multicolumn{3}{c}{}\\
		\\
		$\sA^r$ &:& rescaled version $\enstq{r^{-1}\cdot(\alpha -1)}{\alpha \in \sA}$, tending to $\sP$\\
		$\sP$ & : & Poisson Point Process of intensity $\frac{\dd t}{t}$ on $\intervalleoo{0}{\infty}$\\
		$\left(	\mathsf{W}^{r}(t)\right)_{t\geq 1}$ & : & defined as $\left(r^2\cdot W^{1+r}(1+rt)\right)_{t\geq 1}$\\
		$(\cW_{t})_{t\geq 1}$ & : & weight process used in the construction of $\cM^\bullet$ \\
		$\mathbf{A}^r$ &: & generator of the process $\left((t,	\mathsf{W}^{r}(t))\right)_{t\geq 1}$\\
		$\mathbf{A}$ &: & generator of the process $\left((t,	\cW(t))\right)_{t\geq 1}$\\
		$\mathsf{W}^{r}(1)$ &: & starting value of $\left(	\mathsf{W}^{r}(t)\right)_{t\geq 1}$, also expressed as $r^{-2}\cdot X_{1+r}$\\
		$\cW_1$ &: & starting value of $\left(	\cW_t\right)_{t\geq 1}$, also expressed as $X$\\
		$r \cdot C_0^{-1} \cdot \mathsf{W}^{r}(\infty)$ & : & scaled version of the end-value of the process $\left(	\mathsf{W}^{r}(t)\right)_{t\geq 1}$ \\ 
		$Y$ & : & defined as $\lim_{t\rightarrow \infty} \frac{\cW_t}{t}$\\
		$\eta^{\mathrm{Poi},r}$ & : & defined from $\eta^{\mathrm{Poi}}$, where a tree sampled under $\eta^{\mathrm{Poi},r}(t,w)$ is distributed as $\mathrm{Scale}(r,r^2;\hat{T})$, where  $\hat{T}\sim \eta^{\mathrm{Poi}}(t,w)$.\\
		$\eta^{\mathrm{Br}}$ & : & family $(\eta^{\mathrm{Br}}(t,w), \ t>0, w>0)$ of distributions where a tree under $\eta^{\mathrm{Br}}(t,w)$ is a Brownian tree of mass $w$\\
		\bottomrule
	\end{tabular}
\end{table}

\subsection{Description of the root component of the WMSF of the PWIT}\label{subsec:description of the MST of the PWIT}
Recall the definition of $\mathrm M^\infty$ from Section~\ref{subsec:structure of root component of the WMSF of the PWIT} in the introduction.
In the tree $\mathrm M^\infty$, there is a unique infinite path starting from the root. 
For every vertex $v$ we can define its \emph{type} as the maximal weight of edges on the path between $v$ and $\infty$. 
This is always a number in the interval $\intervalleoo{1}{\infty}$.
In what follows, we describe the tree $\mathrm M^\infty$ as a rooted tree whose vertices have a type and forget about the weights of the edges and the planar structure.
First, we have to recall a few fact about Poisson--Galton--Watson trees and their probability of being infinite. 

\subsubsection{Around the probability of extinction of Poisson--Galton--Watson trees}
For any $\lambda\geq  1$ we write $\theta(\lambda)=\Pp{\abs{\PGW(\lambda)}=\infty}$. The dual parameter $\lambda^*$ is such that $\lambda e^{-\lambda}=\lambda^* e^{-\lambda^*}$. 
 We state a few useful facts about those quantities: 
\begin{itemize}
	\item The value $\theta=\theta(\lambda)$ is defined by the equation $1-\theta=e^{-\lambda \theta}$ and satisfies  $\theta(1)=0$ and $\theta(\lambda)\underset{\lambda \rightarrow \infty}{\rightarrow} 1$.
	\item For any $\lambda>1$ we have $\lambda^*=\lambda (1-\theta(\lambda))$.
	\item For any $\lambda>1$ we have
	\begin{align}\label{eq:theta prime tends to 2 at 1}
		\theta'(\lambda)=\frac{\theta(\lambda)(1-\theta(\lambda))}{1 - \lambda^*} \qquad \text{and} \qquad\theta'(\lambda) \underset{\lambda \rightarrow 1}{\rightarrow} 2.
	\end{align} 
	\item For $\epsilon>0$ we have 
	\begin{align}\label{eq:expansion (1+eps)^*}
			(1+\epsilon)^* =1-\epsilon+\frac{2}{3}\epsilon^2 + O(\epsilon^3).
	\end{align}
	\item For any $\lambda>1$ we have 
	\begin{align}\label{eq:relation theta lambda star}
		\frac{1-\theta(\lambda)}{1-\lambda^*} \leq \frac{1}{\lambda-1}.
	\end{align}
\item For any $\lambda>1$ we have 
\begin{align}\label{eq:relation theta lambda star second}
	\frac{1-\theta(\lambda)}{(1-\lambda^*)^3} \leq \frac{1}{(\lambda -1)^3}.
\end{align}
\end{itemize} 
The first four facts can be found in \cite{nachmias_wired_2024} and \cite[Corollary~3.19]{hofstad_random_2017}.
We only prove \eqref{eq:relation theta lambda star}, since \eqref{eq:relation theta lambda star second} is obtained is a very similar way.
\begin{proof}[Proof of \eqref{eq:relation theta lambda star}]
For $\theta=\theta(\lambda)$, since $1-\theta=e^{-\lambda \theta}$ we have $\lambda=-\frac{\log(1-\theta)}{\theta}$. 
Then since $\lambda^*=\lambda (1-\theta)$ we have $\lambda^*= -\frac{(1-\theta)\log(1-\theta)}{\theta}$.
Proving the inequality \eqref{eq:relation theta lambda star} then reduces to proving
\begin{align*}
	\frac{1-\theta}{1 + \frac{(1-\theta)\log(1-\theta)}{\theta}} \leq \frac{1}{-\frac{\log(1-\theta)}{\theta}-1},
\end{align*}
which in turns corresponds to 
\begin{align*}
	-2(1-\theta) \log(1-\theta) - 2\theta + \theta^2 \leq 0.
\end{align*}
Now we can just check that the right-hand-side is $0$ for $\theta=0$ and the derivative in $\theta$ is negative for $\theta \in \intervalleoo{0}{1}$. This concludes the proof.
\end{proof}

\subsubsection{Distribution of $(\rT^\infty,\rM^\infty)$.}
\label{subsubsec:distribution of T and M}
As explained in the introduction, we can consider the sequence $e_1,e_2,\dots$ of forward maximal edges in the invasion percolation process run from the root $\rho$ of the PWIT, and their corresponding weights $\alpha_1 > \alpha_2 >\dots$. 
Writing $\sA:=\{\alpha_1,\alpha_2,\dots \}$, if we disconnect $\rT^\infty$ (resp.\ $\rM^\infty)$ by removing all the edges $e_1,e_2,\dots$, we get a sequence of connected component that we call $\rT_\alpha^\bullet$ (resp. $\rM_\alpha^\bullet$), indexed by $\alpha\in \sA$. 
Those components are rooted at their lowest vertex $\rho_\alpha$ and pointed at their highest vertex $v_\alpha$.
Since the whole structure of $\rT^\infty$ (resp.\ $\rM^\infty)$ can be recovered from the data of $\sA$ and $(\rT_\alpha^\bullet)_{\alpha \in \sA}$ (resp. $(\rM_\alpha^\bullet)_{\alpha \in \sA}$), we just need to express their joint distribution. 
From \cite[Section~2.2]{addario-berry_local_2013}, we get that:
\begin{itemize}
	\item The sequence $(\alpha_k)_{k\geq 1}$ is a Markov process in $k$, where
	\begin{itemize}
		\item $\alpha_1$ has density $\theta'(y) \dd y$ on $\intervalleoo{1}{\infty}$,
		\item for any $k\geq 1$, conditionally on $\alpha_k$, the random variable  $\alpha_{k+1}$ has density $\frac{\theta'(y) }{\theta(\alpha_k)} \ind{y<\alpha_k} \dd y$.
	\end{itemize}
	\item Conditionally on $\sA$, the $((\rT_\alpha^\bullet, \rM_\alpha^\bullet), \ \alpha \in \sA)$ are independent and 
	\begin{itemize}
		\item The total number $X_\alpha$ of vertices of $\rT_\alpha^\bullet$ has the Borel--Tanner distribution with parameter $\alpha$ where for $m\geq1$,
		\begin{align*}
			\Pp{X_\alpha=m}=\frac{\theta(\alpha)}{\theta'(\alpha)} \frac{e^{-\alpha m}(\alpha m)^{m-1}}{(m-1)!},
		\end{align*}
	see \cite[Eq.~(2.2)]{addario-berry_local_2013}.
		\item Conditionally on $X_\alpha$, the tree $\rT_\alpha^\bullet$ has the distribution of a PGW tree conditioned on having $X_\alpha$ vertices\footnote{Note that the distribution of a tree under $\PGW(\lambda)$, conditioned on its number of vertices, does not depend on $\lambda$.}, pointed at a uniform random vertex.
		\item Conditionally on $\rT_\alpha^\bullet$, the tree $\rM_\alpha^\bullet$ is obtained by running the Poisson--Galton--Watson Aggregation Process from the seed  $\rT_\alpha^\bullet$.
	\end{itemize}
\end{itemize}

\paragraph{The Poisson--Galton--Watson Aggregation Process.}
We recall from \cite[Section~2.2]{addario-berry_local_2013} the definition of the Poisson--Galton--Watson Aggregation Process.
We define the process $(A(s))_{s\ge \alpha}$ started from some time $\alpha >1$ with some finite tree $A(\alpha)$.

At any time $s\geq \alpha$, for any vertex $v$ of $A(s)$, Poisson--Galton--Watson trees attach themselves at $v$ by having their root linked to $v$ by an edge;
at time $s$, attachments at every vertex occur at rate $(1-\theta(s))$.
A tree attaching to $v$ at time $\lambda$ is distributed as $\PGW(\lambda^*)$, independently of everything else.
This process stabilizes after a finite amount of time to some finite tree that we denote by $A(\infty)$.
In the above description, every $\rM^\bullet_\alpha$ is obtained from the corresponding $\mathrm{T}^\bullet_\alpha$ by running a PGW process started at time $\alpha$ from $A(\alpha)=\mathrm{T}^\bullet_\alpha$.

\paragraph{Interpretation in the $\mathrm{Aggreg}$ construction.}
Start the PGWA process at time $\alpha$ as defined above, from the seed tree $\mathrm{T}^\bullet_\alpha$.  
We introduce the weight process $(W^\alpha(s))_{s\geq \alpha}$ as
\begin{equation*}
	W^\alpha(s):= \#\{\text{vertices in the aggregate at time } s\}.
\end{equation*}
Note that the starting value $W^\alpha(\alpha)$ is equal to $X_\alpha$, the number of vertices in the tree $\mathrm{T}^\bullet_\alpha$.
Now for any $s$ we denote by $w_s=W^\alpha(s) -W^\alpha(s-)$ and we consider the jump times of this process $J^\alpha=\enstq{s\geq \alpha }{w_s \neq 0}$. 
Each of those jumps corresponds to the attachment of a tree to the aggregate. For any $s\in J^\alpha$, we call $T_s$ the corresponding tree. This tree is considered as rooted at the closest vertex to the seed graph $\mathrm{T}^\bullet_\alpha$.
The total number of vertices of that tree is given by the size of the jump $w_s$, and conditionally on $w_s$, its distribution is that of a $\PGW$-tree conditioned to have that number of vertices. 

Then, conditionally on $(W^\alpha(s))_{s\geq \alpha}$, and $\mathrm{T}^\bullet_\alpha$ and $(T_s, \ s\in J^\alpha)$, all these trees are glued on each other. 
From the definition of the process, a rooted tree $T_s$ corresponding to a jump at time $s$ is attached on a uniform vertex $x_s$ on the aggregate at time $s-$ by adding an edge between its root and $x_s$. 
In order to be in the context of the aggregation process of Section~\ref{subsec:presentation aggregation process}, we introduce a slight modification in the way we express this attachment procedure: for every rooted tree $T$, we define the tree $\hat{T}$ as the tree $T$ with the addition of a parent to the root of $T$, which is then the root of $\hat{T}$ (which we call the phantom root). 
When we consider those trees as endowed with the counting measure, the phantom root has zero mass. 
Whenever a gluing happens at time $s$, we can now express that as having the phantom root of the tree $\hat{T}_s$ identified to vertex $x_s$ uniformly chosen on the aggregate.

In the end, we can describe this process as an aggregation process in the sense of Section~\ref{subsec:presentation aggregation process}, with
\begin{itemize}
	\item  weight process $(W^\alpha(s))_{s\geq \alpha}$,
	\item  seed metric space $\mathrm{T}^\bullet_\alpha$,
	\item  block distributions $\eta^{\mathrm{Poi}}(t,w)$ where for any $t>1$, a tree under $\eta^{\mathrm{Poi}}(t,w)$ is distributed as a Poisson-Galton-Watson tree conditioned on having exactly $\lfloor w \rfloor$ vertices, endowed with a phantom root (and a single point with no mass if $\lfloor w \rfloor=0$).
\end{itemize}
Additionally, the process $(W^\alpha(s))_{s\geq \alpha}$ is Markovian and we can express its transitions explicitly. 
This will be done in more details in Section~\ref{subsubsec:study of the weight process}.
The process $(W^\alpha(s))_{s\geq \alpha}$ almost surely stabilizes to some value as $s\rightarrow \infty$ and we write $W^\alpha(\infty):=\lim_{s\rightarrow \infty} W^\alpha(s)$.
According to this description we can write $\hat{\rM}^\bullet_\alpha$ as the result of a PGWA process started from the seed tree $\hat{\rT}^\bullet_{\alpha}$ so that 
\begin{align}\label{eq:def N alpha as scaled version of tilde N alpha}
	\hat{\rM}^\bullet_\alpha := \mathrm{Scale} \left(1,W^\alpha(\infty);\widetilde{\rM}^\bullet_\alpha\right),
\end{align}
where
\begin{align}\label{eq:def tilde N alpha as aggregation process}
	\widetilde{\rM}^\bullet_\alpha := \mathrm{Aggreg}((W^\alpha(s))_{s\geq \alpha}\ , \hat{\mathrm{T}}^\bullet_{\alpha}\ , \eta^{\mathrm{Poi}}),
\end{align}
where the scaling part of the expression in \eqref{eq:def N alpha as scaled version of tilde N alpha} just appears from the fact that in our definition of the aggregation process, the measure carried by the aggregate is normalized to a probability measure. 

\subsubsection{Alternative description of the process}\label{subsubsec:alternative description as a branching process}
\paragraph{Another description of the distribution of $\rT_\alpha^\bullet$.}
Another way of constructing the tree $\mathrm{T}^\bullet_{\alpha}$ for a given $\alpha$ is to first construct a spine of length given by a $\mathrm{Geo}(\alpha^*)$-distributed random variable linking the root $\rho_\alpha$ to the marked point $v_\alpha$, and to let every vertex on the spine have a progeny off the spine distributed as independent $\mathrm{PGW}(\alpha^*)$ random trees. 
This is the content of \cite[Theorem~30, Theorem~31]{addario-berry_invasion_2012}.

\paragraph{The types along the spine.} 
The previous paragraph together with the fact that the sequence of weights $(\alpha_k)_{k\geq 1}$ of the forward maximal edges has an explicit description given in Section~\ref{subsubsec:distribution of T and M} allows to give an explicit description of the sequence $(L_k)_{k\geq 0}$ that describes the type of the vertex at height $k$ along the spine of the tree. 
This process admits a scaling limit, this is studied in Section~\ref{subsec:convergence of the set of weights of forward maximal edges}. 

\paragraph{Description as a branching process with types.}
We can describe the trees $\rT^\infty$ and $\rM^\infty$ using a branching process in discrete time.
This branching process considers individuals with types in $\intervalleoo{1}{\infty}$. 
Additionally, individuals can be normal or special and there are two types of reproduction events that happen simultaneously: giving birth to individuals of smaller or equal type, and adopting individuals of larger type. 
\begin{itemize}
	\item A normal individual of type $\alpha$
	\begin{itemize}
		\item gives birth to a Poisson number $\Poi(\alpha^*)$ of normal individuals of type $\alpha$,
		\item adopts normal individuals of larger type that are given by the points of a Poisson process with intensity $(1-\theta(y)) \ind{y> \alpha} \mathrm d y$
	\end{itemize}
	\item A special individual of type  $\alpha$
	\begin{itemize}
		\item gives birth to and adopts normal individuals with the same distribution as normal individuals do,
		\item gives birth to exactly one special individual. 
		This special child has the same type $\alpha$ as the parent with probability $\alpha^*$ and has a different type with probability $(1-\alpha^*)$. 
		In that case its type is taken in the interval $\intervalleoo{1}{\alpha}$ with a distribution that has a density given by $\mathbf 1_{\intervalleoo{1}{\alpha}}(y)\frac{\theta'(y)}{\theta(\alpha)}\mathrm d y$. 
	\end{itemize}
\end{itemize}
The tree $\mathrm M^\infty$, seen as a rooted tree with types, has the same distribution as the genealogy of this process started with one special individual with type taken under the distribution with density $\theta'(y)\mathrm{d} y$ on the interval $\intervalleoo{1}{\infty}$.
The tree $\mathrm T^\infty \subset \mathrm M^\infty$ corresponds to the tree obtained by removing the adoption mechanism in this process.

\subsection{Analogy between the discrete construction and the continuous one: an outline of the proof of Theorem~\ref{thm:convergence of the MST of the PWIT to Minfty}}\label{subsec:outline of proof scaling limit theorem}
From the description of $\mathrm{M}^\infty$, we have
\begin{align}\label{eq:hat T and hat M as a chain}
 \hat{\mathrm{M}}^\infty = \mathrm{Chain}\left(\hat{\rM}_\alpha^\bullet, \  \alpha \in \sA\right),
\end{align}
where we use to phantom root construction as in the previous section: for a discrete rooted tree $T$ (seen as a measured metric space), we define $\hat{T}$ as the tree $T$ to which we add a parent $\hat{\rho}$ of the root $\rho$ of $T$. The measure associated to $\hat{T}$ does not charge $\hat{\rho}$ and $\hat{T}$ is rooted at $\hat{\rho}$. 
We can easily check that the assumptions \ref{assum:chain1}, \ref{assum:chain2} and \ref{assum:chain3} hold here so that the construction \eqref{eq:hat T and hat M as a chain} makes sense.

We want to use the result of Proposition~\ref{prop:convergence of chains}, which allows to have the convergence of metric spaces defined as chains to spaces of the same type. 
For that we write
\begin{align*}
	\sA^r:= \enstq{r^{-1}\cdot (\alpha -1) }{\alpha \in \sA}
\end{align*}
so that 
\begin{align}
 \mathrm{Scale}\left(r,\frac{r^3}{C_0};\hat{\mathrm{M}}^\infty\right)&= \mathrm{Scale}\left(r,\frac{r^3}{C_0};\mathrm{Chain}\left(\hat{\rM}_\alpha^\bullet, \  \alpha \in \sA\right)\right)\notag\\
 &= \mathrm{Chain}\left(\mathrm{Scale}\left(r,\frac{r^3}{C_0}; \hat{\rM}_{1+tr}^\bullet\right), \  t \in \sA^r\right).
\end{align}
From this description, recalling the analogous definition of the continuous object
\begin{align*}
	\cM^\infty=\mathrm{Chain}\left(\cM_x^\bullet, \ x\in \mathscr{P} \right),
\end{align*}
we want to check that the assumptions of Proposition~\ref{prop:convergence of chains} are satisfied. 
For that, in order to check that \ref{assum:convergence chains 1} is satisfied we need to prove the following:
\begin{enumerate}[label=(\alph*)]
	\item\label{step:Er converges to P} The point process $\sA^r$ converges vaguely towards a Poisson point process with intensity measure $\frac{\dd t}{t}$ on $\intervalleoo{0}{\infty}$, this will be the content of Lemma~\ref{lem:values taken by chain converge to PPP}.
	\item We have the convergence \label{step:convergence of Nbullet} 
	\begin{align*}
		\mathrm{Scale}(r,C_0^{-1}\cdot r^3;\hat{\rM}_{1+r}^\bullet) \rightarrow \cM^\bullet
	\end{align*}
	 in distribution as $r\rightarrow 0$, this will be the content of Proposition~\ref{prop:convergence in distribution towards cN}.

	In order to prove this convergence, we recall the description \eqref{eq:def tilde N alpha as aggregation process} and rescale things so that 
	\begin{align*}
\mathrm{Scale}\left(r,1;\widetilde{\rM}^\bullet_{1+r} \right) = \mathrm{Aggreg}((r^{-2}\cdot W^{1+r}(1+rt))_{t\geq 1} \ , \mathrm{Scale}(r,r^2; \hat{\mathrm{T}}^\bullet_{1+r}) \ , \eta^{\mathrm{Poi}, r}).
	\end{align*}
where a tree sampled according to $\eta^{\mathrm{Poi}, r}(t,w)$ has the distribution of $\mathrm{Scale}\left(r,r^2;\hat{\mathrm{T}}_{\lfloor w r^{-2} \rfloor}\right)$  where $\mathrm{T}_{k}$ is a PGW tree conditioned to having exactly $k$ vertices. 
Recall the definition of $\widetilde\cM^\bullet$ as 
\begin{align*}
		\widetilde\cM^\bullet:=\mathrm{Aggreg}((\mathcal W_t)_{t\geq 1},\cT^\bullet, \eta^\mathrm{Br}). 
\end{align*}

Then the convergence of $\mathrm{Scale}\left(r,1;\widetilde{\rM}^\bullet_{1+r} \right)$ to $\widetilde\cM^\bullet$ consists in applying Proposition~\ref{prop:convergence of aggregation processes} for which three main ingredients are required:
\begin{itemize}
	\item convergence of the seed $\mathrm{Scale}(r,r^2; \hat{\mathrm{T}}^\bullet_{1+r}) \rightarrow \cT^\bullet$, this will be done in Section~\ref{subsec:convergence of the seed},
	\item convergence of the weight process $(r^{-2}\cdot W^{1+r}(1+rt))_{t\geq 1} \rightarrow (\cW_t)_{t\geq 1}$,
	\item convergence of $\eta^{\mathrm{Poi}, r}$ towards $\eta^\mathrm{Br}$, which follows from well-known scaling limit theorems, see \cite{aldous_continuum_1991}. 
\end{itemize}
The last step in order to get the convergence $\mathrm{Scale}(r,C_0^{-1}\cdot r^3;\hat{\rM}_{1+r}^\bullet) \rightarrow \cM^\bullet$ is to write $\mathrm{Scale}(r,C_0^{-1}\cdot r^3;\hat{\rM}_{1+r}^\bullet)= \mathrm{Scale}(r,C_0^{-1}\cdot r^3\cdot W^{1+r}(\infty);\widetilde{\rM}_{1+r}^\bullet)$ and just invoke the fact that $C_0^{-1}\cdot r^3\cdot W^{1+r}(\infty) \underset{r\rightarrow 0}{\rightarrow} Y = \lim_{t\rightarrow \infty} \frac{\cW_t}{t}$ in distribution, and that convergence takes place jointly with the previous ones. 
\item We also need to check that some technical assumptions hold at several steps of the way, namely \ref{assum:tight control on weight process} and \ref{assum:uniform control on tail} for the application of Proposition~\ref{prop:convergence of aggregation processes} in the convergence towards $\cM^\bullet$. This will be contained in Section~\ref{subsubsec:control on the discrete weight process}.
\item We also need to check that the  assumptions \ref{assum:convergence chains 2} and \ref{assum:convergence chains 3} indeed hold to be able to apply Proposition~\ref{prop:convergence of chains}. 
This will be done in Section~\ref{subsec:tightness near the root}.
\end{enumerate}
In the end, the proof of Theorem~\ref{thm:convergence of the MST of the PWIT to Minfty} uses everything above and is found in Section~\ref{subsec:proof of main theorem}.

\subsection{Convergence of the set of weights of forward maximal edges}\label{subsec:convergence of the set of weights of forward maximal edges}
The goal of Section~\ref{subsec:convergence of the set of weights of forward maximal edges} is to prove the following lemma.
\begin{lemma}\label{lem:values taken by chain converge to PPP}
	As $r\rightarrow 0$ we have
	\begin{align*}
		\sA^r \rightarrow \mathscr P,
	\end{align*}
	in distribution for the vague topology on $\intervalleoo{0}{\infty}$, where $\mathscr P$ is a Poisson process of intensity $\frac{\dd t}{t}$.
\end{lemma}
In order to prove the convergence stated in the lemma, we use the fact that the distribution of $\mathscr P$ is characterized by the following: for any $K>0$, if we enumerate $\mathscr P \cap \intervalleoo{0}{K}$ in decreasing order as $X_1>X_2>\dots $, then $(X_1,X_2,\dots)$ has the same distribution as 
\begin{align}\label{eq:characterization PPP as sequence}
	(KU_1,KU_1U_2,\dots),
\end{align}
where $(U_i)_{i\geq 1}$ is a sequence of i.i.d.\ uniform random variables on $\intervalleoo{0}{1}$. 

\begin{proof}[Proof of Lemma~\ref{lem:values taken by chain converge to PPP}]
First, recall the random variable $\alpha_1$ with density $\theta'(y)\ind{y>1}$.
Because of \eqref{eq:theta prime tends to 2 at 1}, we can get that for $x\rightarrow 1$, we have the following convergence in distribution 
\begin{align}\label{eq:convergence in law toward uniform}
	\mathrm{Law}\left(\frac{\alpha_1-1}{x-1} \ \bigg| \ \alpha_1 \leq x\right) \underset{x\rightarrow 1}{\rightarrow} \mathrm{Unif}\intervalleoo{0}{1}.
\end{align}
Now, note that the transition of the Markov chain $(\alpha_k)_{k\geq 1}$ is such that conditionally on $\{\alpha_k=x\}$, the term $\alpha_{k+1}$ has the same distribution as $\alpha_1$ conditionally on $\{\alpha_1 \leq x\}$. 
This entails that, for $T_x=\inf\enstq{k\geq 1}{\alpha_k\leq x}$\footnote{Note that we can easily see that $T_x$ is almost surely finite from the description above.}, then $\alpha_{T_x}$ has the distribution of $\alpha_1$ conditional on $\{\alpha_1\leq x\}$. Indeed, for any $1\leq y\leq x$, 
\begin{align*}
	\Pp{\alpha_{T_x}\leq y}&= \sum_{k=1}^{\infty} \Pp{T_x=k} \Ppsq{\alpha_k\leq y}{T_x=k}\\
	&=\sum_{k=1}^{\infty} \Pp{T_x=k} \Ppsq{\alpha_k\leq y}{\alpha_{k-1}>x,\alpha_{k}\leq x}\\
	&=\sum_{k=1}^{\infty} \Pp{T_x=k} \Ppsq{\alpha_1\leq y}{\alpha_1\leq x}=\Ppsq{\alpha_1\leq y}{\alpha_1<x}.
\end{align*}
Now we have $\sA=\{\alpha_1,\alpha_2,\dots\}$ and $\sA^r=r^{-1}\cdot \{\alpha_1-1,\alpha_2-1,\dots\}$. Fix some $K>0$. 
Applying the above for $x=1+rK$, the highest point in $\sA^r\cap \intervalleoo{0}{K}$ is $r^{-1}\cdot (\alpha_{T_{1+rK}}-1)$, 
and has the distribution $\mathrm{Law}\left( \frac{\alpha_1 - 1}{r} \bigg| \ \alpha_1 \leq 1+rK\right)$ 
so that thanks to \eqref{eq:convergence in law toward uniform}, we have $r^{-1}\cdot (\alpha_{T_{1+rK}}-1) \underset{r\rightarrow 0}{\rightarrow} K U_1$ 
with $U_1\sim \mathrm{Unif}\intervalleoo{0}{1}$. 
Using \eqref{eq:convergence in law toward uniform} iteratively, we can then easily get the finite dimensional convergence of $r^{-1}(\alpha_{T_{1+rK}}-1,\alpha_{T_{1+rK}+1}-1,\alpha_{T_{1+rK}+2}-2 \dots )$ towards $(KU_1,KU_1U_2,\dots)$ that appears in \eqref{eq:characterization PPP as sequence}. 
This concludes the proof of the lemma.
 \end{proof}

\paragraph{Connection with the lower envelope process.}
For completeness, in order to connect our current work to the approach of \cite{angel_invasion_2008}, we can also consider the process $(L_k)_{k\geq 0}$ of the types of vertices along the spine of $\rT^\infty$.
From the considerations of Section~\ref{subsubsec:alternative description as a branching process}, 
we get that $(L_k)_{k\geq 0}$ is a Markov process where $L_0$ has density $\theta'(y)$ on $\intervalleoo{1}{\infty}$ and the transitions are given by
\begin{align*}
	\Ppsq{L_{k+1}=L_k}{L_k=\alpha}&= \alpha^*,\\
	\Ppsq{L_{k+1}\in \mathrm{d}y}{L_k=\alpha}&= (1-\alpha^*)\frac{\theta'(y)}{\theta(\alpha)} \ind{y \leq \alpha }\dd y.
\end{align*}
Using \cite[Proposition~3.3]{angel_invasion_2008} which treats a very similar case, we get the following convergence in the Skorokhod topology for any positive $\epsilon$,
\begin{align}\label{eq:convergence towards the poisson lower envelope}
	(k(L_{\lceil k t \rceil}-1))_{t\geq \epsilon} \underset{k\rightarrow\infty}{\rightarrow} (\cL(t))_{t\geq \epsilon}, 
\end{align} 
where $(\cL(t))_{t>0}$ is called the Poisson Lower Envelope Process.
It is defined as such:
Let $\sR$ denote the Poisson point process on the positive quadrant with intensity 1. 
Then $\cL: \intervalleoo{0}{\infty}\rightarrow \intervalleoo{0}{\infty}$ is defined by
\begin{align*}
\cL(t):= \min \enstq{y>0}{(x,y)\in \sR\text{ for some } x \leq t}.
\end{align*}
The process $t\mapsto \cL(t)$ is positive and non-increasing and tends to $\infty$ at $0$ and to $0$ at $\infty$. This process is intimately linked with a Poisson Point Process of intensity $\frac{\dd t}{t}$, as the next lemma asserts. 
\begin{lemma}
The set of values $\sP := \enstq{\cL(t)}{t\in \intervalleoo{0}{\infty}}$, seen as a point process on $\intervalleoo{0}{\infty}$, is distributed as a Poisson point process with intensity measure $\frac{\mathrm{d}t}{t}$.
\end{lemma}
\begin{proof}
	This fact comes easily from the characterization \eqref{eq:characterization PPP as sequence}, the fact that $t\mapsto\cL(t)$ decreases from infinity at times $0$ and the fact that at very jump time $t>0$ of the process, the value $\cL(t)$ conditionally on the trajectory on the interval $\intervalleoo{0}{t}$, has the distribution of $\cL(t-) \cdot U$, where $U\sim \mathrm{Unif}\intervalleoo{0}{1}$. The details are left to the reader. 
\end{proof}

\subsection{Convergence of the seed $\mathrm{Scale}(r,r^2;\mathrm T^\bullet_{1+r})$}
\label{subsec:convergence of the seed}
From Section~\ref{subsubsec:distribution of T and M}, a way to sample $\mathrm{T}^\bullet_{\alpha}$ is to first sample a random variable $X_\alpha$ with the Borel--Tanner distribution with parameter $\alpha$, which is such that for all $m\geq 1$
\begin{align*}
	\Pp{X_\alpha=m}=\frac{\theta(\alpha)}{\theta'(\alpha)} \frac{e^{-\alpha m}(\alpha m)^{m-1}}{(m-1)!},
\end{align*}
and then let $\mathrm{T}^\bullet_{\alpha}$ be a PGW tree conditioned to have $X_\alpha$ vertices. 
The vertex $v_\alpha$ is then uniformly picked in $\mathrm{T}^\bullet_{\alpha}$. 
Using Sterling's approximation we have, for $m\rightarrow\infty$
\begin{align*}
	m!=\sqrt{2\pi m}\cdot \left(\frac{m}{e}\right)^m\cdot (1+o(1)).
\end{align*} 
Then, for any $x\in\intervalleoo{0}{\infty}$ and $r\rightarrow0$ we have
\begin{align*}
\Pp{X_{1+r}=\lfloor xr^{-2}\rfloor}&=\frac{\theta(1+r)}{\theta'(1+r)} \frac{e^{-r\lfloor xr^{-2}\rfloor}}{r^{-1}\sqrt{2\pi x}} (1+r)^{\lfloor xr^{-2}\rfloor}\cdot (1+o(1))\\
&=\frac{r^2}{\sqrt{2\pi x}}\cdot \exp\left(-\frac{x}{2}\right)\cdot (1+o(1))
\end{align*}
and this proves that 
\begin{align}\label{eq:convergence size of seed}
	r^2\cdot X_{1+r} \overset{(d)}{\underset{r \rightarrow 0}{\rightarrow}} X
\end{align}
where $X$ has density $\frac{\exp\left(-\frac{x}{2}\right)}{\sqrt{2\pi x}}\mathrm{d}x$ on $\intervalleoo{0}{\infty}$, which is the density of $\sqrt{2}\cdot G$ where $G$ is  $\mathrm{Gamma}(\frac{1}{2})$-distributed. In fact the above convergence also takes place in $L^2$ and $\Ec{X}=1$ and $\Var{X}=2$.

From there and from known results about the convergence of $\PGW$ trees conditioned on their size towards the Brownian tree, it is easy to check that we have the convergence 
\begin{align}\label{eq:convergence to Tbullet}
	\mathrm{Scale}(r,r^2;\mathrm T^\bullet_{1+r}) \underset{r\rightarrow 0}{\rightarrow} \cT^\bullet,
\end{align}
where $\cT^\bullet$ is a Brownian tree with mass $\sqrt{2}\cdot \mathrm{Gamma}(\frac{1}{2})$-distributed, pointed at a point sampled at random using a normalized version of its weight measure. 

\subsection{Properties and convergence of the discrete weight process}\label{subsec:properties and convergence of the discrete weight process}

\subsubsection{Definition of the weight process as a Markov process} \label{subsubsec:study of the weight process}

Recall the definition of the PGWA process started from at time $\alpha$ from some seed tree $S_\alpha$. 
The aggregate starts at $A(\alpha)=S_\alpha$ and at any time $s\geq \alpha$, growing event happen at each vertex of the aggregate at rate $(1-\theta(s))$.  
Whenever an growing event happens, a tree with distribution $\PGW(s^*)$ is attached to the corresponding vertex. 
From this description, it is easy to check that the weight process $(W(s))_{s\geq \alpha}$ defined as
\begin{equation*}
W(s):= \#\{\text{vertices in $A(s)$ at time } s\},
\end{equation*}
is then a time-inhomogeneous Markov process such that at time $s$, a jump happens with rate $W(s)\cdot(1-\theta(s))$ and the size of the corresponding jump is distributed as the number of vertices of a $\PGW(s^*)$-distributed tree. 
This entails that the process $((s,W(s)))_{s\geq \alpha}$ is a time-homogeneous Markov process with generator $A$ characterized by the fact that, given any compactly supported smooth function $f:\R_+\times \R_+ \rightarrow \R$  we have   
\begin{align}\label{eq:generator for weight of the PGW process}
	Af(s,x)
	&=\partial_s f + x\cdot(1-\theta(s))\cdot  \sum_{k=1}^{\infty}\left(f(s,x+k)-f(s,x)\right) \cdot \Pp{\abs{\PGW(s^*)}=k}.
\end{align}
Note that, by construction, this process satisfies the so-called \emph{branching property}: 
if we denote by $\left(W(s)\right)_{s\ge \alpha}$ and $\left(W'(s)\right)_{s\ge \alpha}$ two independent processes with the same transitions as above, started from respectively $k_1$ and $k_2$, then their sum $\left(W(s)+W'(s)\right)_{s\ge \alpha}$ would still be a Markov process with the same transitions, started at $k_1+k_2$.

In the rest of the section, we will use $(W(s))_{s\geq \alpha}$ to denote a general version of this Markov process started at time $\alpha$ and write $(W^\alpha(s))_{s\geq \alpha}$ for the version of this process started at time $\alpha$ from the value $X_\alpha$, the number of vertices of $\rT_\alpha^\bullet$, which is the version of the process that appears in the construction of $\rM^\bullet_\alpha$, see Section~\ref{subsubsec:distribution of T and M}.

\subsubsection{Moment computations}
In order to compute the moments, we will follow the same strategy as in the proof of Lemma~\ref{lem:continuous weight process divided by t converges},
and rely on the fact that for the functions $F$ given by $x\mapsto x$ and $x\mapsto x^2$, one can prove that 
\begin{align}\label{eq:dynkin formula discrete}
	F(W(t)) - F(W(\alpha)) - \int_{\alpha}^t W(s-) \cdot (1-\theta(s))\cdot  \sum_{k=1}^{\infty}\left(F(W(s-)+k)-F(W(s-))\right) \cdot \Pp{\abs{\PGW(s^*)}=k},
\end{align}
is a martingale,
in the same way as we prove the analogous statement \eqref{eq:dynkin formula from stochastic calculus} for process $(\cW_t)_{t\geq 1}$, by representing $(W(t))_{t\geq \alpha}$ as the solution of an SDE. 

\paragraph{Expectation of $W^{\alpha}(s)$.}
We follow the same steps as in the proof of Lemma~\ref{lem:continuous weight process divided by t converges} and use the fact that the process described in \eqref{eq:dynkin formula discrete} is a martingale.
We get, using the same steps, 
\begin{align*}
\frac{\mathrm d}{\mathrm d s} \Ecsq{W(s)}{W(\alpha)} &= \Ecsq{W(s)}{W(\alpha)} (1-\theta(s)) \Ec{\abs{\PGW(s^*)}}\\
&= \Ecsq{W(s)}{W(\alpha)} \frac{(1-\theta(s))}{1-s^*}. 
\end{align*}
Integrating the above differential equation yields
\begin{align}\label{eq:expression for expectation of volume PGW process}
	\Ecsq{W(s)}{W(\alpha)} = W(\alpha)\cdot \exp\left(\int_{\alpha}^s \frac{1-\theta(\lambda)}{1-\lambda^*} \mathrm d \lambda\right).
\end{align} 
Since the right-hand side tends to a finite limit as $s\rightarrow\infty$ we can write, denoting $W(\infty)$ as the limit of the increasing process $s\mapsto W(s)$ and using monotone convergence
\begin{align}
\Ecsq{W(\infty)}{W(\alpha)} = W(\alpha) \cdot \exp\left(\int_{\alpha}^\infty \frac{(1-\theta(\lambda))}{1-\lambda^*} \mathrm d \lambda\right).
\end{align} 
Assuming that $\alpha\leq 2$, (we are interested in what happens when $\alpha\rightarrow1$), we can write the second term in the last display as
\begin{align}\label{eq:expectation for final weight of PGWAP}
\exp\left( \int_{\alpha}^\infty \frac{(1-\theta(\lambda))}{1-\lambda^*} \mathrm d \lambda \right)
&=\exp\left( \int_{\alpha}^2 \frac{1}{\lambda -1} \mathrm d \lambda \right) \cdot  \exp\left( \int_{\alpha}^\infty \frac{(1-\theta(\lambda))}{1-\lambda^*} - \frac{\ind{\lambda\leq 2}}{\lambda -1} \mathrm d \lambda \right)\notag\\
&= \frac{1}{\alpha -1}\cdot  \exp\left( \int_{\alpha}^\infty \frac{(1-\theta(\lambda))}{1-\lambda^*} - \frac{\ind{\lambda\leq 2}}{\lambda -1} \mathrm d \lambda \right),
\end{align}
where the last integrand is integrable over $\intervallefo{1}{\infty}$ and so the second term in the above product is bounded above by the value
\begin{align}\label{eq:first appeareance C_0}
	C_0:=\exp \left(\int_{1}^\infty \frac{(1-\theta(\lambda))}{1-\lambda^*} - \frac{\ind{\lambda\leq 2}}{\lambda -1} \mathrm d \lambda\right),
\end{align}
and converges to that value when $\alpha\rightarrow1$.

\begin{remark}\label{rem:W(s)/(s-1) is a supermartingale}
	From the computations above, we also get that for any $s\geq \beta\geq \alpha$ we have
	\begin{align*}
			\Ecsq{\frac{W(s)}{s-1}}{W(\beta)} &= \frac{W(\beta)}{\beta-1} \cdot \frac{\beta-1}{s-1} \cdot \exp\left(\int_{\beta}^s \frac{1-\theta(\lambda)}{1-\lambda^*} \mathrm d \lambda\right)\\
			&= \frac{W(\beta)}{\beta-1} \cdot \exp\left(\int_{\beta}^s \frac{1-\theta(\lambda)}{1-\lambda^*} - \frac{1}{\lambda -1}\mathrm d \lambda\right).
	\end{align*}
Since from \eqref{eq:relation theta lambda star} the integrand in the last display is negative for any $\lambda>1$, this entails that $(\frac{W(s)}{s-1})_{s\geq \alpha}$ is a supermartingale in its own filtration. This will be useful later on.
\end{remark}

\paragraph{Second moment of $W(s)$.}
We apply the same method as above with $F(x)=x^2$. 
We use in the computation below the fact that for any parameter $\lambda\in \intervallefo{0}{1}$ the second moment for the size of a Poisson--Galton--Watson tree with parameter $\lambda$ is given by $\Ec{\abs{\PGW(\lambda)}^2}=\frac{1}{(1-\lambda)^3}$, see \cite[Eq. (4.1)]{abraham_note_2013}.
\begin{align*}
	&\frac{\mathrm d}{\mathrm d s} \Ecsq{\left(W(s)\right)^2}{W(\alpha)}\\
	&= 2 \Ecsq{\left(W(s)\right)^2}{W(\alpha)} (1-\theta(s)) \Ec{\abs{\PGW(s^*)}} + \Ecsq{W(s)}{W(\alpha)} (1-\theta(s)) \Ec{\abs{\PGW(s^*)}^2}\\
	&=2 \Ecsq{\left(W(s)\right)^2}{W(\alpha)} \frac{1-\theta(s)}{1-s^*}+ \Ecsq{W(s)}{W(\alpha)} \frac{1-\theta(s)}{(1-s^*)^3}.
\end{align*}
This linear differential equation is of the form $y' +p(s) y = f(s)$ with $p(s)= - 2\frac{1-\theta(s)}{1-s^*}$ and $f(s)=\Ecsq{W(s)}{W(\alpha)} \frac{1-\theta(s)}{(1-s^*)^3}$. 
We integrate the differential equation in $s$ and get
\begin{multline*}
	 \Ecsq{\left(W(s)\right)^2}{W(\alpha)} \\
	 = \exp\left(2 \int_{\alpha}^s \frac{(1-\theta(\lambda))}{1-\lambda^*} \mathrm d \lambda\right) \cdot \left((W(\alpha))^2+ \int_\alpha ^s 
	 e^{ - 2 \int_{\alpha}^\lambda \frac{(1-\theta(t))}{1-t^*} \mathrm d t } \Ecsq{W(\lambda)}{W(\alpha)} \frac{1-\theta(\lambda)}{(1-\lambda^*)^3}\ \dd \lambda\right)
\end{multline*}
and again, the above display has a limit when $s\rightarrow\infty$ so 
\begin{multline*}
\Ecsq{\left(W(\infty)\right)^2}{W(\alpha)} \\
=\exp\left(2 \int_{\alpha}^\infty \frac{(1-\theta(\lambda))}{1-\lambda^*} \mathrm d \lambda\right) \cdot  \left((W(\alpha))^2+\int_\alpha ^\infty e^{ - 2 \int_{\alpha}^\lambda \frac{(1-\theta(t))}{1-t^*} \mathrm d t }  \Ecsq{W(\lambda)}{W(\alpha)} \frac{1-\theta(\lambda)}{(1-\lambda^*)^3}\mathrm d \lambda\right).
\end{multline*}
Using the last display and \eqref{eq:expression for expectation of volume PGW process} we can then get
\begin{align}\label{eq:expression variance}
	\Varsq{W(\infty)}{W(\alpha)}&= \exp\left(2 \int_{\alpha}^\infty \frac{(1-\theta(\lambda))}{1-\lambda^*} \mathrm d \lambda\right) \cdot \int_\alpha^{\infty}  e^{ - 2 \int_{\alpha}^\lambda \frac{(1-\theta(t))}{1-t^*} \mathrm d t } \Ecsq{W(\lambda)}{W(\alpha)} \frac{1-\theta(\lambda)}{(1-\lambda^*)^3}\ \dd \lambda.
\end{align}
We now bound the RHS of the above equality, which is the product of two terms.
We first take care of the integral term and we bound each of the factors in the integrand:	
Using the equality \eqref{eq:expression for expectation of volume PGW process} and then applying the inequality \eqref{eq:relation theta lambda star}, we get that for all $s\geq \alpha$, 
\begin{align*}
\Ecsq{W(s)}{W(\alpha)} = W(\alpha)\cdot \exp\left(\int_{\alpha}^s \frac{1-\theta(\lambda)}{1-\lambda^*} \mathrm d \lambda\right) 
\leq W(\alpha)\cdot \exp\left(\int_{\alpha}^s \frac{1}{\lambda-1} \mathrm d \lambda\right)
= W(\alpha)\cdot \frac{s-1}{\alpha -1}.
\end{align*}
We can bound the two other factors using respectively that $e^{-x}\leq 1$ when $x\geq 0$ and the inequality \eqref{eq:relation theta lambda star second}.
Plugging those three bounds into the integral term of \eqref{eq:expression variance} we get 
\begin{align*}
	\int_\alpha ^\infty e^{ - 2 \int_{\alpha}^\lambda \frac{(1-\theta(t))}{1-t^*} \mathrm d t } \Ecsq{W(\lambda)}{W(\alpha)} \frac{1-\theta(\lambda)}{(1-\lambda^*)^3}\dd \lambda 
	\leq \int_\alpha ^\infty 1 \cdot  W(\alpha)\frac{\lambda-1}{\alpha -1}\cdot  \frac{1}{(\lambda-1)^3}\dd \lambda  \leq \frac{W(\alpha)}{(\alpha-1)^2}.
\end{align*}
For the first term of \eqref{eq:expression variance}, we just use \eqref{eq:expectation for final weight of PGWAP} and \eqref{eq:first appeareance C_0}, which together show that this term is bounded above by $\frac{(C_0)^2}{(\alpha-1)^2}$. 

In the end, putting everything together we have
\begin{align}\label{eq:variance bound for final weight of PGW}
\Varsq{W(\infty)}{W(\alpha)} \leq \frac{(C_0)^2W(\alpha)}{(\alpha -1)^4}.
\end{align}
 \subsubsection{Convergence of the weight process}\label{subsubsec:convergence of the weight process}
We introduce the following rescaled version of the process
\begin{align*}
\left(	\mathsf{W}^{r}(t)\right)_{t\geq 1}=\left(r^2\cdot W^{1+r}(1+rt)\right)_{t\geq 1}.
\end{align*}
We are going to compare this process to the  process $(\mathcal{W}_t)_{t\geq 1}$ and we recall that $Y$ is defined as the almost sure limit $Y:=\lim_{t\rightarrow\infty}  \frac{\cW_t}{t}$.
 \begin{proposition}\label{prop:joint convergence weight process and limit}
 We have the joint convergence in distribution
 	\begin{align*}
 	\left(\left(\mathsf{W}^{r}(t)\right)_{t\geq 1},r\cdot \mathsf{W}^{r}(\infty)\right) \underset{r\rightarrow 0 }{\longrightarrow} \left((\mathcal{W}_t)_{t\geq 1}, C_0\cdot Y\right),
 	\end{align*} 
 	where 
 	\begin{align}\label{eq:def C0}
 	C_0:=\exp\left( \int_{1}^\infty \frac{1-\theta(\lambda)}{1-\lambda^*} - \frac{\ind{\lambda\leq 2}}{\lambda -1} \mathrm d \lambda \right),
 	\end{align}
 and the convergence of the process is in the sense of the Skorokhod convergence on every bounded interval.
 \end{proposition}

\begin{proof}
The proof works in two steps. 
First we prove the convergence of $ \left(\left(t, \mathsf{W}^{r}(t)\right)\right)_{t\geq 1}$ towards $\left((t,\mathcal{W}_t)\right)_{t\geq 1}$ as Feller processes by proving the convergence of their generator and the convergence of their starting distribution.
Then in a second step we prove that this convergence holds jointly with that of $r\cdot \mathsf{W}^{r}(\infty)$.

\paragraph{Step 1: convergence of Feller processes via their generators.}
Fix a smooth function $f$ with compact support in $\intervallefo{1}{\infty}\times \R_+$. We first write down the generator $\mathbf A^r$ of the process $(t,\mathsf{W}^{r}(t))$ applied to the function $f$ as 
\begin{align}\label{eq:generator Ar}
	&{}\mathbf A^r f(t,x)\notag \\
	&=  \partial_t f + \frac{x}{r^2} \cdot r \cdot (1-\theta(1+rt)) \cdot \int_{u=1}^{\infty} (f(t,x+ r^2\lfloor u \rfloor)-f(t,x))\Pp{\abs{\PGW((1+rt)^*)}=\lfloor u \rfloor} \dd u  \notag \\
	&= \partial_t f + (1-O(rt)) \cdot x \cdot r^{-3} \cdot 
	\int_{y=r^2}^{\infty} (f(t, x+r^2\lfloor r^{-2}y \rfloor)-f(t, x))\Pp{\abs{\PGW((1+rt)^*)}=\lfloor r^{-2}y \rfloor} \dd y
\end{align}

For any $\lambda>0$ and $k\geq 1$ we have, using the cycle lemma,
\begin{align}\label{eq:prob that PGW has k vertices}
\Pp{\abs{\PGW(\lambda)}=k}=\frac{1}{k}\cdot \Pp{\sum_{i=1}^{k} \mathrm{Poi}(\lambda)=k-1}=\frac{1}{k}\cdot \Pp{ \mathrm{Poi}(\lambda k)=k-1}= \frac{e^{-\lambda k} (\lambda k)^{k -1}}{k!}.
\end{align}
For a bounded $t\geq 1$ and $r$ going to $0$, we have the expansion $(1+tr)^*=1-tr + \frac{2}{3}(tr)^2 + O((tr)^3)$. 
We plug in $k=\lfloor y r^{-2} \rfloor$ in the equality above and using Stirling's approximation and the asymptotic expansion of the logarithm near $1$ and get the following fact, the proof of which is postponed to the end of Section~\ref{subsubsec:convergence of the weight process}.
\begin{fact}\label{fact:probability PGW is large}
Let $T>0$. There exists constants $c,C>0$ such that for $r\rightarrow 0$ and uniformly for $t\in\intervalleff{1}{T}$, we have 
\begin{equation*}
	\Pp{\abs{\PGW((1+rt)^*)}=\lfloor r^{-2}y \rfloor} \left\lbrace
	\begin{aligned}
		&= \frac{r^{3}\exp\left(-\frac{t^2}{2} y\right)}{y^\frac{3}{2}\sqrt{2\pi}} (1+o(1)) \quad &\text{uniformly for } y\in \intervalleff{r^{\frac32}}{r^{-\frac12}}\\
		&\leq \frac{r^{3}\exp\left(-c \frac{t^2}{2}y+C\right)}{y^\frac{3}{2}\sqrt{2\pi}} \quad &\text{for } y\in \intervallefo{r^{-\frac12}}{\infty}
	\end{aligned}\right.
\end{equation*}
\end{fact}
Now, we can approximate the integral appearing on the right-hand-side of \eqref{eq:generator Ar} by splitting it into three pieces. First,
\begin{align*}
	&r^{-3}\cdot \left|\int_{y=r^2}^{r^{\frac{3}{2}}} (f(t, x+r^2\lfloor r^{-2}y \rfloor)-f(t, x))\Pp{\abs{\PGW((1+rt)^*)}=\lfloor r^{-2}y \rfloor} \dd y \right | \\
	&= r^{-1} \left|\int_{y=1}^{r^{-\frac{1}{2}}} (f(t, x+r^2\lfloor y \rfloor)-f(t, x))\Pp{\abs{\PGW((1+rt)^*)}=\lfloor y \rfloor} \dd y \right | \\
	 &\leq  r^{-1} \cdot \sup_{z \in \intervalleff{0}{r^{-1/2}}} |f(t, x+r^2 z)-f(t, x)| \cdot 1 \leq r^{\frac{1}{2}}  \norm{\partial_x f}_\infty 
\end{align*}
which tends to $0$ as $r\rightarrow 0$. 
Also we get 
\begin{align*}
	&r^{-3}\cdot \left|\int_{y=r^{-1/2}}^{\infty} (f(t, x+r^2\lfloor r^{-2}y \rfloor)-f(t, x))\Pp{\abs{\PGW((1+rt)^*)}=\lfloor r^{-2}y \rfloor} \dd y \right |\\
	&\leq 
	\int_{y=r^{-1/2}}^{\infty} \norm{\partial_x f}_\infty y  \frac{\exp\left(-c \frac{t^2}{2}y+C\right)}{y^\frac{3}{2}\sqrt{2\pi}} \dd y,
\end{align*}
which also tends to $0$ as $r\rightarrow 0$, uniformly in $t\in \intervalleff{1}{T}$.
Last, we have 
\begin{align*}
	r^{-3}\cdot \int_{y=r^{\frac{1}{2}}}^{r^{-\frac{1}{2}}} (f(t, x+r^2\lfloor r^{-2}y \rfloor)-f(t, x))\Pp{\abs{\PGW((1+rt)^*)}=\lfloor r^{-2}y \rfloor} \dd y\\ \underset{r\rightarrow 0}{\rightarrow}  \int_{0}^{\infty} (f(t, x+y)-f(t, x))\cdot \frac{\exp\left(-\frac{t^2}{2}y\right)}{y^\frac{3}{2}\sqrt{2\pi}}\dd y
\end{align*}
by dominated convergence, using again that for any $z\geq 0$ we have $|f(t, x+z)-f(t, x)| \leq \norm{\partial_x f}_\infty \cdot z$. 
Putting everything together we get that for any fixed function $f$ that is smooth and compactly supported we have the convergence
\begin{align*}
\mathbf A^r f(t,x)
&\underset{r\rightarrow 0}{\rightarrow} \partial_t f + x \cdot \int_{0}^{\infty} (f(t, x+y)-f(t, x))\cdot \frac{\exp\left(-\frac{t^2}{2}y\right)}{y^\frac{3}{2}\sqrt{2\pi}}\dd y
\end{align*}
and we recognize the limit as $\mathbf Af(t,x)$ where $\mathbf A$ is the generator of the process $((t,\mathcal{W}_t))_{t\geq 1}$ as described in \eqref{eq:definition generator weight process}.
Recall also the convergence $\mathsf{W}^{r}(1) = r^2\cdot X_{1+r} \underset{r \rightarrow0}{\rightarrow} X = \cW_1$ from Section~\ref{subsec:convergence of the seed}. 
Using a general result \cite[Theorem~17.25]{kallenberg_foundations_2002} for the convergence of Markov process via their generators, we get the convergence of $\left(\left(t,\mathsf{W}^{r}(t)\right)\right)_{t\geq 1}$ towards $\left((t,\mathcal{W}_t)\right)_{t\geq 1}$ for the Skorokhod topology on every bounded interval. 

\paragraph{Step 2: convergence of the limiting value.}
Fix some large $T$. 
Conditionally on the value $\mathsf{W}^{r}(T)$, the computation \eqref{eq:expectation for final weight of PGWAP} ensures that the limiting value $r \cdot \mathsf{W}^{r}(\infty)$ has conditional expectation given by
\begin{align*}
\Ecsq{r \cdot \mathsf{W}^{r}(\infty)}{\mathsf{W}^{r}(T)}&=\Ecsq{r^3\cdot  W^{1+r }(\infty)}{W^{1+r }(1+rT)} \\
&=r^3\cdot  \frac{ W^{1+r }(1+rT)}{rT}\cdot \exp\left( \int_{1+rT}^\infty \frac{(1-\theta(\lambda))}{1-\lambda^*} - \frac{\ind{\lambda\leq 2}}{\lambda -1} \mathrm d \lambda \right)\\
&=\frac{r^2\cdot W^{1+r }(1+rT)}{T}\cdot \exp\left( \int_{1+rT}^\infty \frac{(1-\theta(\lambda))}{1-\lambda^*} - \frac{\ind{\lambda\leq 2}}{\lambda -1} \mathrm d \lambda \right)\\
&\underset{r\rightarrow 0}{\rightarrow}  C_0 \cdot \frac{\cW_{T}}{T},
\end{align*}
where the convergence on the last line is in distribution, jointly with that of $\left(\mathsf{W}^{r}(t)\right)_{t\in \intervalleff{1}{T}}$ towards $\left(\mathcal{W}_t\right)_{t\in \intervalleff{1}{T}}$. 
The conditional variance of the same quantity can be bounded above using \eqref{eq:variance bound for final weight of PGW} as
\begin{align*}
	\Varsq{r \cdot \mathsf{W}^{r}(\infty)}{\mathsf{W}^{r}(T)}
	&= \Varsq{r^3 \cdot W^{1+r}(\infty)}{W^{1+r}(1+rT)} \\
	&\leq r^6 \cdot 
	(C_0)^2 \cdot  W^{1+r}(1+rT) \cdot (rT)^{-4}\\
	&\leq  \cst \cdot \frac{r^2 \cdot W^{1+r}(1+rT)}{T^4}\\
	&\leq \cst \cdot \frac{r^2 \cdot W^{1+r}(1+rT)}{T} \cdot \frac{1}{T^3}\\
	&\underset{r\rightarrow 0}{\rightarrow} \cst \frac{\cW_{T}}{T} \cdot \frac{1}{T^3},
\end{align*} 
again, in distribution, jointly with the previous convergences. 

As $T\rightarrow\infty$, we have $\frac{\cW_{T}}{T}\rightarrow Y$ and so the two last displays satisfy
\begin{align*}
	C_0 \cdot \frac{\cW_{T}}{T} \underset{T\rightarrow \infty}{\rightarrow} C_0 \cdot Y \quad \text{and} \quad \cst \cdot \frac{\cW_{T}}{T} \cdot \frac{1}{T^3} \underset{T\rightarrow \infty}{\rightarrow} 0.
\end{align*} 
This ensures that for large $T$ and small $r$, the value $r \cdot \mathsf{W}^{r}(\infty)$ is very concentrated around its conditional expectation at time $T$ which is close to its $r\rightarrow 0$ limit $C_0 \cdot \frac{\cW_{T}}{T}$, which in turn is close to its $T\rightarrow \infty$ limit $C_0 \cdot Y$. 
This concludes the proof.
\end{proof}
\begin{proof}[Proof of Fact~\ref{fact:probability PGW is large}]
We use the Stirling formula $k!\underset{k\rightarrow \infty}{=}\exp\left(k \log k - k + \log(\sqrt{2\pi k})+o_k(1)\right)$ to write
\begin{align*}
\Pp{\abs{\PGW(\lambda)}=k}&= \frac{e^{-\lambda k} (\lambda k)^{k -1}}{k!}\\
&
\underset{k\rightarrow \infty}{=}\exp \left(- \lambda k + (k-1)\log(\lambda k)-  k \log k + k - \log(\sqrt{2\pi k})+o_k(1)\right)\\
&\underset{k\rightarrow \infty}{=}\exp \left(- \lambda k + (k-1)\log(\lambda)- \log k + k - \log(\sqrt{2\pi k})+o_k(1)\right)\\
&\underset{k\rightarrow \infty}{=}\exp \left((\log \lambda -(\lambda-1)) k -\log(\lambda) -\log(k^{\frac{3}{2}}\sqrt{2\pi})+o_k(1)\right)
\end{align*}
Now, we  plug in $\lambda=(1+tr)^*\underset{r\rightarrow 0}{=}1-tr+\frac{2}{3}(tr)^2 + O((tr)^3)$, thanks to \eqref{eq:expansion (1+eps)^*}. 
We get that  
\begin{align*}
\log(\lambda)- (\lambda-1)&\underset{\lambda \rightarrow 1}{=} -\frac{1}{2} (\lambda -1)^2 + O((\lambda -1)^3)\\
&= -\frac{1}{2}(tr)^2+ O((tr)^3).
\end{align*}
Setting $k=k(x)=\lfloor y r^{-2}\rfloor$, for values $y\geq r^{\frac{3}{2}}$ (so that $k\rightarrow\infty$ as $r\rightarrow 0$) we get that $\lfloor y r^{-2}\rfloor=y r^{-2}(1+o(1))$ and so
\begin{align*}
	\Pp{\abs{\PGW((1+tr)^*)}=\lfloor y r^{-2}\rfloor} = \frac{r^3}{y^\frac{3}{2} \sqrt{2\pi}} \exp\left(-\frac{1}{2}t^2 y  + O(t^3 y r) + o(1)\right).
\end{align*}
The result follows from this asymptotics.
\end{proof}

\subsubsection{Control on the process.}\label{subsubsec:control on the discrete weight process}
We want bounds on the whole process $\left( \mathsf{W}^{r}(t)\right)_{t\geq 1}$ that hold at fixed $r$ namely that there exists two random variable $Z_{r,1}$ and $Z_{r,2}$ such that for all $t\geq 1$ we have 
\begin{align}\label{eq:control on the rescaled weight process}
\mathsf{W}^{r}(t)\leq Z_{r,1}\cdot t \quad \text{ and } \quad  (\mathsf{W}^{r}(t)-\mathsf{W}^{r}(t-))\leq Z_{r,2}\cdot \frac{1\vee (\log^2 t)}{t^2},
\end{align}
and such that $(Z_{r,1})_{0<r\leq 1}$ and $(Z_{r,2})_{0<r\leq 1}$ are tight.

For that, we first check that the processes $(\frac{ \mathsf{W}^{r}(t)}{t})_{t\ge 1}$ for all values of $r$ are positive supermartingales. 
This follows from Remark~\ref{rem:W(s)/(s-1) is a supermartingale}. 
Denote
\begin{align}
	Z_{r,1}:= \sup_{t\geq 1} \left(\frac{\mathsf{W}^{r}(t)}{t}\right) \quad \text{ and } \quad Z_{r,2}:= \sup_{t\geq 1} \left(\frac{\mathsf{W}^{r}(t)-\mathsf{W}^{r}(t-)}{{(1\vee \log^2 t)}/{t^2}}\right).
\end{align}
For what follows, we use the notation $M_t:=\frac{\mathsf{W}^{r}(t)}{t}$ to lighten the expressions.
First, for any $A>0$ we can set $\tau_A:=\inf\enstq{t\geq 1}{M_t \geq A}$ and write 
\begin{align*}
	\Ec{r^2\cdot X_{1+r}}=\Ec{M_{1}} \geq \Ec{M_{t\wedge \tau_A}} \geq A \cdot\Pp{\tau_A \leq t}.
\end{align*}
where we used the optional sampling theorem for supermartingales and the fact that $(M_t)$ only takes positive values.  
Then we have
\begin{align*}
	\Pp{Z_{r,1}\geq A}= \Pp{\tau_A <\infty}= \lim_{t\rightarrow \infty} \Pp{\tau_A \leq t} \leq \frac{\Ec{r^2\cdot X_{1+r}}}{A}\leq \frac{\cst}{A},
\end{align*}
because the convergence \eqref{eq:convergence size of seed} also takes place in $L^1$. 

For the second type of control, we use the same method as in Section~\ref{subsec:a priori estimates on the weight process} and assume that the process $W^{\alpha}$ was defined from some Poisson point process $\sS$ on $\intervalleoo{\alpha}{\infty}\times \R_+ \times \N_+ $ with intensity
\begin{align*}
\dd s \otimes 	\dd m \otimes \mathrm{Law}(\abs{\PGW(s^*)})
\end{align*}
in such a way that $W^{\alpha}(\alpha)=  X_\alpha $ and for $t\geq \alpha$,
\begin{align*}
	W^{\alpha}(t)=  X_\alpha + \sum_{(m,y,s)\in \sS} y \cdot \ind{m< W^{\alpha}(s-),\ s\leq t},
\end{align*}
analogously to how $(\cW_t)_{t\geq 1}$ was defined in Section~\ref{subsec:definition of the continuous weight process from PPP}.
Now reason on the event $\{Z_{1,r} \leq A\}$ and write
\begin{align*}
&\Pp{Z_{2,r}\geq A, Z_{1,r}\leq A}\\ &\leq \Pp{\Poi\left(A\int_{t=1}^{\infty}  t (1- \theta((1+rt)^*)) \dd t \int_{x=A \frac{\log^2 t}{t^2}}^\infty \Pp{|\PGW((1+rt)^*)| \geq r^{-2} x}\right) \geq 1}\\
&\leq \cst \cdot  A \cdot \exp(-A/2).
\end{align*}
This entails that $\Pp{Z_{r,2}\geq A} \leq \Pp{Z_{2,r}\geq A, Z_{1,r}\leq A} + \Pp{Z_{1,r}\geq A}$ and the right-hand-side tends to $0$ as $A\rightarrow \infty$ uniformly in $0< r\leq 1 $.
\subsection{Convergence of $\mathrm{Scale}(r,r^3;\rM^\bullet_{1+r})$}
\label{subsec:convergence of Nbullet}
We finally prove the following result. 
\begin{proposition}\label{prop:convergence in distribution towards cN}
	We have the convergence in distribution as $r\rightarrow 0$ in the Gromov--Hausdorff--Prokhorov topology
	\begin{align}\label{eq:convergence to Nbullet}
		\mathrm{Scale}\left(r,\frac{r^3}{C_0};\rM^\bullet_{1+r}\right)\rightarrow \cM^\bullet.
	\end{align}
	This convergence takes place jointly with \eqref{eq:convergence to Tbullet}.
\end{proposition}
In order to prove such a convergence, we will use Proposition~\ref{prop:convergence of aggregation processes} and the description of both $\mathrm{Scale}(r,r^3;\rM^\bullet_{1+r})$ and $\cM^\bullet$ as resulting from aggregation processes. 
Recall the definition of the objects $\cM^\bullet$ and $\widetilde{\cM}^\bullet$ from the beginning of Section~\ref{sec:the tree Minfty}.

\begin{proof}[Proof of Proposition~\ref{prop:convergence in distribution towards cN}]
In order to prove the proposition, we will check that the assumptions of Proposition~\ref{prop:convergence of aggregation processes} are satisfied in our setting where $\widetilde{\rM}_{1+r}^\bullet$ is defined as the result of an aggregation process with
\begin{itemize}
	\item weight process $\left(\mathsf{W}^{r}(t)\right)_{t\geq 1}$,
	\item seed given by $\mathrm{Scale}\left(r,r^2; \hat{\mathrm{T}}^\bullet_{1+r}\right)$,
	\item block distribution $\eta^{\mathrm{Poi},r}(t,w)$. 
\end{itemize}
The convergence $\left(\mathsf{W}^{r}(t)\right)_{t\geq 1}\underset{r\rightarrow 0}{\rightarrow} (\cW_t)_{t\geq 1}$ is ensured by Proposition~\ref{prop:joint convergence weight process and limit}. 
The convergence of the seed $\mathrm{Scale}\left(r,r^2; \mathrm{T}^\bullet_{1+r}\right)\underset{r\rightarrow 0}{\rightarrow} \cT^\bullet$ was obtained in \eqref{eq:convergence to Tbullet}. 
The convergence of the block distribution $\eta^{\mathrm{Poi},r}(t,w)$ to $\eta^{\mathrm{Br}}(t,w)$ the law of the Brownian tree of mass $w$ follows from classical results. 
Last, the two technical assumptions \ref{assum:tight control on weight process} and \ref{assum:uniform control on tail} are handled respectively by \eqref{eq:control on the rescaled weight process} and a general result of \cite{addario-berry_sub_2013} on tail bounds for the height of conditioned Galton--Watson trees. 

Thanks to the above, we can apply Proposition~\ref{prop:convergence of aggregation processes} to get the convergence 
\begin{align*}
	\mathrm{Scale}\left(r,1; \widetilde{\rM}_{1+r}^\bullet\right) \underset{r\rightarrow 0}{\rightarrow} \widetilde{\cM}^\bullet,
\end{align*}
where $\widetilde{\cM}^\bullet$ is defined as in the previous section from the weight process $(\mathcal{W}_t)_{t\geq 1}$, seed $\cT^\bullet$ and block distribution $\eta^{\mathrm{Br}}$.
Then using the joint convergence from Proposition~\ref{prop:joint convergence weight process and limit}, we get that the convergence in the last display takes place jointly with 
\begin{align*}
	r^3\cdot C_0^{-1} \cdot W^{1+r}(\infty) \rightarrow Y.
\end{align*}
In the end, since $\hat{\rM}_{1+r}^\bullet=\mathrm{Scale}\left(1,W^{1+r}(\infty); \widetilde{\rM}_{1+r}^\bullet\right)$ and $\cM^\bullet= \mathrm{Scale}\left( 1,Y;\widetilde{\cM}^\bullet\right)$ we get 
\begin{align*}
\mathrm{Scale}\left(r,\frac{r^3}{C_0}; \hat{\rM}_{1+r}^\bullet\right) \underset{r\rightarrow 0}{\rightarrow} \cM^\bullet.
\end{align*}
\end{proof} 

\subsection{Tightness near the root}\label{subsec:tightness near the root}
Recall from Section~\ref{subsec:description of the MST of the PWIT} the description of $\mathrm{T}^\infty$ and $\mathrm{M}^\infty$ from $(\mathrm{T}_\alpha^\bullet)_{\alpha\in \sA}$ and $(\rM_\alpha^\bullet)_{\alpha\in \sA}$. 
For any $\alpha>1$ introduce the following
\begin{align*}
	\mathrm{T}_{\geq \alpha}^\bullet = \bigcup_{\beta \in \sA, \beta \geq \alpha} \mathrm{T}_\beta^\bullet  \qquad \text{ and } \qquad \rM_{\geq \alpha}^\bullet  = \bigcup_{\beta \in \sA, \beta \geq \alpha} \rM_\beta^\bullet ,
\end{align*}
where the common marked point of $\mathrm{T}_{\geq \alpha}^\bullet$ and $\rM_{\geq \alpha}^\bullet$ is that of the lowest $\beta$ value, and the common root is $\emptyset$. 
We show the following:
\begin{lemma}\label{lem:tightness near the root}
For any $\epsilon>0$, we have the following 
\begin{enumerate}[label=(\roman*)]
	\item \label{it:tightness mass T 1+Ar} $\limsup_{A\rightarrow \infty} \limsup_{r\rightarrow 0}\Pp{r^2\cdot \mathrm{mass}(\mathrm{T}_{\geq 1+Ar}^\bullet)> \epsilon} = 0 $,
	\item \label{it:tightness diam T 1+Ar} $ \limsup_{A\rightarrow \infty} \limsup_{r\rightarrow 0} \Pp{r \cdot  \mathrm{diam}(\mathrm{T}_{\geq 1+Ar}^\bullet)> \epsilon}= 0$,
	\item\label{it:tightness mass N 1+Ar} $\limsup_{A\rightarrow \infty} \limsup_{r\rightarrow 0}\Pp{r^3\cdot \mathrm{mass}(\rM_{\geq 1+Ar}^\bullet)> \epsilon} = 0 $,
	\item\label{it:tightness diam N 1+Ar} $\limsup_{A\rightarrow \infty} \limsup_{r\rightarrow 0} \Pp{r \cdot  \mathrm{diam}(\rM_{\geq 1+Ar}^\bullet)> \epsilon}= 0$
\end{enumerate}
\end{lemma}

This corresponds to conditions \ref{assum:convergence chains 2} and \ref{assum:convergence chains 3} in the setting of the convergence of $\mathrm{T}^\infty$ (resp. $\mathrm{M}^\infty$) to $\cT^\infty$  (resp. $\cM^\infty$). 
\begin{proof}

First, recall that $L_k$ denotes the type of the vertex at height $k$ along the spine. We let $H_{r}:= \sup \enstq{k\geq 0}{L_k\geq 1 + r}$. 
Using the convergence \eqref{eq:convergence towards the poisson lower envelope} we can prove that 
\begin{align}\label{eq:convergence rHrA}
r\cdot H_{r} \underset{r\rightarrow 0}{\rightarrow} \tau:= \sup \enstq{ t \geq 0}{\cL(t)\geq 1},
\end{align}
and it is easy to check from the definition of $t\mapsto \cL(t)$ that $\tau\sim \mathrm{Exp}(1)$.
Then, in order to get the tree $\mathrm{T}_{\geq 1+r}^\bullet$ we need to consider the biological progeny off the spine of each of the $H_{r}$ vertices on the spine (here we use the language of  the branching process description of Section~\ref{subsubsec:alternative description as a branching process}).
Each of those trees is a $\PGW(\alpha^*)$ where $\alpha$ corresponds to its type. 
By definition, we have $\alpha \geq 1+r$ so that $\alpha^*\leq (1+r)^*$ for all types $\alpha$ of vertices present along the spine below height $H_r$.

In what follows, we consider another pair of trees $(\check{\rT}_{1+\frac{r}{2}}^\bullet, \check{\rM}_{1+\frac{r}{2}}^\bullet)$ that has the same distribution as $(\rT_{1+\frac{r}{2}}^\bullet, \rM_{1+\frac{r}{2}}^\bullet)$, and we couple their construction with that of $\rT_{\geq 1+r}^\bullet$ and $\rM_{\geq 1+r}^\bullet$ in such a way that with high probability we have 
\begin{align}\label{eq:coupling piece of tree close to the root}
	\mathrm{T}_{\geq 1+r}^\bullet\subseteq \check{\mathrm{T}}_{1+\frac{r}{2}}^\bullet \qquad \text{and} \qquad \rM_{\geq 1+r}^\bullet\subseteq \check{\rM}_{ 1+\frac{r}{2}}^\bullet.
\end{align}
For that, we denote by $\check{H_r}$ the distance between the root and the marked point in $\check{\mathrm{T}}_{1+\frac{r}{2}}^\bullet$ and remark that $r\cdot \check{H_r}\underset{r\rightarrow 0}{\rightarrow} \mathrm{Exp}(\frac{1}{2})$ in distribution. 
Then, conditionally on $\check{H_r}$, the tree  $\check{\mathrm{T}}_{1+\frac{r}{2}}^\bullet$ is obtained by considering the progeny off the spine of those vertices which is distributed as $\PGW((1+\frac{r}{2})^*)$. Here, note that $ (1+r)^*\leq(1+\frac{r}{2})^* $. 
From there it is clear that we can couple $H_r$ and $\check{H_r}$ on an event of high probability so that $H_r\leq \check{H_r}$ and, still on that event, couple the progeny off of the spine from vertices in one and the other tree to get $\mathrm{T}_{\geq 1+r}^\bullet\subseteq \check{\mathrm{T}}_{1+\frac{r}{2}}^\bullet$. 

Then, coupling the Poisson--Galton--Watson Aggregation Processes respectively started at time $\beta>1+r$ for each of the $\mathrm{T}_\beta^\bullet$ with $\beta \geq 1+r$ that constitute $\mathrm{T}_{\geq 1+r}^\bullet$ and started at the earlier time $1+\frac{r}{2}$ from  $\check{\mathrm{T}}_{1+\frac{r}{2}}^\bullet$ in the obvious way,  we obtain trees $\rM_{\geq 1+r}^\bullet\subseteq \check{\rM}_{ 1+\frac{r}{2}}^\bullet$, where the right-hand one has the distribution of $\rM_{1+\frac{r}{2}}^\bullet$.

From \eqref{eq:coupling piece of tree close to the root}, the result of the lemma is easily obtained from the convergences \eqref{eq:convergence to Tbullet} and \eqref{eq:convergence to Nbullet} for the mass and diameter of $\mathrm{T}_{1+r}^\bullet$ and $\rM_{1+r}^\bullet$.
\end{proof}
 \subsection{Convergence to $\mathcal{T}^\infty$ and to $\mathcal{M}^\infty$}
 \label{subsec:proof of main theorem}
 We now have everything to prove the main result of this section.
\begin{proof}[Proof of Theorem~\ref{thm:convergence of the MST of the PWIT to Minfty}]
Following the outline given in Section~\ref{subsec:outline of proof scaling limit theorem}, we can check that \ref{step:Er converges to P} is satisfied thanks to Lemma~\ref{lem:values taken by chain converge to PPP} and that \ref{step:convergence of Nbullet} is satisfied thanks to Proposition~\ref{prop:convergence in distribution towards cN}.
This ensures that assumption \ref{assum:convergence chains 1} is satisfied for  $\mathrm{Chain}\left(\mathrm{Scale}(r,r^3; \hat{\rM}_{1+tr}^\bullet), t \in \sA^r\right)=\hat{\mathrm{M}}^\infty$. 
Now, we just need to check \ref{assum:convergence chains 2} and \ref{assum:convergence chains 3}, which are ensured by Lemma~\ref{lem:tightness near the root}. 
Considering all this and  the definition of $\mathcal{M}^\infty$ as $\mathrm{Chain}\left(\cM_x^\bullet, x\in \mathscr{P} \right)$ in \eqref{eq:def cMinfinity} 
 we can now apply Proposition~\ref{prop:convergence of chains} to get the convergence in distribution in $\mathbb L$,
\begin{align*}
	\mathrm{Scale}(r,r^3;\hat{\mathrm{M}}^\infty) \underset{r\rightarrow 0}{\rightarrow} \mathcal{M}^\infty.
\end{align*}

Going through the exact same proof, it is easy to see that the same arguments can be made for $\mathrm{Chain}\left(\mathrm{Scale}(r,r^2; \hat{\mathrm{T}}_{1+tr}^\bullet), t \in \sA^r\right)$, which ensures that the second convergence holds as well
 \begin{align*}
 	\mathrm{Scale}(r,r^2;\hat{\mathrm{T}}^\infty) \underset{r\rightarrow 0}{\rightarrow} \mathcal{T}^\infty.
 \end{align*}
The fact that the two convergence take place jointly is clear from the fact that the convergence of each link of the chain takes place jointly, thanks to Proposition~\ref{prop:convergence in distribution towards cN}.
\end{proof}
\newpage

\appendix

\section{Chain construction}
\label{app:chain}

In this section, we prove results stated in Section~\ref{subsec:chain construction} about the $\mathrm{Chain}$ construction.
\begin{proof}[Proof of Lemma~\ref{lem:chain is well defined}]
Recall the framework defined before the statement of Lemma~\ref{lem:chain is well defined} and assume that conditions \ref{assum:chain1}, \ref{assum:chain2} and \ref{assum:chain3} are satisfied. 
The fact that $\mathrm{Chain}(\mathsf L_t^\bullet, t\in \mathscr{N})$ is a length space can be derived from the definition as a gluing of length spaces along a line.
Then, we just need to check that this tree is indeed in $\mathbb L$, meaning that its closed balls are compact and have finite mass.

Consider the closed ball of radius $R$ around $\rho_\infty$. 
Let us prove that this ball is pre-compact. 
First, let $A_R:=\sup\enstq{A\in \mathscr N}{\sum_{t\in \mathscr{N}\cap \intervalleoo{A}{\infty}} d_t(\rho_t,v_t)> R}$ where $\sup \emptyset$ is interpreted as the infimum of $\mathscr N$. 
Note that this set can only be empty whenever $\mathscr N$ has a smallest element, thanks to \ref{assum:chain3}, so that $A_R$ is a positive real number.  
This entails that the ball of radius $R$ is contained in the set $\bigcup_{t\in \mathscr N \cap (A_R,\infty]}L_t$, seen as a subset of $\mathrm{Chain}(\mathsf L_t^\bullet, t\in \mathscr{N})$.

Now let $\epsilon>0$ and construct an $\epsilon$-covering of that set. 
First, we place a ball with centre $\rho_\infty$. 
From \ref{assum:chain1} and \ref{assum:chain2}, we can find $A$ large enough so that $\sum_{t\in \mathscr{N}\cap \intervalleoo{A}{\infty}} d_t(\rho_t,v_t)<\frac{\epsilon}{2}$ and $\sup_{t\in \mathscr{N}\cap \intervalleoo{A}{\infty}} \diam(\mathsf L_t^\bullet) <\frac{\epsilon}{2}$.
For such an $A$ we have $\mathrm{Ball}(\rho_\infty,R)\setminus \mathrm{Ball}(\rho_\infty,\epsilon) \subset \bigcup_{t\in (A_R,A)\cap \mathscr N}L_t$. Now, since the set $\mathscr{N}$ is discrete, the set $(A_R,A)\cap \mathscr N$ only contains a finite number of points.
Since every one of the finite number of sets $L_t$ for $t\in (A_R,A)\cap \mathscr N$ is compact, we can cover them with a finite number of balls of radius $\epsilon$. 
In the end, we just constructed a covering of $\mathrm{Ball}(\rho_\infty,R)$ with balls of radius $\epsilon$. 
This closed ball is hence compact.
Now we just have to check that it has finite mass. This is ensured by the same argument, using this time the second requirement in \ref{assum:chain1}.
\end{proof}

Proof of Proposition~\ref{prop:convergence of chains}. 
Assume \ref{assum:convergence chains 1}.
Using Skorokhod's representation theorem, we can assume that we are working with a coupling that ensures that this convergence is almost sure. 
This entails that $((t,\mathsf L_t^{\bullet,r}))_{t\in \mathscr N^r}$ converges a.s.\ as $r\rightarrow0$ to $((t,\mathsf L_t^{\bullet,0}))_{t\in \mathscr N^0}$ as point processes on $\intervalleoo{0}{\infty}$ with marks in $\mathbb L^{\bullet,c}$. 
It follows that for any interval $(a,A)$ such that the endpoints $a,A$ are not in $\mathscr N^0$, for $r$ small enough, the number of elements in $\mathscr N^r\cap (a,A)$ stabilizes to some number $N$ which is the cardinal of $\mathscr N^0\cap (a,A)$. 
We denote by $u_1^r,\dots , u_N^r$ their elements in ascending order.
Again, from the convergence assumption, for any $1\leq i \leq N$ we have
\begin{align*}
	u_i^r \underset{r\rightarrow0}{\rightarrow}  u_i^0 \qquad \text{and} \qquad \mathsf L_{u_i^r}^{\bullet,r} \underset{r\rightarrow0}{\rightarrow}  \mathsf L_{u_i^0}^{\bullet,0} \quad \text{in $\mathbb L^{\bullet,c}$.}
\end{align*} 
It is easy to see that this shows that this proves that
\begin{align*}
 \mathrm{Chain}(\mathsf L_t^{\bullet,r}, t\in \mathscr{N}^r\cap (a,A)) \underset{r\rightarrow0}{\rightarrow} \mathrm{Chain}(\mathsf L_t^{\bullet,0}, t\in \mathscr{N}^0\cap (a,A)) 
\end{align*}
in $\mathbb L^c$ (see for example \cite[Lemma~2.4]{senizergues_geometry_2021}, in the context where the tree $\theta$ is just a line with $N$ vertices).
Now, using \ref{assum:chain1} and \ref{assum:chain2}, we have, for a fixed $a$,
\begin{align*}
	 \mathrm{Chain}(\mathsf L_t^{\bullet,0}, t\in \mathscr{N}^0\cap (a,A)) \underset{A\rightarrow\infty}{\rightarrow} \mathrm{Chain}(\mathsf L_t^{\bullet,0}, t\in \mathscr{N}^0\cap (a,\infty)).
\end{align*}
Combined with the previous display, this entails that we can find a deterministic sequence $(A_r)_r$ tending to infinity as $r\rightarrow0$ so that almost surely
\begin{align*}
	\mathrm{Chain}(\mathsf L_t^{\bullet,r}, t\in \mathscr{N}^r\cap (a,A_r)) \underset{r\rightarrow0}{\rightarrow} \mathrm{Chain}(\mathsf L_t^{\bullet,0}, t\in \mathscr{N}^0\cap (a,\infty)).
\end{align*}
Now, we can check that assumptions \ref{assum:convergence chains 2} and \ref{assum:convergence chains 3} ensure that 
\begin{align*}
	d_{\mathrm{GHP}}(\mathrm{Chain}(\mathsf L_t^{\bullet,r}, t\in \mathscr{N}^r\cap (a,A_r),\mathrm{Chain}(\mathsf L_t^{\bullet,r}, t\in \mathscr{N}^r\cap (a,\infty))))  \underset{r\rightarrow0}{\rightarrow} 0
\end{align*}
in probability so that 
\begin{align*}
	\mathrm{Chain}(\mathsf L_t^{\bullet,r}, t\in \mathscr{N}^r\cap (a,\infty))\underset{r\rightarrow0}{\rightarrow} \mathrm{Chain}(\mathsf L_t^{\bullet,0}, \  t\in \mathscr{N}^0\cap (a,\infty)),
\end{align*}
in probability.
Since we assumed that \ref{assum:chain3a} holds almost surely for $(\mathsf L_t^{\bullet,0}, \  t\in \mathscr{N}^0)$, for any fixed realization the ball of radius $r$ around the root of $\mathrm{Chain}(\mathsf L_t^{\bullet,0}, \  t\in \mathscr{N}^0)$ is contained in $\mathrm{Chain}(\mathsf L_t^{\bullet,0}, \  t\in \mathscr{N}^0\cap (a,\infty))$ for $a>0$ small enough. 
This ensures that we have the convergence 
\begin{align*}
	\mathrm{Chain}(\mathsf L_t^{\bullet,r}, t\in \mathscr{N}^r)\underset{r\rightarrow0}{\rightarrow} \mathrm{Chain}(\mathsf L_t^{\bullet,0}, \  t\in \mathscr{N}^0),
\end{align*}
in probability in $\mathbb L$, for this particular coupling. This ensures that in general, this convergence takes place in distribution, so that the proposition is proved.

\section{Aggregation processes}\label{sec:general version of the aggregation process}
\label{app:agg}

In this section, we introduce a general version of an \emph{aggregation process}.
We prove that under some reasonable assumptions that will be met in our applications, such a process yields a compact measured metric space in the limit and prove some properties of that limit (namely Hausdorff dimension, diffusivity of the measure).

This construction is very similar to the general line-breaking construction of Curien and Haas \cite{curien_random_2017}, except that we glue trees instead of segments, and that the times of gluing here can be dense instead of being discrete.
The ideas that we develop in this section are very similar to those already present in \cite{senizergues_random_2019} in a similar context of gluing random metric spaces together. 

In the next section, we prove the convergence of sequences of such process, under some reasonable assumptions, and this is the result that allows us to prove convergence results between different objects with this construction.   

\subsection{Construction and notation}
\subsubsection{Reminder of the aggregation process construction}
We recall the construction of an aggregation process from Section~\ref{subsec:presentation aggregation process}. We start with 
\begin{itemize}
	\item a \emph{weight process} $(W_t)_{t\geq t_0}$, which is a càdlàg increasing pure-jump process, such that $W_{t_0}>0$, (we denote by $\mathscr J$ its set of jump times),
	\item a \emph{seed} 
	$\mathsf S^\bullet$,
	which is a compact, rooted, pointed, measured length space (\emph{i.e.} an element of $\mathbb L^{\bullet,c}$) and has total mass equal to $W_{t_0}$. 
	\item a \emph{block-distribution family} $\eta=\left(\eta(t,w): \ t>t_0, w>0\right)$ of distributions on $\mathbb L^{c}$ that is such that for any $t,w$, an object $\mathsf{B}(t,w)$ sampled under the measure $\eta(t,w)$ almost surely has mass $w$. 
\end{itemize}
Note that in all this section, \textbf{we always consider $(W_t)_{t\geq t_0}$ as deterministic}. For constructions using a random $(W_t)_{t\geq t_0}$ we will first condition on $(W_t)_{t\geq t_0}$. 
Then, independently for any jump time $t\in \mathscr J$ of the weight process $(W_t)_{t\geq t_0}$ we have
\begin{itemize}
	\item A rooted metric space $\mathsf{B}_t$, with root $\rho_t$, carrying a mass measure $\nu_s$ of total mass $w_t=W_t-W_{t-}$, and which is sampled under the measure $\eta(t,w_t)$.
	\item $X_t$ a random point sampled on $\mathsf{S}^\bullet \sqcup \bigsqcup_{s\in\intervalleoo{t_0}{t}}\mathsf{B}_s$ according to a normalized version of $\nu_S+\sum_{t\in\intervalleoo{t_0}{t}} \nu_s$ which has finite total mass $W_{t-}$.
\end{itemize}
Then, the random metric space $\mathsf A_\infty^\bullet :=\mathrm{Aggreg}((W_t)_{t\geq t_0},\mathsf S^\bullet,\eta)$ that we construct is obtained, as a random pointed metric space, by identifying the pairs $(X_t,\rho_t)$ for all $t\in \mathscr{J}$, and keeping the root and marked point of $S^\bullet$. 
We endow this set with a probability measure $\mu_\infty$ defined as the following weak limit
\begin{align*}
	\mu_\infty:= \lim_{t\rightarrow \infty} \frac{1}{W_t}(\nu_S+\sum_{s\in\intervalleoo{t_0}{t}} \nu_s).
\end{align*}

\subsubsection{Construction coupled with a weighted recursive tree}
Let us be more precise here in the way that we define the points $(X_t, t\in \mathscr J)$: 
\begin{itemize}
	\item We let $(U_t, t\in \mathscr J)$ be independent uniform variables in $\intervalleff{0}{1}$,
	\item For every $t\in \mathscr J$ we let $K_t=\inf\enstq{s \leq t}{ W_s \geq U_t\cdot W_{t-}}$ so that $K_t$ takes values in $\{t_0\}\cup \mathscr J\cap \intervalleoo{t_0}{t}$. 
	Remark that for any $s\in \{t_0\}\cup \left(\mathscr J\cap \intervalleoo{t_0}{t}\right)$ we have, conditionally on the weight process,  $\Pp{K_t=s}=\frac{w_s}{W_{t-}}$.
	\item Then, conditionally on $K_t$, we let $X_t$ be sampled on $\sfB_{K_t}$ under a normalized version of the measure $\nu_{K_t}$, with $\sfB_{t_0}$ interpreted as the seed $\mathsf{S}$.
\end{itemize}
We can easily check that defined in the way described above, the points $X_t$ have the appropriate distribution. 
The introduction of the random variables $(K_t, t\in \mathscr J)$ allows us to define the following order relation $\prec$ on  $\{t_0\}\cup \mathscr J$ defined as the order relation generated by $s\prec t$ whenever $K_t=s$.

This relation encodes the tree structure along which the blocks $(\mathsf{B}_t,\ t \in \mathscr{J})$ are glued. 
A way to picture this structure is to imagine a discrete (locally infinite) tree $\ttT$ whose vertices are labelled by $\mathscr J$ and such for any $t\in \mathscr J$ the parent of the vertex labelled $t$ is $K_t$.
We can also define the corresponding increasing tree process $(\ttT_t)_{t\geq t_0}$ defined so that for any $t$, the tree $\ttT_t$ corresponds to the set of vertices of labels smaller than or equal to $t$.

From the construction, the process $(\ttT_t)_{t\geq t_0}$ is a form of weighted recursive tree as studied in \cite{mailler_random_2019} in the sense that for any $t\in \mathscr J$, the parent of vertex labelled $t$ is chosen among $\{t_0\}\cup \mathscr J \cap \intervalleoo{t_0}{t}$ with probability proportional to their weight, independently for all $t\in \mathscr J$.
Some properties of weighted recursive trees remain true in this general setting, like the following proposition.
\begin{proposition}\label{prop:height of random point is sum of bernoulli}
For any jump time $t\in \mathscr J$ the random variables
\begin{align*}
	(\ind{s\prec t}, s\in \mathscr J \cap \intervalleoo{t_0}{t})
\end{align*}
are independent Bernoulli random variables with respective parameter $\frac{w_s}{W_s}$.
%
\end{proposition}
\begin{proof}
	See \cite[Section~2.2]{curien_random_2017} or \cite[Corollary~8]{mailler_random_2019}. 
\end{proof}
\subsubsection{Organization of the section}
We recall the assumptions \ref{assum:weight process} and \ref{assum:blocks} which state that 
\begin{itemize}
	\item For some non-increasing function $\delta: \intervallefo{t_0}{\infty} \rightarrow\intervalleff{0}{1}$ that tends to $0$ as $t\rightarrow \infty$, for some random variables $Z_1$ and $Z_2$ we have for all $t\geq t_0$,
	\begin{align*}
		W_t\leq Z_1 \cdot t \quad \text{ and } \quad w_t\leq Z_2 \cdot t^{-2+\delta(t)}. 
	\end{align*}
\textbf{Note that here we see $(W_t)_{t\geq t_0}$ as deterministic so $Z_1$ and $Z_2$ are just constants.} Note that without loss of generality, we can always assume that $Z_1,Z_2\geq 1$, and we do so in the rest of this section.
	\item 
	We have
	\begin{align*}
		\sup_{t\geq t_0, w>0}\Pp{\frac{\diam \mathsf B(t,w)}{\sqrt{w}}\geq x}<C_1\cdot \exp\left(-C_2 x\right)
	\end{align*} 
\end{itemize}

The first goal of this section is to prove Proposition~\ref{prop:aggregate is well-defined}, which ensures that under the conditions above, the random measured metric space $\mathrm{Aggreg}((W_t)_{t\geq t_0},\mathsf S^\bullet,\eta)$ is well-defined as an element of $\mathbb{L}^{\bullet,c}$. 
Along the way, we will also provide some probabilistic bounds on the tail of the diameter of the obtained object that only depend on $Z_1$ and $Z_2$, in such a way that, whenever we work with a random weight process, it is possible to integrate those bounds with respect to $Z_1$ and $Z_2$, see Lemma~\ref{lem:diameter bounds of aggregate}. 
The way this is done is by first proving the almost sure compactness of the object (Lemma~\ref{lem:aggregate is compact + bound on hausdorff distance to seed} in Section~\ref{app:subsec:bounds on diameter and compactness}) and then proving that the probability measure that it carries is indeed almost surely well-defined, which is done in Proposition~\ref{prop:convergence to mu infty} of Section~\ref{app:subsec:endowing the aggregate with a probability measure}. 
The proof of Proposition~\ref{prop:aggregate is well-defined} can be found in Section~\ref{app:subsubsec:proof of prop aggregate is well-defined}.

The second goal is to prove Proposition~\ref{prop:dimension of aggregate is 3}, which under some additional conditions on the weight process $(W_t)_{t\geq t_0}$ and on the block distribution $\eta$, allows us to get its Hausdorff and Minkowski dimension, as well as some information about the measure $\mu_\infty$. This is done in Section~\ref{subsec:app:upper-bound on Minkowski dimension} and Section~\ref{subsec:app:lower-bound on the Hausdorff dimension}.

Third, we prove Proposition~\ref{prop:convergence of aggregation processes}, which ensures that under the appropriate conditions, a sequence of objects defined as the result of aggregation processes converge to an object constructed from a limiting version of the aggregation process. This is done in Section~\ref{subsec:app:convergence of aggregation processes}.

\subsection{Bounds on diameter and compactness of $\sfA_\infty$}
\label{app:subsec:bounds on diameter and compactness}
In this section, we prove that under assumptions \ref{assum:weight process} and \ref{assum:blocks} the object $\sfA_\infty^\bullet$ obtained from the aggregation process is almost surely compact and we provide some quantitative estimate on the tail of its diameter.
In this section, we omit the $\bullet$ superscript on pointed objects to lighten the notation.
\subsubsection{Bounding distance in the aggregate}
Recall the definition of $\sfA_\infty$ as the metric gluing
	\begin{align*}
		(A_\infty,\rho,v,d) :=\overline{\left(\mathsf S \sqcup \bigsqcup_{t> t_0} \mathsf B_t\right)/ \sim}
	\end{align*}
where every $\rho_t$ is identified with the corresponding point $X_t$. Then for any for any jump $s\in \mathscr{J}$ and any point $x\in \mathsf{B}_s$, by definition, the distance from $x$ to the seed $\mathsf{S}$ in $\sfA_\infty$ is bounded above by the sum of the contribution of all the block along the chain that links $\mathsf{B}_s$ to $\mathsf{S}$ in $\sfA_\infty$. Hence, recalling that $K_t$ is defined such that $X_t \in \mathsf{B}_{K_t}$,
\begin{align}\label{eq:distance in the aggregate is bounded above by sum of diameters}
	d(x,\mathsf{S})= d(x,\rho_s)+\sum_{t_0\prec u \prec s} d(X_u,\rho_{K_u}) &\leq \sum_{t_0\prec u\preceq s}\diam(\mathsf{B}_u).
\end{align}
The idea is then to use the probabilistic description of the set $\enstq{u}{t_0\prec u \prec s}$ and the tails of the diameter of the blocks $\mathsf{B}_u$ given by \ref{assum:blocks}.

For that, we first decompose the set of jump times $\mathscr J$ into
\begin{align*}
	\mathscr J= \mathscr E^* \sqcup \bigsqcup_{k=1}^\infty \mathscr E(2^{-k})
\end{align*}
where for any $\epsilon$ we set $\mathscr E(\epsilon):=\enstq{t>t_0}{\epsilon \le  w_t< 2\epsilon}$ and $\mathscr E^*=\enstq{t> t_0}{w_t \geq 1}$.
 
Now for $s\in \mathscr J$, for $x\in \mathsf B_s$, we decompose further the sum appearing on the right-hand-side of \eqref{eq:distance in the aggregate is bounded above by sum of diameters} according to the sizes of the different jumps involved in the sum
\begin{align}\label{eq:upper bound distances in the aggregate}
	d(x,\mathsf{S}) 
	&\leq \sum_{t_0\prec u\preceq s}\diam(\mathsf B_u)\notag\\
	&\leq \sum_{t_0\prec u\preceq s} \sum_{k=1}^\infty \diam(\mathsf B_u) \ind{u\in \mathscr E(2^{-k})} +\sum_{t_0\prec u\preceq s} \diam(\mathsf B_u) \ind{u\in \mathscr E^*}\notag\\
	&\leq \sum_{k=1}^\infty \underset{Y(k)}{\underbrace{\left(\max_{u\in \mathscr E(2^{-k})}\diam(\mathsf B_u)\right)}}  \underset{X(k)}{\underbrace{\left(\max_{r\in \mathscr{J}}\sum_{t_0\prec u\preceq r} \ind{u\in \mathscr E(2^{-k})}\right)}} + \underset{Z}{\underbrace{\sum_{u \in \mathscr E^*} \diam(\mathsf B_u)}} 
\end{align}
Remark that the last display does not depend on $x$ nor $s$ anymore, so that the Hausdorff distance $d_H(\sfA_{\infty},\mathsf{S})$ is smaller than the last display.
Now it just remains to prove that the right-hand-side is almost surely finite under our assumptions. 

The first part $X(k)$ will be handled using concentration inequalities about sums of independent Bernoulli random variables, thanks to Proposition~\ref{prop:height of random point is sum of bernoulli}.
For the second part $Y(k)$, we will use the tail behaviour for the diameter of the blocks that is ensured by \ref{assum:blocks}. 
Also, $Z$ is just a finite sum of random variables so it is almost surely finite. 

\subsubsection{Controls on the quantities $X(k)$ and $Y(k)$}
We prove here a technical result that contains bounds for the quantities $X(k)$ and $Y(k)$ that appear on the right-hand-side of \eqref{eq:upper bound distances in the aggregate}.
\begin{lemma}\label{lem:bounds on Xk and Yk}
Under assumptions \ref{assum:weight process} and \ref{assum:blocks}, there exists constants $c,C$ and independent random variables $X$ and $Y$ such that for all $k\geq 1$ we have
	\begin{align*}
		X(k) \le X \cdot (\log(Z_1 Z_2) +1) \cdot 7k\cdot \quad \text{ and } \quad Y(k) \le Y\cdot k \cdot  2^{-\frac{k}{2}} 
	\end{align*}
	and such that for $x\geq C$,
	\begin{align*}
		\Pp{X\geq x} \leq C \cdot Z_1 \cdot Z_2\cdot \exp\left(-cx^2 \right)
	\end{align*}
	and
	\begin{align*}
		\Pp{Y\geq x} \leq C \cdot Z_1 \cdot Z_2\cdot\exp\left(-c x \right). 
	\end{align*}
\end{lemma}
As a preliminary result, we prove that deterministically, for all $k\geq 0$ we have
\begin{align}\label{eq:upper bound number of blocks of given scale}
	|\mathscr E(2^{-k})| \leq Z_1\cdot Z_2 \cdot 2^{2k}.
\end{align}
\begin{proof}[Proof of \eqref{eq:upper bound number of blocks of given scale}]
Thanks to \ref{assum:weight process}, there exists $Z_1,Z_2$ and some non-increasing function $\delta:\intervallefo{t_0}{\infty}\rightarrow \intervalleff{0}{1}$ so that for all $t\geq t_0$,
\begin{align*}
	W_t\leq Z_1\cdot  t \quad \text{ and } \quad w_t\leq Z_2\cdot t^{-2+\delta(t)}.
\end{align*}
From the above display, if for any $\epsilon$ we define $t_\epsilon$ as
$t_\epsilon= Z_2\cdot \epsilon^{-1}$, then for any $t> t_\epsilon$ we have 
\begin{align*}
	w_t\leq Z_2\cdot t^{-2+\delta(t)} \leq Z_2 \cdot t_\epsilon^{-1} < \epsilon.
\end{align*}
This entails that $\mathscr E(\epsilon)\cap \intervalleoo{t_\epsilon}{\infty}=\emptyset$. 
Hence we have
\begin{align*}
	W_{t_\epsilon} \geq \sum_{t\in \mathscr E(\epsilon)} w_t \geq |\mathscr E(\epsilon)| \cdot \epsilon.
\end{align*}
And so we get
\begin{align*}
	|\mathscr E(\epsilon)| \leq  \epsilon^{-1} \cdot W_{t_\epsilon} \leq  \epsilon^{-1}\cdot Z_1 \cdot t_\epsilon \leq Z_1\cdot Z_2 \cdot \epsilon^{-2}.
\end{align*}
Applying the former to $\epsilon=2^{-k}$ concludes the proof of \eqref{eq:upper bound number of blocks of given scale}.
\end{proof}

Now we turn to the proof of Lemma~\ref{lem:bounds on Xk and Yk}.
\begin{proof}[Proof of Lemma~\ref{lem:bounds on Xk and Yk}]
We first study the random variables $(X(k),\ k\geq 1)$.  
For that we write, using a union bound,
\begin{align*}
\Pp{X(k)\geq x}&= \Pp{\max_{u \in \mathscr E(2^{-k})}\left(\sum_{v\preceq u} \ind{v\in \mathscr E(2^{-k})}\right)\geq x}\\
&\leq \sum_{u \in \mathscr E(2^{-k})} \Pp{\sum_{v\preceq u} \ind{v\in \mathscr E(2^{-k})}\geq x}
\end{align*}
But for any such $u$ in the sum
\begin{align*}
\Pp{\sum_{v\preceq u} \ind{v \in \mathscr E(2^{-k})}\geq x} = \Pp{ \sum_{v\in\mathscr E(2^{-k})} \ind{v\preceq u} \geq x},
\end{align*}
and all the $\ind{v\preceq u}$, for $v< u$ are independent Bernoulli r.v. with parameters $\frac{w_v}{W_v}$, thanks to Proposition~\ref{prop:height of random point is sum of bernoulli}.
The terms $\ind{v\preceq u}$ for $v>u$ are identically $0$ and the term $\ind{u\preceq u}$ is identically $1$, which we interpret as Bernoulli random variables of parameter $0$ and $1$ respectively.
We can bound the sum of the parameters of those Bernoulli random variables by
\begin{align*}
1+\sum_{v \in \mathscr E(2^{-k})} \frac{w_v}{W_v} \leq 1 + \sum_{v \in \mathscr E(2^{-k})} \frac{w_v}{\sum_{u \in  \mathscr E(2^{-k})}w_u \ind{u \leq v}}
&\leq 1+ 2 \sum_{i=1}^{|\mathscr E(2^{-k})|} \frac{1}{i} \\
&\leq 1+ 2 \left(1+ \log(|\mathscr E(2^{-k})|)\right)\\
&\leq 2\log(Z_1Z_2) + 4k\log 2+3\\
&\leq (\log(Z_1Z_2)+1)\cdot 7k,
\end{align*}
where we use the bound \eqref{eq:upper bound number of blocks of given scale}. 
Then using the last display together with some classical estimates for sums of independent Bernoulli r.v.\ we get for any $y>0$
\begin{align*}
\Pp{\sum_{v\in \mathscr E(2^{-k})} \mathrm{Ber}\left( \frac{w_v}{W_v} \right)\geq (1+y) (\log(Z_1Z_2)+1)\cdot 7k}
 &\leq \exp\left(-\frac{1}{3}y^2\cdot 7k \right)
\end{align*}  
so that for all $x>1$,
\begin{align*}
&\Pp{X(k)\geq ((\log(Z_1 Z_2) +1) \cdot 7k)\cdot x}\\
&\leq |\mathscr E(2^{-k})|\cdot  \exp\left(-\frac{1}{3}\cdot (x-1)^2 \cdot 7k \right) \\
&\underset{\eqref{eq:upper bound number of blocks of given scale}}{\leq} Z_1 Z_2 \cdot  \exp\left(\left(-\frac{1}{3}\cdot (x-1)^2 \cdot 7k \right)  + 2k \log 2\right).
\end{align*}
We can find constant $c_1,c_2,C>0$ so that the expression that appears in the exponential in the last display is bounded above for all $x\geq 1$ and all $k\geq 1$ by $(c_1-c_2\cdot x^2) k + C$.
In the end, for all $x> 1$,
\begin{align*}
	\Pp{X(k)\geq ((\log(Z_1 Z_2) +1) \cdot 7k)\cdot x}
	&\leq C\cdot Z_1\cdot Z_2 \cdot \exp((c_1-c_2 x^2) \cdot k).
\end{align*}
Introducing 
\begin{align*}
	X:=\sup_{k\geq 1} \left(\frac{X(k)}{(\log(Z_1 Z_2) +1) \cdot 7k}\right),
\end{align*}
we can use an union-bound and write, for all $x$ sufficiently large so that $(c_1-c_2 x^2)\leq -1$ we have
\begin{align*}
	\Pp{X\geq x} \leq \sum_{k=1}^\infty C\cdot Z_1\cdot Z_2 \cdot \exp((c_1-c_2 x^2) \cdot k) &= C\cdot Z_1\cdot Z_2 \cdot \frac{\exp(c_1-c_2 x^2)}{1 - \exp(c_1-c_2 x^2)} \\
	&\leq C\cdot Z_1\cdot Z_2 \cdot \exp(-c_2 x^2),
\end{align*}
by adequately changing the value of $C$ at the last inequality.

Now let us turn to the study of $(Y(k), \ k \geq 1)$.
Recall that the blocks $(\mathsf B_s)_{s \in \mathscr{J}}$  are random and independent, and that they are such that
\begin{align*}
\Pp{\frac{\diam \mathsf B_s}{\sqrt{w_s}}\geq x}\leq C_1 \exp\left(-C_2 x\right).
\end{align*} 
Then, using a union-bound, the Markov inequality and \eqref{eq:upper bound number of blocks of given scale} we get 
\begin{align*}
\Pp{Y(k)\geq x\cdot k\cdot  2^{\frac{k}{2}}}
&=\Pp{\max_{u\in \mathscr E(2^{-k})}\diam(\mathsf B_u)\geq x \cdot k\cdot  2^{-k/2}}\\
 &\leq |\mathscr E(2^{-k})| \cdot C_1 \exp\left(-C_2 x k\right)\\
& \leq Z_1 \cdot Z_2\cdot C_1 \cdot  \exp \left( (2\log 2 -C_2 x)k\right)
\end{align*}
 which is summable in $k$ if we take $x$ large enough (in a way that does not depend on $Z_1$ and $Z_2$).
 In the end, introducing 
 \begin{align*}
 	Y:=\sup_{k\geq 1} \left(\frac{Y(k)}{k\cdot 2^{\frac{k}{2}}}\right),
 \end{align*}
 using a union bound we get 
\begin{align*}
\Pp{Y\geq x\cdot k\cdot  2^{\frac{k}{2}}}\leq  C \cdot Z_1 \cdot Z_2 \cdot\exp(-c x),
\end{align*}
by adequately changing the values of the constants $c,C$.
This concludes the proof of the lemma.
\end{proof}

\subsubsection{Proof of compactness of $\sfA_\infty$}
We now prove the following lemma.  
\begin{lemma}\label{lem:aggregate is compact + bound on hausdorff distance to seed}
	Under \ref{assum:weight process} and \ref{assum:blocks}, the aggregation process (considered without measure) yields an almost surely compact metric space $\mathsf{A}_\infty$. 
\end{lemma}

\begin{proof}[Proof of Lemma~\ref{lem:aggregate is compact + bound on hausdorff distance to seed}]
First, under the assumption of the lemma we can apply Lemma~\ref{lem:bounds on Xk and Yk} to ensure that almost surely 
\begin{align*}
\sum_{k\geq 1}^{\infty} X(k) Y(k) \leq  7 XY(1+\log(Z_1 Z_2)) \sum_{k=1}^{\infty}k^2 2^{-\frac{k}{2}}<\infty.
\end{align*}
Recalling \eqref{eq:upper bound distances in the aggregate}, this already ensure that the metric space $\sfA_\infty$ is bounded. 
Since we want to prove the stronger statement that it is compact, we need to work a little bit more.
First introduce for any $k\geq 1$ the set 
	\begin{align}\label{eq:definition A^(k)}
		\sfA^{(k)}:= \mathsf S \cup \bigcup_{t\in \mathscr E^*}\sfB_t \cup \bigcup_{i=0}^k\bigcup_{t\in \mathscr E(2^{-i})}\sfB_t
	\end{align}
seen as a subset of $\sfA_\infty$.
Since it is a finite union of compact sets, it is itself compact as well. 

Then, from the same argument that led to \eqref{eq:upper bound distances in the aggregate}, we can show that for any $1\leq k\leq \ell$ we have the following upper-bound on the Hausdorff distance between the compact subsets $\sfA^{(k)}$ and $\sfA^{(\ell)}$ 
\begin{align}\label{eq:A^(k) is a Cauchy sequence}
	d_{\rmH}(\sfA^{(k)},\sfA^{(\ell)})\leq 2 \sum_{i\geq k+1}^{\infty} X(i) Y(i).
\end{align}
Now since the right-hand-side of the inequality almost surely tends to $0$ as $k\rightarrow \infty$, this ensures that the sequence $(\sfA^{(k)})_{k\geq 0}$ of compact subsets of the Polish space $\sfA_\infty$ is a Cauchy sequence for the Hausdorff distance, hence it converges. 
We can then identify $\sfA_\infty$ as its limit, which is then indeed compact. This proves the claim of the lemma.
\end{proof}
To finish, we state a more quantitative result. 
\begin{lemma}\label{lem:diameter bounds of aggregate}
	Under \ref{assum:weight process} and \ref{assum:blocks}, the aggregation process (considered without measure) is such that
	for any $p\geq 1$,
	\begin{align*}
		\Ec{d_{\rmH}(\sfA_{\infty},\mathsf{S})^p}^{\frac{1}{p}}\leq  C_p\cdot\left(Z_1 Z_2^{3/2} +350 (1+\log(Z_1 Z_2)) (Z_1 Z_2)^{\frac{1}{p}}\right),
	\end{align*}
	where $C_p$ is a constant that only depends on $p$ and on $\eta$ but not on $(W_t)_{t\geq t_0}$.
\end{lemma}
\begin{proof}
	Thanks \eqref{eq:upper bound distances in the aggregate} and the previous lemma we have, 
	\begin{align*}
		d_{\rmH}(\sfA_{\infty},\mathsf{S}) &\leq Z+  XY(1+\log(Z_1 Z_2)) \cdot 7 \sum_{k=1}^{\infty}k^22^{-\frac{k}{2}} \\
		&\leq Z+350 \cdot XY (1+\log(Z_1 Z_2)).
	\end{align*}
Using the triangular inequality for the $L^p$ norm we get 
\begin{align*}
	\Ec{d_{\rmH}(\sfA_{\infty},\mathsf{S})^p}^{\frac{1}{p}} &\leq \Ec{Z^p}^{\frac{1}{p}}+  350 (1+\log(Z_1 Z_2))\cdot  \Ec{\left(XY \right)^p}^{\frac{1}{p}}\\
	&\leq \Ec{Z^p}^{\frac{1}{p}}+  350 (1+\log(Z_1 Z_2))\cdot  \Ec{X^{2p}}^{\frac{1}{2p}} \cdot \Ec{Y^{2p}}^{\frac{1}{2p}}
\end{align*}
where we used Cauchy-Schwarz in the second inequality. 
Using the tails of $X$ and $Y$ that are so that for $x$ larger than some constant	
\begin{align*}
	\Pp{X\geq x} \leq C \cdot Z_1 \cdot Z_2\cdot \exp\left(-cx^2 \right)\quad \text{and} 
	 \quad\Pp{Y\geq x} \leq C \cdot Z_1 \cdot Z_2\cdot\exp\left(-c x \right),
\end{align*}
we can get by integrating that there exists a constant $C_p$ such that 
\begin{align*}
	\Ec{X^{2p}}^{\frac{1}{2p}}\leq C_p (Z_1 Z_2)^{\frac{1}{2p}} \qquad \text{and} \qquad \Ec{Y^{2p}}^{\frac{1}{2p}}\leq C_p (Z_1 Z_2)^{\frac{1}{2p}}.
\end{align*}

Now for the term involving $Z$, without any additional assumption, we can use some crude bounds
	\begin{align*}
		\sum_{u \in \mathscr E^*} \diam(\mathsf B_u)= \sum_{u \in \mathscr E^*} \sqrt{w_u} \cdot \frac{\diam(\mathsf B_u)}{\sqrt{w_u}} &\leq (\max_{u \in \mathscr E^*} \sqrt{w_u}) \cdot \sum_{u \in \mathscr E^*}\frac{\diam(\mathsf B_u)}{\sqrt{w_u}}.
	\end{align*}
	In the last display, we have $(\max_{u \in \mathscr E^*} \sqrt{w_u})\leq \sqrt{Z_2}$ and $ \sum_{u \in \mathscr E^*}\frac{\diam(\mathsf B_u)}{\sqrt{w_u}}$ is a sum of independent variables with a common exponential tail containing $|\mathscr E^*|\leq Z_1 Z_2$ terms. 
	This ensures that, up to increasing the value of $C_p$, we get 
	\begin{align*}
		\Ec{Z^p}^{\frac{1}{p}} \leq Z_1 Z_2^{\frac{3}{2}} \cdot C_p.
	\end{align*}
	Putting everything together concludes the proof of the lemma.
	
\end{proof}

\subsection{Endowing $\mathsf{A}_\infty$ with a probability measure: definition of $\mu_\infty$}
\label{app:subsec:endowing the aggregate with a probability measure}
Now that under assumptions \ref{assum:weight process} and \ref{assum:blocks}, Lemma~\ref{lem:aggregate is compact + bound on hausdorff distance to seed} ensures that the object $\mathsf A_\infty :=\mathrm{Aggreg}((W_t)_{t\geq t_0},\mathsf S,\eta)$ is almost surely well-defined as a compact random metric space, we want to endow it with some probability measure. 
Recall the definition for any $t\geq t_0$ of the measure $\mu_t$
\begin{align*}
	\mu_t:=\frac{1}{W_t}\sum_{t_0 \leq s \leq t} \nu_s,
\end{align*}
where as usual we abuse notation by seeing every one of the blocks $\sfB_s$ as a subset of $\sfA_\infty$. 
The measure process $(\mu_t)_{t\geq t_0}$ takes values in the set of probability measures on $\sfA_\infty$ which is compact. 
Endowing the set of probability measure on $\sfA_\infty$ with the so-called Lévy-Prokhorov distance yields a compact metric space that has the topology of weak convergence.

\paragraph{Elements of notation.}
In what follows, we will need some elements of notation. For $t\geq t_0$: 
\begin{itemize}
	\item We let $\sfA_t\subset \sfA_\infty$ be defined as 
	\begin{align*}
		\mathsf{S} \cup \bigcup_{s\in \intervalleof{t_0}{t}} \sfB_s,
	\end{align*}
the subset of $\sfA_\infty$ constituted by the blocks arrived at time $s\leq t$. 
	\item We let $p_{t}:\sfA_\infty \rightarrow \sfA_t$ be the projection towards $\sfA_t$, which maps any $x\in \sfA_\infty$ to the point $p_t(x)$ which makes the distance $d(x,p_t(x))$ minimal. 
	\item We let $\sD(t)=\enstq{s\in \mathscr J}{s\succeq t}$, be the set of descendants of $t$, (recall the definition of $\preceq$ from the beginning of the section).
	\item We let $\sfD(t)= \overline{\bigcup_{s \in\sD(t)} \mathsf B_s}$ be the corresponding subset of $\sfA_\infty$.
	\item In a similar manner, for $\mathcal C\subset \sfA_t$ we let $\sD(t,\mathcal C)=\enstq{s\in \mathscr J}{p_{t}(\rho_s)\in \mathcal C}$. 
	\item Similarly, define $\sfD(t,\mathcal C)= \mathcal C \cup \overline{\bigcup_{s \in\sD(t,\mathcal C)} \mathsf B_s}$.
\end{itemize}

\begin{proposition}\label{prop:convergence to mu infty}
	The process of measures $(\mu_t)_{t\geq t_0}$ on $\sfA_\infty$ converges weakly to a limit $\mu_\infty$. 
	Furthermore, a random point $Y_\infty$ that has distribution $\mu_\infty$ conditionally on $\sfA_\infty$ is such that the events
	\begin{align*}
		\{Y_\infty \in \sfD(t)\} \quad \text{ for } t\in \mathscr J,
	\end{align*}
	are independent with respective probability $\frac{w_t}{W_t}$.
\end{proposition}
The proof of Proposition~\ref{prop:convergence to mu infty} can be found in Section~\ref{subsubsec:proof of convergence measure AG}.
The methods of proof of this section rely on adaptions of arguments already present in \cite[Theorem~2.5]{curien_random_2017} and \cite[Lemma 5 and Section~4]{senizergues_random_2019}. 
\subsubsection{Generalized P\'olya urns with weight process $(W_t)_{t\geq t_0}$}
We start by describing an urn process, a slightly more general version of the one studied by Pemantle in \cite{pemantle_time_1990}. 
Let $a, b$ be two non-negative real numbers, with $a + b > 0$ ,
and $t_0$ be some real number and $(W_s)_{s\ge t_0}$ a non-increasing càdlàg pure jump process, whose jumps we call $(w_s)_{s\in \mathscr J}$ where $\mathscr J$ is the set of jump times of the process $s\mapsto W_s$. 

We refer to the
following process as a time-dependent Pólya urn starting at time $t_0$ with $a$ red balls and $b$ black balls and weight process $(W_s)_{s\ge t_0}$:
\begin{itemize}
	\item  At time $t_0$, the urn contains $a$ red balls and $b$ black balls with $a+b=W_{t_0}$,
	\item Then at every jump time $s$ of the process, a ball is drawn at random and replaced in the urn, along with $w_s$ additional balls of the same colour.
	For any $s\geq t_0$ we call $R_s$ the proportion of red balls in the urn at time $s$.
\end{itemize}
From that description, it is unclear if the definition of the process makes sense because we could draw infinitely many times in the urn in every finite interval of time. 
However, we can define a version of this whole process using a family $(V_r)_{r\in \mathscr{J}}$ i.i.d. random variables that are uniform on $\intervalleff{0}{1}$:
\begin{align}\label{eq:definition generalized polya urn martingale}
R_t = R_s + \sum_{t< r \leq s} \frac{w_r}{W_r}\cdot \left(\ind{V_r \leq R_{r-}}-R_{r-}\right). 
\end{align} 
In particular, with this definition we deterministically have for any $s \le t$,
\begin{align*}
|R_s -R_t| \leq \sum_{t< r \leq s} \frac{w_r}{W_r}\leq \frac{1}{W_t}\cdot (W_s-W_t)
\end{align*}
which ensures that the function that we define is almost surely càdlàg. From \eqref{eq:definition generalized polya urn martingale}, it is also immediate that this process is a martingale in its own filtration. Because its values are in $\intervalleff{0}{1}$ we automatically have the convergence of $R_s$ to some random variable $R_\infty$ almost surely and in all $L^p$ for $p\geq 1$.

From \cite[Theorem~4]{pemantle_time_1990}, easily adapted to this setting, if the process $(R_{t})_{t\geq t_0}$ starts with a positive value $R_{t_0}>0$, then we almost surely have
\begin{align}\label{eq:limiting value of pemantle urn is as positive}
	R_\infty>0.
\end{align}

\paragraph{Bounding the moments.}
In order to get a bound on the moments of $R_\infty$, we use the following reasoning, which is the same as the one used in the proof of \cite[Lemma~5]{senizergues_random_2019}.
For any jump time $s$ and integer $N\geq 2$ we have
\begin{align*}
\Ec{(R_s)^N}&= \Ec{\left(R_{s-}+ \frac{w_s}{W_s}\cdot \left(\ind{V_s \leq R_{s-}}-R_{s-}\right)\right)^N}\\*
&=\Ec{\left(\frac{W_{s-}}{W_s}R_{s-}+ \frac{w_s}{W_s}\ind{V_s \leq R_{s-}}\right)^N}\\
&\leq \left(\left(\frac{W_{s-}}{W_s}\right)^N + N \left(\frac{W_{s-}}{W_s}\right)^{N-1} \frac{w_s}{W_s} \right)\cdot  \Ec{(R_{s-})^N} + \sum_{\ell =0}^{N-2} \binom{N}{\ell}\Ec{(R_{s-})^{\ell +1}} \left(\frac{w_s}{W_s}\right)^{k-\ell}\\
&\leq \Ec{(R_{s-})^N} + \sum_{\ell =0}^{N-2} \binom{N}{\ell}\Ec{(R_{s-})^{\ell +1}} \left(\frac{w_s}{W_s}\right)^{k-\ell}.
\end{align*}
Hence, we write
\begin{align}\label{eq:recursion on moment for mass of red balls}
\Ec{(R_t)^N}&= \Ec{(R_{t_0})^N}+ \sum_{t_0< s \leq t} \left(\Ec{(R_s)^N}-\Ec{(R_{s-})^N}\right)\notag \\
&\leq \Ec{(R_{t_0})^N}+ \sum_{t_0< s \leq t}  \sum_{\ell =0}^{N-2} \binom{N}{\ell}\Ec{(R_{s-})^{\ell +1}} \left(\frac{w_s}{W_s}\right)^{k-\ell} \notag \\
&  \leq  \Ec{(R_{t_0})^N}+ \sum_{\ell =0}^{N-2} \binom{N}{\ell} \left(\sup_{s\geq t_0} \Ec{(R_{s-})^{\ell +1}}\right) \left(\sum_{s \geq t_0}  \left(\frac{w_s}{W_s}\right)^{N-\ell}\right).
\end{align}
This entails that we can have bounds on the moments of $R_\infty$ recursively. 
A useful inequality for those martingales is the following:
\begin{align}\label{eq:bound on variance of limiting mass}
	\Varsq{R_\infty}{R_{t}}= \Varsq{\sum_{r> t} \frac{w_r}{W_r}\cdot \left(\ind{V_r \leq R_{r-}}-R_{r-}\right)}{R_{t}}\leq R_{t}\cdot \sum_{r>t} \left(\frac{w_r}{W_r}\right)^2.
\end{align}

\subsubsection{Proof of Proposition~\ref{prop:convergence to mu infty}}
\label{subsubsec:proof of convergence measure AG}
We can already prove the almost sure well-definedness of $\mu_\infty$.
\begin{proof}[Proof of Proposition~\ref{prop:convergence to mu infty}]
	We can check the following: for any measurable set $\mathcal C \subset \sfA_t$, the process $s\mapsto M_{t,s}(\mathcal C)=\mu_s(\sfD(t,\mathcal C))$ has the distribution as the proportion of red balls in a time-dependent Pólya urn starting at time $t$ with $\mu_t(\mathcal C)$ red balls and $W_t - \mu_t(\mathcal C)$ black balls and weight process $(W_s)_{s\ge t}$. In particular, it is a martingale so it converges to an almost sure limit. 
	
	Note that this is already enough to get a convergence of the sequence $(\mu_t)_{t\geq t_0}$ because
	\begin{itemize}
		\item for a fixed $t$, using the martingale convergence on $M_{t,s}(\mathcal C)$ for a well-chosen countable number of subsets $\mathcal C$ yields the convergence of the process of measures $(p_t)_*\mu_s$;
		\item from the previous section, the Prokhorov distance between $\mu_s$ and $(p_t)_*\mu_s$ is smaller than $d_\rmH(\sfA_t,\sfA_\infty)$, which tends to $0$ as $t\rightarrow\infty$.
	\end{itemize}
	Hence by taking $t$ large and then $s$ large, we can show that $(\mu_s)_{s\geq t_0}$ is Cauchy for the Prokhorov distance. 
	Since we are working on a complete space, the associated space of measures is also complete and so $(\mu_s)$ converges to some limit $\mu_\infty$.
	
	The statement about the ancestry of a random point $Y_\infty$ taken under $\mu_\infty$ follows from the similar fact expressed for a point taken under $\mu_s$, see Proposition~\ref{prop:height of random point is sum of bernoulli} and the fact that $\mu_\infty$ is the limit of the sequence $(\mu_t)_{t\geq t_0}$. 
\end{proof}

\subsubsection{Quantitative results on the behavior of $(\mu_t)_{t\geq t_0}$}
We want to be a more precise than Proposition~\ref{prop:convergence to mu infty} so as to have some rate of convergence for these measures. 
This will be useful in the next section when we prove the GHP convergence of spaces constructed by an aggregation process.

First we prove a technical result that just follows from the assumption \ref{assum:weight process} on the weight process $(W_t)_{t\geq t_0}$.
\begin{lemma}\label{lem:sum of w/W to the power k}
	Under assumption \ref{assum:weight process} there exists a constant $C$ such that
\begin{align}
	\sum_{s \geq t}  \left(\frac{w_s}{W_s}\right)^{2} \leq \frac{C}{(W_t)^2} \cdot (Z_2)^2 \cdot Z_1 \cdot t^{-1+\delta(t)}.
\end{align}
If we furthermore assume \ref{assum:aghaus:weight process grows linearly} than for any $k\geq 2$
	\begin{align}
		\sum_{s \geq t}  \left(\frac{w_s}{W_s}\right)^{k} \leq t^{3-3k+o(1)}
	\end{align}
\end{lemma}
\begin{proof}[Proof of Lemma~\ref{lem:sum of w/W to the power k}]
	Under assumptions \ref{assum:weight process} we have
\begin{align*}
	W_t \leq Z_1 t \quad \text{and} \quad w_t \leq t^{-2+\delta(t)}.
\end{align*}
Hence for $s\geq t$,
\begin{align*}
	\left(\frac{w_s}{W_s}\right)^{k} \leq \frac{(Z_2\cdot t^{-1+\delta(t)})^{k-1}\cdot w_s}{(W_s)^k}.
\end{align*}
Hence 
	\begin{align*}
	\sum_{t< s \leq 2t}  \left(\frac{w_s}{W_s}\right)^{k} \leq \frac{(Z_2\cdot t^{-1+\delta(t)})^{k-1}}{W_t^2}  \cdot \left(\sum_{t< s \leq 2t}w_s\right) 
	\leq \frac{(Z_2\cdot t^{-1+\delta(t)})^{k-1}}{W_t^2}  \cdot Z_1 \cdot 2t 
\end{align*}
The first result follows by summing the last display for value of $t$ that double each time and taking $k=2$. 
The second assertion just follows from using the lower bound $W_t\geq Z_3 t$.
\end{proof}

We can now prove some quantitative estimates.
\begin{lemma}\label{lem:control limit mass of every block} Under \ref{assum:weight process} and \ref{assum:aghaus:weight process grows linearly}, for any $N\geq 1$ there exists a function $o(1)$ such that for all $t\geq t_0$ such that for all $t\in\mathscr{J}$,
	\begin{align*}
	\sup_{s\geq t}\Ec{M_{t,s}(\sfB_t)^N} \leq t^{-3N+o(1)}.
	\end{align*}
\end{lemma}
\begin{proof}[Proof of Lemma~\ref{lem:control limit mass of every block}]
This proof follows the same idea as the proof of the analogous statement for the discrete time version of this process \cite[Lemma~5]{senizergues_random_2019}. 
The proof is by induction. For $N=1$ then the statement just follows from the fact that we are working with a martingale and that the initial value has the correct order.  
Then, for $N\geq 2$ we use the recursion given in \eqref{eq:recursion on moment for mass of red balls} and apply it to the process $s\rightarrow M_{t,s}(\sfB_t)$ to get 
\begin{align*}
	\sup_{s\ge t} \Ec{(M_{t,s}(\sfB_t))^N} &\leq \Ec{(M_{t,t}(\sfB_t))^N}+ \sum_{\ell =0}^{N-2} \binom{N}{\ell} \left(\sup_{s\geq t} \Ec{((M_{t,s-}(\sfB_t)))^{\ell +1}}\right) \left(\sum_{s \geq t}  \left(\frac{w_s}{W_s}\right)^{N-\ell}\right)\\
	&\leq \left(\frac{w_t}{W_t}\right)^N +\sum_{\ell =0}^{N-2} \binom{N}{\ell} t^{-3(\ell +1) + o(1)} \cdot t^{3-3(N-\ell)+o(1)}\\
	&\leq t^{-3N+o(1)},
\end{align*}
where in the last inequality we used assumption \ref{assum:weight process}, the induction hypothesis and the previous lemma. 
\end{proof}
\subsection{Proof of Proposition~\ref{prop:aggregate is well-defined}}\label{app:subsubsec:proof of prop aggregate is well-defined}
\begin{proof}[Proof of Proposition~\ref{prop:aggregate is well-defined}]
By construction, the space $\sfA_\infty^\bullet$  (considered without measure first) is a random rooted, pointed, length space. 
From Lemma~\ref{lem:aggregate is compact + bound on hausdorff distance to seed}, we get that under the assumptions \ref{assum:weight process} and \ref{assum:blocks}, it is almost surely compact. 
From Proposition~\ref{prop:convergence to mu infty}, it can almost surely be endowed with a probability measure $\mu_\infty$, thus making it a random rooted, pointed, measured compact length space, thus a random variable in $\bL^{\bullet,c}$.
Reasoning conditionally on the random weight process $(W_t)_{t\geq t_0}$, we get from Lemma~\ref{lem:diameter bounds of aggregate}, which is stated in the case of a deterministic weight process, that for any $p\geq 1$,
\begin{align*}
	\Ecsq{d_{\rmH}(\sfA_{\infty},\mathsf{S})^p}{(W_t)_{t\geq t_0}}\leq  \left(C_p\cdot\left(Z_1 Z_2^{3/2} +350 (1+\log(Z_1 Z_2)) (Z_1 Z_2)^{\frac{1}{p}}\right) \right)^p.
\end{align*}
Integrating the inequality above with respect to $(W_t)_{t\geq t_0}$ ensures that if $Z_1$ and $Z_2$ admit sufficiently high moment then $\Ec{d_{\rmH}(\sfA_{\infty},\mathsf{S})^p}<\infty$. This concludes the proof.
\end{proof}
\subsection{Upper-bound on Minkowski dimension}\label{subsec:app:upper-bound on Minkowski dimension}
In this section, we provide an upper-bound on the Minkowki dimension of the metric space $\sfA_\infty$, under some additional assumption \ref{assum:Minkowski coverings} on the laws of the blocks that we glue and the assumption that $\mathsf{S}$ has Minkoswki dimension less than $3$. This is the content of Proposition~\ref{prop:dimension of aggregate is 3}\ref{it:minkowski dimension is less than 3}. 

For any metric space $X$, we denote by $N_\epsilon(X)$ the minimal number of balls of radius $\epsilon$ needed to cover the space $X$. 
The upper-Minkowski dimension of $X$ is obtained as
\begin{align*}
	\overline{\dim}_{\mathrm{Mink}}(X)=\limsup_{\epsilon \rightarrow 0} \frac{\log N_\epsilon(X)}{- \log \epsilon}. 
\end{align*} 
By monotonicity, it is sufficient to take the above limsup along a sequence $(\epsilon_k)$ that is such that $\epsilon_k\rightarrow 0$ and $\frac{\log \epsilon_{k+1}}{\log{\epsilon_k}}\underset{k \rightarrow \infty}{\rightarrow} 1$

\begin{proof}[Proof of Proposition~\ref{prop:dimension of aggregate is 3}\ref{it:minkowski dimension is less than 3}]
	Using the same notation as in the last paragraph of the previous subsection, we want to construct a covering of $\sfA_\infty$ by balls of small radius. 
	
	Fix some $k\geq 1$ and set $\epsilon_k=2^{-\frac{k}{2}}$. 
	Then we consider an $\epsilon_k$-covering of the set $\sfA^{(k)}$ defined in \eqref{eq:definition A^(k)}. 
	The number of balls needed is then bounded above by
	\begin{align*}
		N_{\epsilon_k}(\mathsf{S}) +\sum_{t\in  \mathscr J: w_t\geq 2^{-k}} N_{\epsilon_k}(\sfB_t).
	\end{align*}
	Thanks to assumption \ref{assum:Minkowski coverings}, the expectation of the sum appearing in the last display is then smaller than 
	\begin{align*}
		\sum_{t\in  \mathscr J: w_t\geq 2^{-k}} C \cdot (1 \vee w_t \epsilon_k^{-2}) &\leq C \cdot  \epsilon_k^{-2} \cdot  \left(  \sum_{t\in  \mathscr J: w_t\geq 2^{-k}} w_t \right)  \\
		&\leq C \cdot (2^{-k})^{-\frac{3}{2}+f(k)},
	\end{align*}
	where the function $f(k)$ depends on $Z_1$ and $Z_2$ and tends to $0$ as $k\rightarrow\infty$. 
	Using Markov's inequality we have
	\begin{align*} \Pp{\sum_{t\in  \mathscr J: w_t\geq 2^{-k}} N_{\epsilon_k}(\sfB_t)\geq  (2^{-k})^{-\frac{3}{2}+f(k)-\frac{1}{\sqrt{k}}}} \leq C 2^{-\sqrt{k}},
	\end{align*}
	and the right-hand-side is summable in $k$. Hence using the Borel-Cantelli lemma, with probability one we have $\sum_{t\in  \mathscr J: w_t\geq 2^{-k}} N_{\epsilon_k}(\sfB_t)\leq  (2^{-k})^{-\frac{3}{2}+o(1)}$. 
	Note that because we assumed that $\mathsf{S}$ had upper-Minkowski dimension less than $3$ we have $N_{\epsilon_k}(\mathsf{S})\leq (\epsilon_k)^{3+o(1)}=(2^{-\frac{k}{2}})^{-\frac{3}{2}+o(1)}$, so that
	\begin{align*}
		N_{\epsilon_k}(\sfA^{(k)}) \leq N_{\epsilon_k}(\mathsf{S}) + \sum_{t\in  \mathscr J: w_t\geq 2^{-k}} N_{\epsilon_k}(\sfB_t) \leq  (2^{-\frac{k}{2}})^{-\frac{3}{2}+o(1)}.
	\end{align*}
	
	From there we consider $\epsilon'_k =2^{-\frac{k}{2}}+2\sum_{i\geq k+1}^{\infty} X(i) Y(i)$, where the $X(k),Y(k)$ are those appearing in \eqref{eq:upper bound distances in the aggregate}.
	Thanks to Lemma~\ref{lem:bounds on Xk and Yk}, we almost surely have $2^{-\frac{k}{2}} \leq \epsilon'_k \leq \cst k^2 2^{-k/2}$. 
	Now thanks to \eqref{eq:A^(k) is a Cauchy sequence} we have 
	$d_{\rmH}(\sfA^{(k)},\sfA_\infty) \leq 2\sum_{i\geq k+1}^{\infty} X(i) Y(i)$ so any point of $\sfA_\infty$ is at most at distance $2\sum_{i\geq k+1}^{\infty} X(i) Y(i)$ away from any point of $\sfA^{(k)}$.
	This ensures that considering the collection of balls with the same centres as in the first covering but where the radius is changed from $\epsilon_k$ to $\epsilon'_k$ yields a covering of $\sfA_\infty$.
	
	In the end, on an event of probability $1$ we have 
	\begin{align*}
		\limsup_{k\rightarrow \infty} \frac{\log N_{\epsilon'_k}(\sfA_\infty)}{- \log \epsilon'_k}\leq 3. 
	\end{align*}
	with a sequence $(\epsilon_k')$ that tends to $0$ and such that $\frac{\log \epsilon_{k+1}'}{\log{\epsilon_k'}}\rightarrow 1$. This concludes the proof.
\end{proof}
\subsection{Lower-bound on the Hausdorff dimension.}\label{subsec:app:lower-bound on the Hausdorff dimension}
In this section, we prove a lower-bound of $3$ on the Hausdorff dimension of $\mathsf{A}_\infty$ under the additional assumptions \ref{assum:aghaus:weight process grows linearly}, \ref{assum:aghaus:small jumps don't contribute to weight process} and \ref{assum:aghaus:random point on block is far from root}. 
This is the content of Proposition~\ref{prop:dimension of aggregate is 3}.\ref{it:Hausdorff dimension is more than 3}.
We follow the same idea as in the proof of  \cite[Proposition~12]{senizergues_random_2019}.  
We will show that almost surely, for a point $Y_\infty$ sampled from $\mu_\infty$, we have 
\begin{align*}
	\mu_\infty(\mathrm{Ball}(Y_\infty,r)) \leq r^{3 + o(1)} \quad \text{ as } r\rightarrow 0.
\end{align*}
Using the so-called mass-distribution principle, the last display entails that the Hausdorff dimension of our object is at least $3$. It also ensures that $\mu_\infty$ is diffuse.  

In order to get the last display, we first introduce the following notation 
\begin{align*}
	Y_t:=p_t(Y_\infty).
\end{align*}
We can remark that if $Y_\infty \in \sfD(t)$ for some $t\geq t_0$ then we have 
\begin{align}
	\mathrm{Ball}(Y_\infty, d(Y_\infty,Y_{t-})) \subset \sfD(t),
\end{align} 
where $\sfD(t)$ is the substructure descending from $\sfB_t$. Then, since we have controls on the moments of $\mu_\infty(\sfD(t))= M_{t,\infty}(\sfB_t)$ and an explicit description of $d(Y_\infty,Y_{t-})$ we can get that roughly $d(Y_\infty,Y_{t-})\approx t^{-1}$ and $M_{t,\infty}(\sfB_t) \leq t^{-3}$, which will yield the result. 
\begin{proof}[Proof of Proposition~\ref{prop:dimension of aggregate is 3}.\ref{it:Hausdorff dimension is more than 3}.]
First, for any $\delta$ positive, we consider the set $\mathscr{S}^\delta=\enstq{t\in \mathscr J}{w_t\geq t^{-2-\delta}}$.
For any $t\in \mathscr{S}^\delta$ and any $N\geq 1$ we have $\Ec{M_{t,\infty}(\sfB_t)^N}\leq t^{-3N+o(1)}$, thanks to Lemma~\ref{lem:control limit mass of every block}.
For any $\epsilon>0$ we can use Markov's inequality and get that
$\Pp{M_{t,\infty}(\sfB_t)\geq t^{-3+\epsilon}}\leq t^{-N\epsilon + o(1)}$. 
Taking $N$ large enough, the last display is summable over $t\in \mathscr{S}^\delta$. Doing this reasoning for a sequence of $\epsilon$ tending to $0$ and using the Borel-Cantelli lemma, we get that along $t\in \mathscr{S}^\delta$ we have 
\begin{align*}
\mu_\infty (\sfD(t))= M_{t,\infty}(\sfB_t)\leq  t^{-3 +o(1)}, 
\end{align*}
where the function $o(1)$ in the exponent is random. 

Now, let $\epsilon$ be a positive number and for any $n$ we let $t_n=\exp\left((1+\epsilon)^n\right)$. 
\begin{lemma}\label{lem:local behaviour close to random point}
Under our assumptions, with probability $1$, for all $n$ large enough, there exists $\hat t_n\in \intervallefo{t_{n-1}}{t_{n}}\cap \mathscr S^\delta$ such that $Y_\infty \in \sfD(\hat t_n)$ and $d(Y_\infty, Y_{\hat t_n-})\geq a \cdot t_n^{-1-\delta/2}$.
\end{lemma}
Assuming that the result of the lemma holds true, we have for any $t$ large enough, where we implicitly define $n$ such that $t_{n} \leq t \leq t_{n+1}$,
\begin{align*}
	\mu_{\infty}(\mathrm{Ball}(Y_\infty,a\cdot t^{-1-\delta/2}) )\leq \mu_{\infty}(\mathrm{Ball}(Y_\infty, d(Y_\infty,Y_{\hat t_n-}))) \leq \mu_{\infty}(\sfD(\hat t_n))\leq t^{-\frac{3}{(1+\epsilon)^2}+o(1)}
\end{align*}
This entails using the mass distribution principle that the Hausdorff dimension of $\sfA_\infty$ is at least $\frac{3}{(1+\delta/2)(1+\epsilon)^2}$. Taking $\epsilon\rightarrow 0$ and $\delta\rightarrow 0$ finishes the proof of the lower-bound on the Hausdorff dimension

Note that the last claim of Proposition~\ref{prop:dimension of aggregate is 3}.\ref{it:Hausdorff dimension is more than 3}, that $\mu_\infty$ has full support, follows from the fact that the seed and all the blocks carry a measure that has full support and the property of the weight martingale \eqref{eq:limiting value of pemantle urn is as positive} that ensures that any subset with positive $\mu_t$ mass also has some positive $\mu_\infty$ mass.
\end{proof}
\begin{proof}[Proof of Lemma~\ref{lem:local behaviour close to random point}]
	Let $n\geq 1$. 
	Remark that for any $t\in \intervallefo{t_{n-1}}{t_{n}}\cap \mathscr S^\delta$, 
	the probability that $Y_\infty\in \sfD(t)$ is $\frac{w_t}{W_t}$ and
	conditionally on $Y_\infty\in \sfD(t)$ we have $d(Y_\infty, Y_{t-})\geq d(Y_t,Y_{t-})$
	and the latter has (conditional) distribution that has the same law as $d(\rho,U(t,w_t))$.
Recalling assumption \ref{assum:aghaus:random point on block is far from root}, we have $\Pp{d(\rho,U(t,w_t))\geq a\sqrt{w_t}}\geq b$ and since $t$ is assumed to be in $\intervallefo{t_{n-1}}{t_{n}}\cap \mathscr S^\delta$ we have $w_t\geq t^{-2-\delta}\geq t_{n}^{-2-\delta}$. 
From that we get that
\begin{align*}
	\Pp{Y_\infty \in \sfD(t) \text{ and }d(Y_\infty, Y_{t})\geq a \cdot t_n^{-1-\delta/2}} \geq b \cdot \frac{w_t}{W_t}.
\end{align*}	
and those events are independent for different values of $t$.
	Hence
	\begin{align*}
		&\Pp{\nexists t\in \intervallefo{t_n}{t_{n+1}}\cap \mathscr S^\delta, \ Y_\infty \in \sfD(t) \text{ and }d(Y_\infty, Y_{t})\geq a \cdot t_n^{-1-\delta/2}} \\
		&\leq \prod_{t\in \intervallefo{t_n}{t_{n+1}}\cap \mathscr S^\delta} \left(1- b \frac{w_t}{W_t}\right)\\
		&\leq \exp\left(-b \cdot \sum_{t\in \intervallefo{t_n}{t_{n+1}}\cap \mathscr S^\delta}\frac{w_t}{W_t} \right)\\
		&\leq \exp\left(-b \cdot C_{\delta,\epsilon} \cdot \epsilon(1+\epsilon)^n\right),
	\end{align*}
	where in the last line we used the lemma stated below.
	Since the last display is summable in $n$, we can use the Borel-Cantelli lemma to conclude. 
\end{proof}

\begin{lemma}\label{lem:sum of w/W along a slice}
	Under the assumptions, for all $\epsilon,\delta>0$ there exists a constant $C_{\delta,\epsilon}$ such that for all $n\geq 1$ 
	\begin{align*}
	 \sum_{t\in \intervallefo{t_n}{t_{n+1}}\cap \mathscr S^\delta}\frac{w_t}{W_t} \geq C_{\delta,\epsilon} \epsilon (1+\epsilon)^n 
	\end{align*}
\end{lemma}
\begin{proof}[Proof of Lemma~\ref{lem:sum of w/W along a slice}]
	We work under the assumptions \ref{assum:aghaus:weight process grows linearly} and \ref{assum:aghaus:small jumps don't contribute to weight process} for some $\delta>0$.
	This entails that $W_t=Z_3 t + o(t)$ and $\sum_{s\leq t, s\in \mathscr S^\delta}w_s = W_t + o(W_t)=Z_3 t + o(t) $. 
	First, we split the sum into slices as follows
	\begin{align*}
		\sum_{t\in \intervallefo{t_n}{t_{n+1}}\cap \mathscr S^\delta}\frac{w_t}{W_t}\geq \sum_{i=0}^{I_n} \sum_{t\in \intervallefo{2^i t_n}{2^{i+1} t_n}\cap \mathscr S^\delta}\frac{w_t}{W_t},
	\end{align*}
	where $I_n$ is defined as $I_n:=\max\enstq{i\geq 0}{2^it_n \leq t_{n+1}}$. 
	Now, for a given $i$, we have 
	\begin{align*}
		\sum_{t\in \intervallefo{2^i t_n}{2^{i+1} t_n}\cap \mathscr S^\delta}\frac{w_t}{W_t} &\geq \frac{1}{W_{2^{i+1} t_n}} \cdot  \sum_{t\in \intervallefo{2^i t_n}{2^{i+1} t_n}\cap \mathscr S^\delta} w_t\\
		&\geq \frac{1}{Z_3 2^{i+1} t_n + o(2^{i+1} t_n)} \cdot  (Z_3 \cdot 2^{i+1} t_n + o(2^{i} t_n) - (Z_3 \cdot 2^{i} t_n + o(2^{i+1} t_n)))\\
		&\geq \frac{1}{2}+o(1),
	\end{align*}
	so that $\sum_{t\in \intervallefo{t_n}{t_{n+1}}\cap \mathscr S^\delta}\frac{w_t}{W_t}\geq I_n(\frac{1}{2}+o(1))$. 
	Now, by definition of $I_n$ and of $t_n$ we have
	\begin{align*}
		I_n\sim \frac{\log(t_{n+1})-\log(t_n)}{\log 2}\sim \frac{\epsilon (1+\epsilon)^n}{\log 2}
	\end{align*}  as $n\rightarrow \infty$. 
	This completes the proof. 
\end{proof}
\subsection{Convergence of tree aggregation processes}\label{subsec:app:convergence of aggregation processes}
Finally in this section we prove Proposition~\ref{prop:convergence of aggregation processes}, which ensures that under the appropriate conditions, we have the convergence
  \begin{align*}
	\mathrm{Aggreg}((W_t^r)_{t\geq t_0},\mathsf S^{\bullet,r},\eta^r) \underset{r\rightarrow0}{\rightarrow} \mathrm{Aggreg}((W_t)_{t\geq t_0},\mathsf S^\bullet,\eta),
\end{align*}
in distribution in $\bL$. 
\begin{proof}[Proof of Proposition~\ref{prop:convergence of aggregation processes}]
%

Recall the setting introduced in Section~\ref{subsubsec:convergence of aggregation process}.
We write, for every $r$,
\begin{align*}
	 \mathsf{A}^{r,\bullet}_\infty :=\mathrm{Aggreg}((W_t^r)_{t\geq t_0},\mathsf S^{\bullet,r},\eta^r)
\end{align*}
 and $\mathsf{A}^{\bullet}_\infty :=\mathrm{Aggreg}((W_t)_{t\geq t_0},\mathsf S^\bullet,\eta)$ for the limit candidate. 
We will keep the same notation for both constructions, simply adding a superscript $r$ for the one that depends on the parameter $r$. For the whole proof below, we drop the $\bullet$ superscript for readability.

Let $k\geq 1$ and $t\geq t_0$.
We denote by $\sfA^{(k)}_t$ the subset of $\sfA_\infty$ obtained by only keeping the seed and the blocks corresponding to jump times $s\in \mathscr{J}\cup \intervalleof{t_0}{t}$ and whose weight $w_s$ are larger than $2^{-k}$. 
For simplicity, we assume that with probability $1$, no jump in $\mathscr{J}$ has size exactly $2^{-k}$. (Otherwise we could just replace $2^{-k}$ with another similar threshold that has that property). 
\begin{align*}
	\sfA^{(k)}_t=\mathsf{S}\cup \underset{w_s > 2^{-k}}{\bigcup_{s\in\mathscr J\cap \intervalleof{t_0}{t}}} \sfB_s.
\end{align*}
We consider the finite measure $\mu_t^{(k)}:=\mathbf{1}_{\sfA^{(k)}_t} \cdot \mu_t$ on $\sfA^{(k)}_t$ to see it as a rooted, measured, compact metric space. We define $\sfA^{(k),r}_t$ and $\mu_t^{(k),r}$ similarly for $\sfA_\infty^r$. 

First, let us show that for any fixed $k$ and $t$ then 
\begin{align}\label{eq:convergence Akzt}
	\sfA^{(k),r}_t \rightarrow \sfA^{(k)}_t
\end{align}
in distribution for the GHP topology.

By assumption \ref{assum:convergence weight process}, and using the Skhorokhod embedding theorem, we can couple $(W_s^r)_{s\geq t_0}$ and $(W_s)_{s\geq t_0}$ in such a way that $(W_s^r)_{s\in \intervalleff{t_0}{t}} \rightarrow(W_s)_{s\in \intervalleff{t_0}{t}}$ for the Skorokhod topology.
This entails that the position and size of the jumps of $(W_s^r)_{s\in \intervalleff{t_0}{t}}$ converge to that of $(W_s)_{s\in \intervalleff{t_0}{t}}$ on that interval. 

We enumerate the jump times $s_0,s_1,\dots$ of the process $(W_t)$ on the interval $\intervalleff{t_0}{T}$ say in non-increasing order of the size of the associated jump. 
The above convergence ensures that we can find times $s_0^r,s_1^r,\dots$ to that, as $r\rightarrow0$ we have $s_k^r\rightarrow s_k$ and $w_{s_k^r}^r \rightarrow w_{s_k}$ and $W_{s_k^r}^r \rightarrow W_{s_k}$. (Note that for a given $r$ and $k$, it is possible that $s_k^r$ is not a jump time of the process).
Now we couple the constructions:
\begin{itemize}
	\item We couple the seeds $\mathsf{S}^\bullet$ and  $\mathsf{S}^{r,\bullet}$ in such a way that their GHP distance is small (and tends to $0$ as $r \rightarrow \infty$), as given by assumption \ref{assum:convergence seed}.
	\item For any $k\geq 1$ for which $s_k^r\in \mathscr J^r$, we couple the blocks $\mathsf{B}_{s_k}$ and $\mathsf{B}_{s_k^r}^r$ which have respective distribution $\eta(s_k,w_{s_k})$ and $\eta^r(s_k^r,w_{s_k}^r)$ in such a way that their GHP distance is small. 
	This amounts to saying that we can isometrically embed $\mathsf{B}_{s_k}$ and $\mathsf{B}_{s_k^r}^r$ in the same metric space so that the Hausdorff distance of their image and Prokhorov distance of their push-forward measure is small. 
	This is possible thanks to assumption \ref{assum:convergence block distribution}. 
	\item We couple the uniform random variables $(U_t, \ t\in \mathscr J)$ and $(U_t^r, \ t\in \mathscr J^r)$ in such a way that $U_{s_k}=U_{s_k^r}$, whenever $s_k^r\in \mathscr J^r$. The rest of those can be chosen in an i.i.d.\ fashion.
	\item From all the above, for any $k\geq 1$ we almost surely have $K_{s_k}=K_{s_k^r}^r$ when $r$ is large enough. 
	In that case, we couple the random point $X_{s_k}$ with $X_{s_k^r}^r$ on respectively $\mathsf{B}_{s_k}$ and $\mathsf{B}_{s_k^r}^r$, in such a way that $X_{s_k}$ and $X_{s_k^r}^r$ are close together. 
	This is possible because they are embedded in such a way that the Prokhorov distance between their respective measure is small. 
\end{itemize}
Now, since $\sfA^{(k)}_t$ and $\sfA^{(k),r}_t$ are determined by a finite (though random) number of gluing events, they can be shown to be close to one another in GHP topology when $r$ is large enough. 
More precisely, $\sfA^{(k)}_t$ only depends on the blocks and gluing events corresponding to all the jumps in 
\begin{align*}
	\enstq{u \in \mathscr{J} \cap \intervalleoo{t_0}{t}}{u \preceq s \text{ for some } s\in\mathscr{J} \cap \intervalleoo{t_0}{t} \text{ with } w_s >2^{-k} },
\end{align*}
and similarly for $\sfA^{(k),r}_t$. 
Then the coupling of all those objects and quantities presented above ensures that the convergence \eqref{eq:convergence Akzt} holds, (see \cite[Lemma~2.4]{senizergues_growing_2022} for a similar statement only for the GH topology, the addition of measures is straightforward).

Now, in order to prove Proposition~\ref{prop:convergence of aggregation processes}, we just need to show that for any $\epsilon>0$, we have 
\begin{align*}
	\limsup_{t\rightarrow \infty} \limsup_{k\rightarrow \infty }\Pp{d_{\mathrm{GHP}}(\sfA^{(k),r}_t, \sfA^{r}_\infty) \geq \epsilon } \underset{r\rightarrow0}{\rightarrow} 0,
\end{align*} 
as well as a similar statement $\limsup_{t\rightarrow \infty} \limsup_{k\rightarrow \infty }\Pp{d_{\mathrm{GHP}}(\sfA^{(k),r}_t, \sfA^{r}_\infty) \geq \epsilon } = 0$ for the limit. 
Indeed, for any $\epsilon>0$, it would suffice to first take $t$ and $k$ large enough so that 
 $\Pp{d_{\mathrm{GHP}}(\sfA^{(k),r}_t, \sfA^{r}_\infty) \geq \frac{\epsilon}{3}}$ and $\Pp{d_{\mathrm{GHP}}(\sfA^{(k)}_t, \sfA_\infty) \geq \frac{\epsilon}{3} }$ are smaller than $\frac{\epsilon}{3}$ 
 and then take $r$ small enough so that $d_{\mathrm{GHP}}(\sfA^{(k),r}_t \sfA^{(k)}_t)\leq \frac{\epsilon}{3}$ with probability larger than $1-\frac{\epsilon}{3}$. 
This would ensure that on the coupling that we determined above, we have   $\Pp{d_{\mathrm{GHP}}(\sfA_\infty, \sfA^{r}_\infty) \geq \epsilon} \leq \epsilon$, so that $\sfA^{r}_\infty \rightarrow \sfA_\infty$ in probability.

Now, we fix some $\epsilon>0$ and we aim at showing that  we can find $k$ and $t$ such that for $r$ small enough we can couple $\sfA_\infty$ and $\sfA_\infty^{r}$ in such a way that their GHP distance is smaller than $\epsilon$ with probability at least $1-\epsilon$. 
This will show that $\sfA_\infty^{r}$ converges in distribution as $r\rightarrow0$ towards $\sfA_\infty$.

First, 
\begin{align*}
	d_{\rmH}(\sfA^{(k)}_t,\sfA_\infty) \underset{t\rightarrow \infty }{\rightarrow} d_{\rmH}(\sfA^{(k)},\sfA_\infty) \underset{k\rightarrow\infty }{\rightarrow} 0,
\end{align*}
thanks to \eqref{eq:A^(k) is a Cauchy sequence} so this term is small with arbitrary high probability, by taking $t$ and $k$ large enough. 
The same is true for $d_{\rmH}(\sfA^{r,(k)}_t,\sfA_\infty^r)$ using \ref{assum:tight control on weight process}. 
Then 
\begin{align*}
	d_{\rmP}(\mu_t^{(k)},\mu_\infty) \leq d_{\rmP}(\mu_t^{(k)},\mu_t) + d_{\rmP}(\mu_t,\mu_\infty) \leq (1- \mu_t(\sfA_t^{(k)})) + d_{\rmP}(\mu_t,\mu_\infty),
\end{align*}
where the second term can be made small by taking a large $t$ and the first one can be made small by taking a large $k$. 
Last 
\begin{align*}
d_{\rmP}(\mu_t^{(k),r},\mu_\infty^r) \leq d_{\rmP}(\mu_t^{(k),r},\mu_t^r) + d_{\rmP}(\mu_t^r,\mu_\infty^r) \leq (1- \mu_t^r(\sfA_t^{(k),r})) + d_{\rmP}(\mu_t^r,\mu_\infty^r).
\end{align*}
The first term is small if we take $t$ and $k$ large and then $r$ large. 

We just have to take care of the last term $d_{\rmP}(\mu_t^r,\mu_\infty^r)$.
For that, suppose that we have constructed $\sfA^{r,(k)}_t$ and $\sfA^{(k)}_t$ on the same probability space and that we are on the event where $d_{\mathrm{GHP}}(\sfA^{r,(k)}_t,\sfA^{(k)}_t)\leq \epsilon$. This entails that we can isometrically embed $\sfA^{r,(k)}_t$ and $\sfA^{(k)}_t$ in the same metric space with isometries $\phi$ and $\phi'$ so that $d_{\rmH}(\phi(\sfA^{r,(k)}_t), \phi'(\sfA^{(k)}_t))\leq \epsilon$ and $d_{\rmP}(\phi_*\mu_t^{(k),r} , \phi'_*\mu_t^{(k)})\leq \epsilon$.
Now since $\sfA^{(k)}_t$ is compact, it can be covered with a finite number $N(\epsilon)$ of balls of radius $\epsilon$. If we abuse notation and consider $\sfA^{r,(k)}_t$ and $\sfA^{(k)}_t$ as embedded in the same metric space as above, then using the same centres but taking the radii to be $2\epsilon$ instead of $\epsilon$ allows us to cover $\sfA^{r,(k)}_t$ as well. 

This gives us a covering of $\sfA^{r,(k)}_t$ of a certain number of pieces $\mathcal C_1^r, \mathcal C_2^r \dots \mathcal C_N^r$ which all have diameter smaller than $2\epsilon$, and where $N=N(\epsilon)$.
From there, we can show using estimates on the mass martingales that the $\mu_\infty^r$-mass of any of the $\sfD(t,\mathcal C^r)$  will not substantially deviate from their $\mu_t$-mass and this will prove that indeed those two measures are close in Prokhorov sense.

Then, for any $1\leq i \leq N(\epsilon)$ we know that we have
\begin{align*}
	\mu_\infty^r(\sfD(t,\mathcal C_i^r))= M_{t,\infty}(\mathcal C_i^r) \quad \text{ and }\quad \mu_t^r(\sfD(t,\mathcal C_i^r))= M_{t,t}(\mathcal C_i^r)= \mu_t(\mathcal C_i^r).
\end{align*}
The variance of $\mu_\infty^r(\sfD(t,\mathcal C_i^r))$ is, thanks to \eqref{eq:bound on variance of limiting mass} and Lemma~\ref{lem:sum of w/W to the power k}, bounded above by 
\begin{align*}
	\sum_{s >t}\left(\frac{w_s}{W_s}\right)^2 \leq \frac{1}{(W_t^r)^2}\cdot Z_1^r \cdot  (Z_2^r)^2 \cdot t^{-1+\delta(t)}.
\end{align*}
Note that using Cauchy-Schwarz inequality we have
\begin{align*}
\left(\sum_{i=1}^{N} |M_{t,\infty}(\mathcal C_i^r)- \Ec{M_{t,\infty}(\mathcal C_i^r)}|\right)^2\leq N(\epsilon)\cdot \sum_{i=1}^{N} |M_{t,\infty}(\mathcal C_i^r)- \Ec{M_{t,\infty}(\mathcal C_i^r)}|^2.
\end{align*}
Now using Markov's inequality we get 
\begin{align*}
	\Pp{\sum_{i=1}^N |M_{t,\infty}(\mathcal C_i^r)- \Ec{M_{t,\infty}(\mathcal C_i^r)}| \geq \epsilon}
	&\leq \frac{1}{\epsilon^2}\cdot N \cdot \sum_{i=1}^N\Var{M_{t,\infty}(\mathcal C_i^r)}\\
	&\leq N \cdot \epsilon^{-2}\cdot N \cdot \sum_{s\geq t}\left(\frac{w_s}{W_s}\right)^2\\
	& \leq \epsilon^{-2} N^2 \cdot \frac{1}{(W_t^r)^2}\cdot Z_1^r \cdot  (Z_2^r)^2 \cdot  t^{-1+\delta(t)}.
\end{align*} 
and the last display tends to $0$ in probability as $t\rightarrow \infty$, uniformly in $r$.
This ensures that with high probability as $t\rightarrow\infty$ we have
\begin{align*}
	\sum_{i=1}^N |\mu_\infty^r(\sfD(t,\mathcal C_i^r))- \mu_t^r(\sfD(t,\mathcal C_i^r))| \leq \epsilon.
\end{align*}
This ensures that we can couple two random variables $X$ and $Y$ with respective distribution $\mu_t^r$ and $\mu_\infty^r$ in such a way that with probability at least $1-2\epsilon$, the points $X$ and $Y$ fall into the same $\sfD(t,\mathcal C_i^r)$.
Since 
\begin{align*}
	\sup_{1\leq i\leq N}\diam(\sfD(t,\mathcal C_i^r))\leq 2\epsilon+2d_{\rmH}(\sfA_t^r,\sfA_\infty^r),
\end{align*}
this proves that the Prokhorov distance between $\mu_t^r$ and $\mu_\infty^r$ is smaller than $3\epsilon$ with high probability as $t$ is large.
This completes the proof.
\end{proof}

\end{document}